\documentclass[11pt,letterpaper]{amsart}
\usepackage{amsmath,bbm, color}
\usepackage{hyperref}
\usepackage{bbm}

\usepackage{url}
\usepackage[english]{babel}
\usepackage{paralist}
\usepackage{amsthm,amsfonts,amssymb,amsmath}
\usepackage{comment}

\usepackage[latin1]{inputenc}

\newcommand{\new}[1]{\textcolor{black}{#1}}
\newcommand{\tRe}{\textup{Re}}
\newcommand{\tIm}{\textup{Im}}

\usepackage{latexsym,mathrsfs,bm}
\newcommand{\vb}{\boldsymbol v}
\newcommand{\lambb}{\boldsymbol \lambda}

\newcommand{\sumh}{\sideset{}{^h}\sum}
\newcommand{\sumd}{\sideset{}{^d}\sum}
\newcommand{\sumsharp}{\sideset{}{^\#}\sum}

\newcommand{\R}{\mathbb{R}}

\newcommand{\N}{\mathcal{N}}

\newcommand{\es}[1]{\begin{equation}\begin{split}#1\end{split}\end{equation}}
\newcommand{\est}[1]{\begin{equation*}\begin{split}#1\end{split}\end{equation*}}

\newcommand{\M}{\mathcal{M}}
\newcommand{\cH}{\mathcal{H}}
\newcommand{\LL}{\mathcal{L}}

\renewcommand{\mod}[1]{~\pr{\textnormal{mod}~#1}}

\newtheorem{thm}{Theorem}[section]
\newtheorem*{thm*}{Theorem}

\newtheorem{prop}[thm]{Proposition}
\newtheorem{lem}[thm]{Lemma}
\newtheorem{lemma}[thm]{Lemma}
\newtheorem{cor}[thm]{Corollary}
\newtheorem{defi}{Definition}

\theoremstyle{remark}
\newtheorem{rem}{Remark}

\newtheorem{rem*}{Remark}
\newcommand{\pr}[1]{\left( #1\right)}

\newcommand{\e}[1]{\operatorname{e}\pr{ #1}}

\newcommand{\bfrac}[2]{\left(\frac{#1}{#2}\right)}

\newcommand{\m}{\mathfrak m}

\renewcommand{\c}{\mathfrak{c}}

\newcommand{\lt}{\left}
\newcommand{\rt}{\right}

\newcommand{\bsm}{\lt(\begin{smallmatrix}}
\newcommand{\esm}{\end{smallmatrix}\rt)}

\newcommand{\sym}{\textup{sym}}

\newcommand{\cS}{\mathcal{S}}

\newcommand{\Dis}{\textup{Dis}}
\newcommand{\Ctn}{\textup{Ctn}}
\newcommand{\Hol}{\textup{Hol}}
\newcommand{\non}{\textup{non}}
\newcommand{\sing}{\textup{s13}}

\def\sumh{\operatornamewithlimits{\sum\nolimits^h}}

\def\sumprime{\operatornamewithlimits{\sum\nolimits^\prime}}

\newcommand{\summany}{\operatorname*{\sum... \sum}}

\let\originalleft\left
\let\originalright\right
\renewcommand{\left}{\mathopen{}\mathclose\bgroup\originalleft}
\renewcommand{\right}{\aftergroup\egroup\originalright}

\numberwithin{equation}{section}

\newcommand{\W}{\mathcal W}
\newcommand{\Rep}{\textrm{Re}}

\newcommand{\mL}{\mathcal L}

\newcommand{\elltwo}{\ell_{\infty}}

\begin{document}
\title[The $n^{th}$ centered moments of a large orthogonal family]{The $n^{th}$ centered moments of a large orthogonal family of automorphic $L$-functions}

\date{
\today}

\author[V. Chandee]{Vorrapan Chandee}
\address{Mathematics Department \\ Kansas State University \\ Manhattan, KS 66503}
\email{chandee@ksu.edu}

\author[Y. Lee]{Yoonbok Lee$^{\dagger}$}
\thanks{$^{\dagger}$corresponding author.}
\address{Department of Mathematics \\ Incheon National University \\ Incheon 22012, Korea}
\email{leeyb@inu.ac.kr, leeyb131@gmail.com}

\author[X. Li]{Xiannan Li}
\address{Mathematics Department \\ Kansas State University \\ Manhattan, KS 66503}
\email{xiannan@ksu.edu }

\subjclass[2020]{11M50, 11F11, 11F72 }
\keywords{$n$-th centered moments, Low-lying zeros, Orthogonal family, $GL(2)$ $L$-functions.}

\allowdisplaybreaks
\numberwithin{equation}{section}
\begin{abstract}
We obtain the $n$th centered moments of one level densities of a large orthogonal family of $L$-functions associated with holomorphic Hecke newforms of level $q$, averaged over $q\sim Q$.  We verify the Katz-Sarnak conjecture for these statistics, in the range where the sum of the supports of the Fourier transforms of test functions lies in $(-4, 4)$.  In so doing, we need to understand certain phantom oversized terms, which allow us to extract the right off-diagonal contributions.  We further need to resolve the combinatorial problem that arises when matching our main terms with random matrix predictions.

\end{abstract}

\maketitle

\section{Introduction} 

It is a fundamental heuristic that the statistical behavior of zeros of families of $L$-functions mirrors the corresponding statistics of eigenvalues of classical compact groups of random matrices. The first indication of this starts with Montgomery's pair correlation conjecture \cite{Mont} and his conversation with Dyson. Later, Katz and Sarnak \cite{KaSa} established that, for various families of zeta and $L$-functions over function fields, the distribution of low-lying zeros near the central point coincides with that of eigenvalues near 1 in the scaling limit of classical compact groups such as the unitary, symplectic, or orthogonal groups, depending on the symmetry type of the family. They further conjectured that this correspondence extends to families of $L$-functions over number fields, giving rise to a heuristic framework for predicting zero statistics and symmetry types. 

To be more specific, we define the one-level density of zeros as
$$ \mathcal{OL}(\Phi, C) =   \frac{1}{| \mathcal H(C)|}  \sum_{f \in \mathcal H(C)} \ \sum_{j}  \Phi(\mathcal U \gamma_{j, f}  )   ,$$
where $\mathcal H(C)$ is an appropriate family of automorphic forms to which we have associated $L$-functions with analytic conductor around size $C$, $\mathcal U = \frac{\log C}{2\pi}$, and $\Phi$ is a Schwartz class function. In the above, we have written the nontrivial zeros of $L$-functions associated to $f \in \mathcal H({C})$ as $\frac 12 + i\gamma_{j, f}$, where $\gamma_{j, f}$ is real under the Generalized Riemann Hypothesis (GRH). The one-level density conjecture states
$$ \displaystyle \mathcal {OL}(\Phi, C) = \int \Phi(x) W_G(x) \> dx +  \new{ o(1) }, $$
where $W_G(x)$ is a density function depending only on some underlying symmetry group. For example, $W_G(x) = 1$ for unitary group, and $W_G(x) = 1 + \frac{\delta_0}{2}$ for the orthogonal group, where $\delta_0$ is the usual Dirac delta distribution. Evidence for this conjecture appears in various families of $L$-functions but with {\it restricted support} on $\widehat \Phi$, the Fourier transform of $\Phi$. For example, Iwaniec, Luo and Sarnak \cite{ILS} studied the one-level density for the family associated with cuspidal new forms of fixed weight $k$ and squarefree level $q$.  Under GRH, they showed that the conjecture holds as long as $\widehat \Phi$ is supported in $(-2, 2)$ as $q \rightarrow \infty$.  Hughes and Rudnick \cite{HR} studied the one-level density with the family of Dirichlet $L$-functions of non-trivial characters mod $q$ for a fixed odd prime $q$, which is associated with unitary group.  They proved the conjecture when $\widehat \Phi$ is supported on $[-2, 2].$ In \cite{BCL}, Baluyot and the first and third authors develop a new approach that yields a stronger result for a larger family of $L$-functions. In particular, they consider the orthogonal family of $L$-functions attached to holomorphic Hecke newforms of level $q$, averaged over all $q \asymp Q$. Assuming GRH, they showed that the one-level density for this extended family matches the Katz-Sarnak prediction when the support of the Fourier transform of the test function is contained in the interval $(-4,4)$, the widest support in the literature.  The family studied in this work is amenable to such an extension; in contrast, the best known analogous result for a large family of Dirichlet $L$-functions due to Drappeau, Pratt and Radziwi\l\l \cite{DPR} has the support restricted to $(-2 - 50/1093, 2 + 50/1093)$.

The bandwidth restriction on the support of $\widehat {\Phi}$ is not merely a  technical condition.  The uncertainty principle from harmonic analysis tells us that if we want to isolate the contribution of low zeros by choosing $\Phi$ with narrow support, the wider the support needs to be for $\widehat{\Phi}$.  This is highly desirable since it is arithmetically significant whether $L(s, f)$ vanishes at $s = 1/2$ in many examples.  More generally, in such examples, the order of vanishing of $L(s, f)$ at $s= 1/2$ contains important arithmetic information.  In order to extract such refined information about the low-lying zeros, one can consider not only extending the support of the test function, but also studying higher moments of the sum over zeros.  To describe our results, we now fix some notation.

Let $S_k(q)$ be the space of cusp forms of fixed even weight $k \ge 4$ for the group $\Gamma_0(q)$ with trivial nebentypus,  where  
$$
\Gamma_0(q) := \left\{ \left( \left.{\begin{array}{cc}
   a & b \\
   c & d \\
  \end{array} } \right)  \ \right|  \  ad- bc = 1 , \ \  c \equiv 0 \mod q \right\}.
$$
Let $\mathcal H_k(q) $ be an orthogonal basis of the space of newforms in $ S_k(q)$ consisting of Hecke cusp newforms, normalized so that the first Fourier coefficient is $1$. 
For convenience, we normalize our sums over $f$ to play well with spectral theory. To be more specific, we define the harmonic average of $\alpha_f \in \mathbb{C} $ over $\mathcal H_k(q)$ to be 
\begin{equation}\label{eqn: harmonicweightsummation}
\sumh_{f \in \mathcal H_k(q)} \alpha_f = \frac{\Gamma(k-1)}{(4\pi)^{k-1}}\sum_{f \in \mathcal H_k(q)} \frac{\alpha_f}{\|f\|^2},
\end{equation}
where  $ \displaystyle \|f \|^2 = \int_{\Gamma_0(q) \backslash \mathbb H} |f(z)|^2 y^{k-2} \> dx \> dy    $ and $\mathbb H$ is the upper half plane.

For each $f \in \mathcal H_k(q) $,   the $L$-function associated to $f$ is defined  by
\es{ \label{def:Lsf}
L(s, f) = \sum_{n \geq 1} \frac {\lambda_f(n)}{n^{s}} &= \prod_p \pr{1 - \frac{\lambda_f(p)}{p^s} + \frac{\chi_0(p)}{p^{2s}}}^{-1} \\
&= \prod_p \pr{1 - \frac{\alpha_f(p)}{p^s} }^{-1} \pr{1 - \frac{\beta_f(p)}{p^s} }^{-1}}
for $\Rep(s) > 1 $, where the $\lambda_f (n) $ are the Hecke eigenvalues of $f$ and $\chi_0$ denotes the trivial Dirichlet character modulo $q$. Since $f$ is a newform, $L(s, f)$ is entire and satisfies the functional equation
 $$\Lambda\pr{\tfrac 12 + s, f} = \epsilon_f \Lambda\pr{ \tfrac 12 - s, f},$$  
where the completed $L$-function $\Lambda(s, f)$ is defined by
\est{\Lambda\pr{ \tfrac 12 + s, f} := \pr{\frac q{4\pi^2}}^{\frac s2 +   \frac14  } \Gamma \pr{s + \frac k2} L\pr{ \tfrac 12 + s, f},}
and $\epsilon_f = \pm 1 $ is the sign of the functional equation.  When $\epsilon_f = 1$, we say that $f$ is even.  Otherwise, we say $f$ is odd.  Note that the functional equation implies that $L(\frac12 , f ) = 0 $ for all odd $f$.

Assume GRH for $L(s,f)$. We list the nontrivial zeros $\frac12 + i \gamma_{j, f}$ of $L(s,f)$ as 
$$ \cdots  \leq \gamma_{-3, f} \leq \gamma_{-2, f} \leq \gamma_{-1,f} \leq 0 \leq \gamma_{1, f} \leq \gamma_{2, f} \leq \gamma_{3, f} \leq \cdots $$
for an even form $f$ and 
$$ \cdots \leq \gamma_{-3, f} \leq \gamma_{-2, f} \leq \gamma_{-1,f} \leq \gamma_{0, f} = 0 \leq \gamma_{1, f} \leq \gamma_{2, f} \leq \gamma_{3, f} \leq \cdots$$
for an odd form $f$. 
By the functional equation we see that $ \gamma_{ -j , f } = - \gamma_{j,f} $.    
Let $ \Psi(x)$ be a smooth function, compactly supported in $(a,b)$ for fixed $ 0<a<b$ and let $ \Phi_i (x) $ be an even Schwartz class function for $ i \leq n $. 
Then the $n$-th centered moment  for $\mathcal H_k (q) $ is defined by
\es{\label{def:nthcenterNT}
\mathscr{L}_n(Q) := \frac{1}{N_0 (Q)}\sum_q \Psi\bfrac{q}{Q} \sumh_{f \in \mathcal H_{k}(q)} \prod_{i = 1}^n \left[ \sum_{j} \Phi_{i} \left( \frac{\gamma_{j,f} }{2 \pi} \log Q\right) - \widehat{\Phi}_i(0) - \frac{\Phi_i (0)}{2} \right],
}
where
\begin{equation*}
N_0 (Q) := \sum_q \Psi\bfrac{q}{Q} \sumh_{f \in \mathcal H_{k}(q)} 1 \sim   c \widetilde{ \Psi } (1) Q  
\end{equation*}
for some constant $ c>0$ by Lemma \ref{lem:asympforNQ},  
$$ \widetilde \Psi (s) := \int_0^\infty \Psi(x)x^{s-1}dx  $$
is the Mellin transform of $\Psi(x)$, and  
$$ \widehat{\Phi}(t) := \int_{-\infty}^{\infty} \Phi(x) e^{- 2 \pi i tx } dx      $$
is the Fourier transform of $\Phi(x)$.
This is analogous to the $n^{th}$ centered moments appearing in Hughes-Miller's work \cite{HM}.  The study of such $n^{th}$ moments is motivated by applications towards high order non-vanishing results at the critical point, and in particular towards proving that a high percentage of $L$-functions do not vanish to high order.  \footnote{Indeed, if many $L$-functions vanish to high order, then the quantity in \eqref{def:nthcenterNT} must be very large, for appropriate choices of test functions $\Phi_i$.}

In the aforementioned work of Baluyot, Chandee and Li \cite{BCL}, a one-level density result corresponding to $n=1$ was derived, with $\widehat \Phi$ compactly supported in $(-4, 4).$  To be more precise, assuming GRH, their result \cite{BCL} shows that for $\Phi_1$ an even Schwartz function with $\widehat \Phi_1$ compactly supported in $(-4, 4),$  
$$ \lim_{Q \rightarrow \infty}  \mathscr{L}_1 (Q) =0.
$$

In this paper, we are interested in studying the more complex quantity $\mathscr{L}_n (Q) $ for general $n\ge 1$.  To be more precise, let $O(N)$ denote the group of $N\times N$ orthogonal matrices.  Further let $SO(  N) $ be the subgroup of $O(N)$ with determinant $1$ and $ O^{-} ( N)$ be the coset of $O( N)$ with determinant $ -1$, so that $O(N) $ is the disjoint union of $SO(N) $ and $ O^{-}(N)$.  
 If $ e^{i\theta}$ is an eigenvalue of an orthogonal matrix, then so is $ e^{-i \theta}$. Thus, we may write the eigenvalues of $X_N \in SO(2N)$ as $   e^{ \pm  i \theta_1} , \ldots ,   e^{ \pm  i \theta_N}$, and the eigenvalues of $ X_{N} \in O^{-} (2N+2) $ are $  \pm 1 = \pm e^{ i \theta_0}    ,   e^{ \pm  i \theta_1} , \ldots ,  e^{ \pm i \theta_{N} }$ with $ 0 = \theta_0 \leq \theta_1 \leq \cdots \leq \theta_N \leq \pi$, where $ \theta_{-k} := - \theta_k $.  Let 
\begin{equation} \label{def C even}
 C_{even}(n) :=   \lim_{N \to \infty}     \int_{SO (2N)}  \prod_{\ell =1}^n \left[  \sum_{ 0< |  j |  \leq N     } \Phi_\ell  \bigg( \frac{N \theta_{j }}{ \pi}   \bigg) - \widehat{\Phi}_\ell (0) - \frac{ \Phi_\ell ( 0 )}{2} \right]  dX_{SO(2N)}      
 	\end{equation}
 and	
\begin{equation}\label{def C odd}
C_{odd}(n)    := \lim_{N \to \infty}     \int_{ O^{-} (2N+2 )} \prod_{\ell=1}^n \left[  \sum_{ 0 \leq |  j|  \leq N     } \Phi_\ell  \bigg( \frac{N \theta_{j }}{ \pi}   \bigg) - \widehat{\Phi}_\ell (0) - \frac{ \Phi_\ell ( 0 )}{2} \right]   dX_{O^{-}(2N+2)}   ,
 	\end{equation}
 where  $dX_S$ is the measure induced by the Haar measure on $O(N)$, normalized such that $S$ has measure $1$. Then the $n$-th centered moment for $O(N)$ is 
 \begin{equation}\label{def C}
  C(n) := \frac{1}{2} (C_{even}(n) + C_{odd}(n)).
  \end{equation}
 
Our main theorem for general $n$ is below.
  \begin{thm}\label{thm:main} Assume GRH. Let $\Phi_i$ be an even Schwartz function with $\widehat \Phi_i$ compactly supported in $(-\sigma_i, \sigma_i)$, where $\sum_{i = 1}^n \sigma_i < 4$. Then with notation as before,
$$ \lim_{Q \rightarrow \infty} \mathscr{L}_n(Q) = C(n).
$$
\end{thm}
In contrast to previous work on the $n^{\text{th}}$ centered moments for orthogonal families \cite{HM, REUandMiller, Coh et al}, we encounter off-diagonal main terms contributing to $C(n)$ which requires careful identification.  We describe $C(n)$ more precisely below, but here we first qualitatively discuss the case where all the test functions $\Phi_i = \Phi$ are the same, as this easily facilitates comparison with previous works.  In this case, our result allows for $\widehat{\Phi}$ to be supported in $(-4/n, 4/n)$.

Hughes and Rudnick studied such statistics in the case of Dirichlet $L$-functions in \cite{HR} and \cite{HR2}, where the test functions $\Phi_i = \Phi$ for all $i$, with the support of $\widehat{\Phi}$ restricted to $(-2/n, 2/n)$.  The moments they derived matched Gaussian moments.  However, based on calculations on the random matrix side, they conjectured that such moments would not be Gaussian if the support is suitably extended.  Hughes and Miller \cite{HM} studied this for the orthogonal family of automorphic $L$-functions similar to ours, but without an average over the level $q$.  Their work was extended by the recent work of Cohen et al. \cite{Coh et al}, with the best known result when $\Phi_i = \Phi$ for all $i$ and again when $\widehat{\Phi}$ has support in $(-2/n, 2/n)$.  These works verify that the moments are non-Gaussian when the average is restricted to even forms or odd forms.  However, the non-Gaussian term from the even forms and the odd forms precisely cancel in their work, so the non-Gaussian behavior is not visible for the full family. 

It is noteworthy that with the support restricted to $(-2/n, 2/n)$ as discussed above, the $n^{\text{th}}$ centered moments do not distinguish between the unitary group and the full orthogonal group, as both exhibit Gaussian behavior.  To see the non-Gaussian behavior for our family, the support needs to be further extended, and our Theorem \ref{thm:main} and Corollary \ref{cor:Phisupported in 4/n} verify the expected deviation from Gaussian for the full orthogonal family for the first time.  

To be more precise, we need to introduce some notation involving set partitions. 
\begin{defi}\label{def:set partition}
A set partition $\underline G = \{ G_1 , \ldots , G_\nu \} $ of a finite set $K$ is a decomposition of $K$ into disjoint nonempty subsets $G_1 , \ldots , G_\nu$.   Let $\Pi_K$ be the collection of these set partitions.  Let $ \pi_{K,1} = \{ \{k \} \ | \ k \in K \}  \in \Pi_K $ and define $ \Pi_{K,2} $ by the set of $ \underline G \in \Pi_K $ such that $ | G_i | = 2 $ for all $ G_i \in \underline G $. We also 
 let $ \Pi_n := \Pi_{ [n]} $ and $ \Pi_{n,2} := \Pi_{ [n], 2 }$ for a positive integer $n$, where
$$ [n] := \{ 1, 2, \ldots, n \} . $$
\end{defi}

We have the following expression for $C(n)$.
\begin{thm}\label{thm:Cn}
Suppose that $ \sum_{i\leq n } \sigma_ i <  4 $. Then we have
\begin{equation} \label{def:Cn} 
C (n) = C_0 (n) + C_2 (n), 
\end{equation}
where
\begin{equation} \label{def:C0(n)}
 C_0 (n)  :=  \sum_{ \underline G \in \Pi_{n,2} }           \prod_{G_i \in \underline{G}}  \mathscr{I}_{2 }  ( G_i  )      
\end{equation}
     and
     \begin{equation} \label{def:C2(n)}
       C_2 (n): =  \sum_{ \substack{    K_0 \sqcup K'     \sqcup K'' = [n]   \\
   |K'| = 2   } }            \mathscr{V} ( K' , K'' )  \sum_{  \underline G \in \Pi_{K_0,2} }  \prod_{G_i \in \underline G}  \mathscr{I}_{2 }  (G_i  )      ,
     \end{equation}
where   $ \mathscr{V} (K', K'' )  $ is defined in \eqref{def V K} and 
\begin{equation} \label{def:I2Kr}
\mathscr{I}_{2 }  ( \{ k_1 , k_2 \}  )  := 2 \int_{-\infty}^\infty |t| \widehat{\Phi}_{k_1} (t) \widehat{\Phi}_{k_2} (t) dt.
\end{equation}
In particular, we have
\begin{multline}\label{def:V}
 \mathscr{V} ( \{k_1, k_2 \}, G) = \sum_{\substack{G_1 \sqcup G_2 \sqcup G_3 \sqcup G_4 = G \\ G_3 \subset \{k_1 + 1,..., n \} \\ G_4 \subset \{k_2 + 1,...,n \}  } } (-2)^{|G|+ |G_1| + |G_2|}  \\
 \times \int_{[0, \infty)^{|G_1| + |G_2|}}  \mathscr{I} \left( \Phi_{k_1 , G_3 }, \Phi_{k_2 , G_4} ; \sum_{j \in G_1} w_j , \sum_{j \in G_2 } w_j  \right) \prod_{j \in G_1 \sqcup G_2} \widehat \Phi_j(w_j) \> dw_j    
  \end{multline}
for $ \{ k_1  , k_2 \} \sqcup G \subset [n]$,  where   
\begin{equation}\label{def:I12}\begin{split}
\mathscr{I}( \Phi_1, \Phi_2 ; U_1, U_2 ) :=  &  \int_{0}^{\infty} \int_{0}^{\infty} 
 \widehat \Phi_1\left( t_1+1+ U_1 \right)    \widehat \Phi_2 \left( t_2 +1+ U_2 \right)  \> dt_1 \> dt_2   \\
 & - 4\int_{0}^{\infty}  t
 \widehat \Phi_1\left( t+1+U_1  \right)  \widehat \Phi_2\left( t +1+ U_2  \right)   \> dt
 \end{split}\end{equation}
and 
 \begin{equation} \label{eqn:PhikG} 
 \Phi_{k, G}(x) := \Phi_k(x) \prod_{j \in G} \Phi_j(x).
 \end{equation}

\end{thm}

To illustrate Theorem \ref{thm:main}, let $\Phi_i = \Phi$ for all $i$, and define
 $$ \sigma^2_{\Phi} := 2 \int_{-\infty}^\infty |t| \widehat{\Phi} (t)^2 \> dt.$$ This coincides with $\mathscr I_2(\{k_1, k_2 \})$, as defined in \eqref{def:I2Kr}, when $\Phi_{k_1} = \Phi_{k_2} = \Phi.$ 
 
\begin{cor} \label{cor:Phisupported in 4/n} Let $\Phi$ be an even Schwartz function with $\widehat \Phi$ is compactly supported in $\left( -\tfrac 4n, \tfrac 4n\right)$, and $\mathscr L_n(Q)$ be defined as before with $\Phi_i = \Phi$ for all $1\le i\le n$. Then 
\est{\lim_{Q \rightarrow \infty} \mathscr L_n(Q) =  (n - 1)!! ( \sigma^2_{\new\Phi})^{n/2} \delta_{even}(n) + C_2(n)}
    where $(n-1)!!$ denotes the product of all the positive integers up to $n\new{-1}$ that have the same parity as $n\new{-1}$, and $\delta_{even}(n)$ equals 1 if $n$ is even and 0 otherwise.
\end{cor}

\begin{rem} 
In our Corollary \ref{cor:Phisupported in 4/n}, the first term  $(n - 1)!! ( \sigma^2_{\new\Phi})^{n/2} \delta_{even}(n)$ corresponds to Gaussian moments, while the second term $C_2(n)$ represents deviation from Gaussian behavior.

When $\widehat \Phi$ is compactly supported in $\left( -\tfrac 2n, \tfrac 2n\right)$, a little calculation shows that $C_2(n)=0$ so that the moments exhibit Gaussian behavior.  However, our Corollary \ref{cor:Phisupported in 4/n} implies that generically the moments are not Gaussian when the support of $\widehat{\Phi}$ is not inside $\left( -\tfrac 2n, \tfrac 2n\right)$.
\end{rem}

We also mention the work of Cheek et al. \cite{REUandMiller}, which studies the same family by extending the work of Baluyot, Chandee and Li \cite{BCL}.  In their Theorem 1, they derive a result with complicated restrictions on the support.  When $\widehat{\Phi_i}$ are taken to have the same support, their support conditions look roughly similar to the support $(-2/n, 2/n)$ in the work of Cohen et al. \cite{Coh et al} for large $n$.  In contrast, their work presents an unexpected feature when the test functions are not the same, and the supports differ.  In particular, their Corollary 1.2 states a result for $n=2$ where $\sigma_1 = 3/2$ and $\sigma_2 = 5/6$.  The underlying cause of this curiously asymmetric setup is the presence of an oversized phantom contribution from the continuous spectrum.  Once properly understood, we will see that the phantom contribution vanishes.  A proper identification of the phantom term reveals additional off-diagonal contributions from the continuous spectrum, which is closely related to the non-Gaussian behavior exhibited in Corollary \ref{cor:Phisupported in 4/n}.  In contrast, in the previous work of Baluyot, Chandee and Li \cite{BCL} on the case $n=1$, this phantom contribution did not present difficulties, and there were no off-diagonal contributions to the main term.  We describe this in more detail in the outline in \S \ref{sec:outline}.     

The study of $\mathscr{L}_n(Q)$ presents a number of significant new difficulties for larger $n$.  One well known difficulty in such problems is that after a successful asymptotic evaluation of $\lim_{Q \rightarrow \infty}\mathscr{L}_n(Q)$, it is not clear that the resulting expression agrees with the random matrix prediction.  To be precise, proving that $ \lim_{Q \rightarrow \infty} \mathscr{L}_n(Q)$ agrees with $C(n)$ is a challenging combinatorics problem.  A well known example is the work of Gao \cite{Gao} where the number theory side was computed, but it was not until the work of Entin, Roditty-Gershon and Rudnick \cite{ERRudnick} when this was successfully matched with the random matrix prediction.  

In the previous works \cite{HR} \cite{HM} \cite{Coh et al}, this combinatorial matching was accomplished with a difficult argument involving cumulants.  The support allowed in our result is double or more compared to previous works, rendering such an argument even more arduous.  In this paper, we instead build on the work of Mason and Snaith \cite{MS}, which extends earlier work of Conrey and Snaith \cite{CS1, CS2}.  This simplifies combinatorial arrangements in the case of large support and allows us to find the explicit integral representation for $C(n)$ stated in Theorem \ref{thm:Cn} in terms of the $ \Phi_i$ and $ \widehat{\Phi}_i$, which is of independent interest.

As mentioned before, our result would lead to high quality bounds towards the proportion of $L$-functions which do not vanish to large order and other related problems.  We omit such bounds here due to the length and technical depth of the current paper. 

\subsection{Outline of the paper} \label{sec:outline}

We now provide an outline to the rest of the paper, focusing more on the flow of ideas, and suppressing technical details.

In \S \ref{sec:prelimresults}, we introduce some notation and preliminary results. In \S \ref{sec:outline of proof main theorem}, we setup the initial steps in the proof of Theorem \ref{thm:main}.  In particular, by the explicit formula, we want to study a quantity roughly of the form
$$ \mathscr{L}_n(Q) := \frac{1}{Q}\sum_q \Psi\bfrac{q}{Q} \sumh_{f \in \mathcal H_{k}(q)} \prod_{i = 1}^n \bigg[\frac{1}{\log Q}\sum_{\substack{p_i \\ p_i \nmid q}}\frac{\log p_i \lambda_f(p_i)}{\sqrt{p_i}} \widehat\Phi_i\left(\frac{\log p_i}{\log Q}\right) + O\bfrac{\log \log Q}{\log Q} \bigg], $$
where the $O\bfrac{\log \log Q}{\log Q}$ comes, for example, from the contributions of the prime squares.  Our first step is to get rid of these $O\bfrac{\log \log Q}{\log Q}$, which requires some dexterity.  This is because we now need to bound quantities involving 
$$\bigg| \sum_{\substack{p_i \\ p_i \nmid q}}\frac{\log p_i \lambda_f(p_i)}{\sqrt{p_i}} \widehat\Phi_i\left(\frac{\log p_i}{\log Q}\right)\bigg|,
$$and the sum over primes inside the absolute value may be too long to allow the use of Cauchy-Schwarz or H\"older inequality due to the fact that we allow the support of $\Phi_i$ to differ.  Instead, we take advantage of the uncertainty principle by exchanging the (morally long) sums over primes for short sums over zeros.  We then bound the short sums over zeros by long sums over zeros using positivity.  The long sums over zeros convert to short sums over primes, which can be bounded easily.

Next, we reduce the sums over primes to sums over distinct primes, dependent on some set partition of $\{1,...,n\}$, and isolate those set partitions which contribute.  The relevant Propositions for the above are stated in \S \ref{sec:outline of proof main theorem}, and proven in \S \ref{section:proof props}.

We now want to apply Petersson's formula to understand a sum of the form
$$\sumh_{f \in \mathcal H_{k}(q)} \frac{1}{\log^n Q}\sum_{\substack{m \le Q^{4-\delta} \\ (m, q) = 1}}\frac{a(m) \lambda_f(m)}{\sqrt{m}},
$$
where $a(m)$ is some coefficient which restricts $m$ to products of $n$ primes and our restriction $m \le Q^{4-\delta}$ is inherited from the support conditions on $\widehat{\Phi_i}$.

In the application of Petersson's formula for primitive forms, we see a complicated inclusion-exclusion-type of formula, which we need to prune in \S \ref{section:proof of main proposition}.  Ignoring such technicalities, we are left to consider a quantity roughly of the form

\begin{align*}
    \frac{1}{Q} \sum_{\substack{m \asymp Q^{4-\delta} }}\frac{a(m)}{\sqrt{m}}   \sum_q \Psi\bfrac{q}{Q} \sum_{c} \frac{S(m, 1; cq)}{cq} J_{k-1}\bfrac{\sqrt{m}}{cq},
\end{align*}
where we have removed the condition $(m, q) = 1$ and assumed $m\asymp Q^{4-\delta}$ for convenience.  In the transition region of the Bessel function, we have
$$c \asymp \frac{\sqrt{m}}{q} \asymp Q^{1-\delta/2}
$$is smaller than $q$.  Hence, it makes sense to switch to the complementary level $c$ by applying Kuznetsov's formula to the sum over $q$.  This sort of complementary level trick was previously used in the work of Baluyot, Chandee, and Li in the context of one level density \cite{BCL}.  The earliest appearance that we are aware of is in \S 8 of Deshouillers and Iwaniec's foundational work \cite{DI} in their bound on the contribution of exceptional eigenvalues on average over the level.  The details here appear in \S \ref{sec:applyKuznetsov}.  

The result of this is a sum over the complementary level $c\asymp Q^{1-\delta/2}$ of holomorphic cusp forms, Maass forms, and Eisenstein contributions.  The contribution of the holomorphic forms and Maass forms are bounded in \S \ref{sec:dispart}.  The contribution of the continuous spectrum is separated into the contribution of the trivial character and the non-trivial characters.  The non-trivial characters give a small contribution, and this is shown in \S \ref{sec:ctn}.  In both of these bounds, we write the orthonormal basis from Kuznetsov's formula in terms of primitive forms, and then GRH is invoked to bound the sums over primes by $Q^\epsilon$, so that the resulting bound looks like
$$\frac{1}{Q}\sum_{c\asymp Q^{1-\delta/2}} Q^\epsilon \ll Q^{-\delta/2 + \epsilon}.
$$

In \S \ref{sec:trivialchar off diag main terms}, we begin the treatment of the contribution of the trivial character.  Here, the prime sums by themselves can be genuinely huge, giving a contribution that appears far larger than the main term.  There are a number of examples in the literature where oversized contributions from Eisenstein series are canceled out.  The first example of this appears in the work of Duke, Friedlander, Iwaniec \cite{DFI}.  We also mention the work of Blomer, Humphries, Khan, and Milinovich \cite{BHKM}, which has a setup similar to our work.  In both works, they start with an average over Maass forms and Eisenstein series, and the Eisenstein contribution on one side cancels out the Eisenstein contributions on the other side of Kuznetsov.  We start with holomorphic modular forms, and so we instead use the orthogonality of the space of holomorphic cusp forms with the continuous spectrum.  To explain this conceptually, we note that if we had no restrictions on the level $q$, the contribution of the Eisenstein series is weighted by
\begin{equation}\label{eqn:outlineJortho}
\int_0^{\infty} (J_{2ir}(\xi) - J_{-2ir} (\xi) )  J_{k-1}(\xi) \,\frac{d\xi}{\xi} = 0,
\end{equation}
which is simply an echo of the orthogonality of the space of cusp forms with the Eisenstein spectrum.  In our work we have the presence of $\Psi\bfrac{q}{Q}$ restricting $q\asymp Q$, which using Mellin inversion gives us an integral transform of the form 
\begin{equation}\label{eqn:outlineJs}
\int_0^{\infty} (J_{2ir}(\xi) - J_{-2ir} (\xi) )  J_{k-1}(\xi) \xi^s \,\frac{d\xi}{\xi}.
\end{equation}
In general, this means that the contribution of the Eisenstein spectrum is nonzero.  However, when we restrict our attention to only the contribution of the trivial character, and when we additionally sum over the complementary level, we are led to study a quantity very roughly like
\begin{align*}
&\int_{(-\epsilon_1)}  \int_{-\infty}^{\infty} \bfrac{Q}{4\pi}^{s} \widetilde \Psi(s) \zeta(1 - s) \zeta(2 - s) \sum_{m \asymp Q^{4-\delta}} \frac{a_{it}(m)}{m^{1/2+s/2}} \\
\times 
&\int_0^{\infty} (J_{2it}(\xi) - J_{-2it} (\xi) )  J_{k-1}(\xi) \xi^s \,\frac{d\xi}{\xi} \> dt \> ds,
\end{align*}
for some coefficient $a_{it}(m)$ depending on the spectral parameter $t$.  \footnote{This has been oversimplified for illustrative purposes and we refer the reader to \eqref{eqn:SigmaCtn0} - \eqref{def:K1Lsz}  for the precise version.}  The phantom term comes from the pole of $\zeta(1-s)$ at $s = 0$, which appears to give a contribution of size roughly

$$\frac{1}{Q} \sum_{m \asymp Q^{4-\delta}} \frac{1}{m^{1/2}} \asymp Q^{1-\delta/2},
$$and this is much larger than the main term of size $1$.

However, this pole is cancelled by the zero of \eqref{eqn:outlineJs} at $s=0$ due to the orthogonality relation \eqref{eqn:outlineJortho}.  We then extract off-diagonal main terms from this contribution near the $s=1$ line.  Here, we have neglected to present the inherent complexity of the task, especially the special combinatorics of this problem.  The complex combinatorial phenomena presents serious impediments in all previous works of this type.

We refer the reader to \S \ref{sec:trivialchar off diag main terms} and \S \ref{sec:ctnrescalc} for the details, where a number of combinatorial arrangements are made, parallel to the computations over random matrices in \S \ref{section:nth centered moment ON}.  In this outline, we only point out one particular feature of this computation, which gives some hints towards the combinatorics involved, and also reflects the inherent properties of the family.

For simplicity, suppose that $n=2$, so that we have two prime sums, one of length $P_1$ and the other of length $P_2$.  The Prime Number Theorem \footnote{Here, we can assume a small error term, assuming RH.} would show us that the contributions of the prime sums give rise to factors like $P_j^{\pm it}$, where $t$ is the spectral parameter.  The contribution of $(P_1 P_2)^{\pm it}$ can be shown to be negligible by setting $z = it$ and shifting the contour appropriately in $z$.  Thus the main term has to involve terms like $P_1^{it} P_2^{-it}$ or $P_1^{-it} P_2^{it}$.  We refer the reader to Lemma \ref{lemma:SigmaCtn0 to SigmaCtn1}  for the actual statement.  This pairing phenomenon correlates with the conjectural behavior of the moments of this family involving even swaps (e.g. \S 4.5 of \cite{CFKRS}).  Both are closely related to the fact that our $L$-functions has root number $\pm 1$ which square to $1$.  

When applying Kuznetsov in \S \ref{sec:applyKuznetsov}, we need to remove a coprimality condition of the form $(m, q) = 1$.  This condition was desirable before to apply Petersson's formula, but is now an impediment.  Removing this condition results in sums which can be treated similarly to our main sum, and which would result in contributions which are a power of $\log Q$ less than the actual main term.  The proof of this is sketched in \S \ref{sec:fillprimes}.

Lastly, we prove Theorem \ref{thm:Cn} in \S \ref{section:nth centered moment ON}, which is logically independent of the other sections.  However, we emphasize that the random matrix theory calculation and the computations on the number theory side mirror each other in the computations of the main terms. The resulting formula for $C(n)$ in Theorem \ref{thm:Cn} provided a useful guide for the computations of the main terms on the number theory side.

We start the random matrix computation from the observation that the integrals on $USp(2N)$ and $ O^- ( 2N+2)$ are essentially the same as in Lemma \ref{lemma O- to USp}. This allows us to apply the results for $SO(2N)$ and $USp(2N)$ for non-normalized $n$-th centered moments in Mason and Snaith \cite{MS}.  The resulting expressions are combinatorially complicated and we proceed to make some simplifications.  In particular, we show nice cancellation in the deduction \eqref{eqn Cn Tplusn} from \eqref{eqn:C even odd nice cancellation} where the terms with odd $|K'| $ in \eqref{eqn:C even odd nice cancellation} cancel each other out.
Then, applications of Fourier inversion and complex analysis leads to the proof of Theorem \ref{thm:Cn}.


\section{Notation and Preliminary Results} \label{sec:prelimresults}

\subsection{Notation}
Throughout the paper, we adopt the standard convention in analytic number theory of letting $\epsilon$ denote an arbitrarily small positive real number, whose value may vary from line to line. In contrast, the symbols $\epsilon_i$ and $\delta$ represent fixed positive constants. We  use $p$ (and subscripts of $p$) exclusively to denote prime numbers. For a finite set $K$ of positive integers and a positive integer $\kappa$, we define the product of primes as
\begin{equation*}\label{def:mathfrak p K}
\mathfrak p (K) := \prod_{ j \in K }p_j , \qquad \mathfrak p (\kappa)  := \mathfrak p([\kappa]) = p_1  \cdots p_\kappa.
\end{equation*}

We use $\sumsharp$ a sum over mutually distinct indices. We write $\e{x} = \exp(2\pi i x)$, and   $A \sqcup B$ is the disjoint union of sets $A$ and $B$. Also, a function $\delta_{\textrm{condition}}$ equals 1 if the condition is satisfied, and 0 otherwise.

\subsection{Petersson's formula and related results}
 We state the orthogonality relations for our family. These are the standard Petersson's formula (e.g. see \cite{Iwaniec}), and a version of Petersson's formula that is restricted to newforms and is due to Ng~\cite{Ng}.

Recall that $S_k(q)$ is the space of cusp forms of weight $k$ and level $q$. Let $B_k(q)$ be any orthogonal basis of $S_k(q)$. Define
\begin{equation} \label{def:Deltaq(m,n)}
\Delta_q(m,n ) = \Delta_{k, q}(m, n) = \sumh_{f\in B_k(q)}\lambda_f(m)\lambda_f(n),
\end{equation}
where the summation symbol $\sumh$ means we are summing with the same weights found in \eqref{eqn: harmonicweightsummation}. The usual Petersson's formula (e.g. see \cite{Iwaniec}) is the following.
\begin{lemma}\label{lem:usualPetersson}
If $m,n,q$ are positive integers, then
$$ \Delta_q(m, n) = \delta(m, n)+ 2\pi i^{-k} \sum_{c\geq 1} \frac{S(m, n;cq)}{cq} J_{k-1}\bfrac{4\pi \sqrt{mn}}{cq},
$$
where $\delta(m,n)=1$ if $m=n$ and is $0$ otherwise, $S(m,n;cq)$ is the usual Kloosterman sum, and $J_{k-1}$ is the Bessel function of the first kind.
\end{lemma}

Lemma~\ref{lem:usualPetersson}, the Weil bound for Kloosterman sums, and standard facts about the Bessel function imply the following lemma (see \cite[Corollary 2.2]{ILS}).
\begin{lemma}\label{lem:petertruncate}
If $m,n,q$ are positive integers, then
\begin{equation*}
\Delta_q(m, n) = \delta(m, n) + O\left(\frac{\tau(q) (m, n, q)(mn)^{\epsilon}}{q ((m, q)+(n, q))^{1/2}} \bfrac{mn}{\sqrt{mn} + q}^{1/2} \right),
\end{equation*}
where $\tau(q)$ is the divisor function and $\delta(m,n)=1$ if $m=n$ and is $0$ otherwise.
\end{lemma}

For our purposes, we need to isolate the newforms of level $q$.  To be precise, recall that $\mathcal H_k(q)$ is the set of newforms of weight $k$ and level $q$ which are also Hecke eigenforms. We need a formula for
\begin{equation*}
\Delta_q^*(m, n) := \sumh_{f\in \mathcal H_k(q)} \lambda_f(m)\lambda_f(n).
\end{equation*}
A formula is known for squarefree level $q$ due to Iwaniec, Luo and Sarnak~\cite{ILS}, and for $q$ a prime power due to Rouymi~\cite{rouymi}. These formulas have been generalized to all levels $q$ by Ng \cite{Ng} (see also the works of Barret et al.\cite{BBDDM}, and Petrow \cite{Pe}). Ng's Theorem 3.3.1 contains some minor typos, but the corrected version is as follows.
 \begin{lemma}\label{lem:PeterssonNg}
Suppose that $m,n,q$ are positive integers such that $(mn,q)=1$, and let $q = q_1q_2$, where $q_1$ is the largest factor of $q$ satisfying $p|q_1 \Leftrightarrow p^2|q$. Then
\begin{align*}
\Delta_q^*(m,n) = \sum_{\substack{q=L_1L_2d \\ L_1|q_1 \\ L_2|q_2}} \frac{\mu(L_1L_2)}{L_1L_2}  \prod_{\substack{p|L_1 \\ p^2 \nmid d}}  \left( 1-\frac{1}{p^2} \right)^{-1} \sum_{\elltwo |L_2^{\infty} }\frac{\Delta_d(m,n\elltwo^2)}{\elltwo}.
\end{align*}
Furthermore, the condition that $L_1|q_1$ and $L_2|q_2$ is equivalent to the condition that $L_1|d$ and $(L_2,d)=1$. 
\end{lemma}

For a proof, see \cite[Lemma 2.3]{BCL} and its remark.

\subsection{Kuznetsov's formula} \label{sec:kuz}

In this section, we state some relevant results from spectral theory. We refer the reader to \cite{DI} and \cite{Iwaniec} for background reading.  

We start by introducing some notation that will appear in Kuznetsov's formula. There are three parts in Kuznetsov's formula---contributions from holomorphic forms, Maass forms, and Eisenstein series---and we now define the Fourier coefficients of these forms.

\subsubsection*{Holomorphic forms}

Let $B_{\ell}(N)$ be an orthonormal basis of the space of holomorphic cusp forms of weight $\ell$ and level $N$, and  $\theta_{\ell}(N)$ be the dimension of the space $S_{\ell}(N)$. We can write $B_{\ell}(N) = \{f_1, f_2,...., f_{\theta_{\ell}(N)} \}$, and the Fourier expansion of $f_j \in B_{\ell}(N)$ can be expressed as follows
$$ f_j(z) = \sum_{n \geq 1} \psi_{j, \ell}(n) (4\pi n)^{\ell / 2} \e{n z}.$$

We call $f$ a Hecke eigenform if it is an eigenfunction of all the Hecke operators $T(n)$ for $(n, N) = 1$.  In that case, we denote the Hecke eigenvalue of $f$ for $T(n)$ as $\lambda_f(n)$.  Writing $\psi_f(n)$ as the Fourier coefficient, we have that
$$\lambda_f(n) \psi_f(1) = \sqrt{n} \psi_f(n),
$$for $(n, N) = 1$.  When $f$ is a newform, this holds for all $n$.  We also have the Ramanujan bound
$$\lambda_f(n) \ll \tau(n) \ll n^\epsilon.
$$

\subsubsection*{Maass forms}
Let 
$
\lambda_j := \frac{1}{4}+\kappa_j^2,
$
where
$
0=\lambda_0 \leq \lambda_1\leq \lambda_2 \leq \dots
$
are the eigenvalues, each repeated according to multiplicity, of the Laplacian $-y^2 ( \frac{\partial^2}{\partial x^2} + \frac{\partial^2}{\partial y^2})$ acting as a linear operator on the space of cusp forms in $L^2(\Gamma_0(N) \backslash \mathbb{H})$, where by convention we choose the sign of $\kappa_j$ that makes $\kappa_j\geq 0$ if $\lambda_j\geq \frac{1}{4}$ and $i\kappa_j >0$ if $\lambda_j <\frac{1}{4}$. For each of the positive $\lambda_j$, we may choose an eigenvector $u_j$ in such a way that the set $\{u_1,u_2,\dots\}$ forms an orthonormal system, and we define $\rho_j(m)$ to be the $m$th Fourier coefficient of $u_j$, i.e.,
$$
u_j(z) =  \sum_{m\neq 0} \rho_j(m)W_{0, i\kappa_j} (4\pi |m|y) \e{mx}
$$
with $z=x+iy$, where $W_{0, it}(y) =  \left( y/\pi\right)^{1/2}K_{it}(y/2)$ is a Whittaker function, and $K_{it}$ is the modified Bessel function of the second kind. 

We call $u$ a Hecke eigenform if it is an eigenfunction of all the Hecke operators $T(n)$ for $(n, N) = 1$.  In that case, we denote the Hecke eigenvalue of $u$ for $T(n)$ as $\lambda_u(n)$.  Writing $\rho_u(n)$ as the Fourier coefficient, we have that
\begin{equation}\label{eqn: fouriercoeffintermsofeigenvalue}
\lambda_u(n) \rho_u(1) = \sqrt{n} \rho_u(n)
\end{equation}
for $(n, N) = 1$.  When $u$ is a newform, this holds in general.  We also have that
\begin{equation}\label{eqn:KimSbdd}
\lambda_u(n) \ll \tau(n) n^{\theta} \ll n^{\theta+ \epsilon},
\end{equation}
where we may take $\theta = \frac{7}{64}$ due to work of Kim and Sarnak \cite{KimS}.

\subsubsection*{Eisenstein series}
We follow the treatment of Blomer and Khan \cite{BK}, whose work is in turn based on the work of Knightly and Li \cite{KL}.

The Eisenstein series for $\Gamma_0(N)$ are parametrized by a pair $(\chi, M)$ and the spectral parameter $s = 1/2 + it$.  Here $\chi$ is a primitive Dirichlet character modulo $\c_\chi$, and we have that $\c_\chi^2|M|N$.  We chose this parametrization as the principal character contribution from the Eisenstein series needs to be explicitly calculated, and this is more convenient for that purpose.
We write $M = \c_\chi M_1 M_2$ where $(M_2, \c_\chi)=1$ and $M_1 | \c_\chi^{\infty}$.  

The Eisenstein series $E_{\chi, M, N}(z, s)$ of level $N$ corresponding to   $(\chi, M)$ has the Fourier expansion 
$$E_{\chi, M, N}(z, 1/2 + it) = \rho^{(0)}_{\chi, M, N}(t, y) + \frac{2 \pi^{1/2 + it} y^{1/2}}{\Gamma(1/2 + it)} \sum_{n \not= 0} \rho_{\chi, M, N}(n, t) K_{it}(2\pi|n|y) \e{nx}.$$
The coefficients $\rho_{\chi, M, N}$ are defined by
\es{\label{eqn:Eisensteincoeff}
	\rho_{\chi, M, N}(n, t) := &\frac{\widetilde{C}(\chi, M, t) \sqrt{M_1 \zeta_{(M, N/M)}(1) }}{ \sqrt{M_2 N}   L^{(N)}(1+2it, \chi^2)}    |n|^{it} \rho'_{\chi, M, N}(n, t), \\ 
     \rho'_{\chi, M, N}(n, t):= & \sum_{m_2 | M_2} m_2  \mu\bfrac{M_2}{m_2} \bar{\chi}(m_2)\sum_{\substack{n_1 n_2 = n /( M_1 m_2)  \\ (n_2, N/M)=1}} \frac{\bar \chi(n_1)\chi(n_2)}{n_2^{2it}} ,
}
where $L^{(N)}$ is the Dirichlet $L$-function with the Euler factors at primes dividing $N$ omitted and
\begin{equation}\label{def:zetaNs}
\zeta_N (s) :=   \prod_{p | N} \left(1- \frac{1}{p^s }  \right)^{-1}.
\end{equation}
Moreover, $|\widetilde{C}(\chi, M, t)|=1$. In our application, we always have an expression of the form $\rho_{\chi, M, N}(n, t) \overline{\rho_{\chi, M, N}(m, t)}$ and $\widetilde{C} \overline{\widetilde{C}} = 1$, so we do not need anything more explicit.

\subsubsection*{Kuznetsov's formula} We state the version given by \cite[Lemma 10]{BM}, but with Fourier coefficients of Eisenstein series given by \eqref{eqn:Eisensteincoeff}. 
\begin{lem}\label{lem:kuznetsov}
Let  $m$, $n$ and $N$ be positive integers and let $J_\alpha (\xi) $ be the Bessel function of the first kind. Suppose that $\phi:(0,\infty)\rightarrow \mathbb{C}$ is smooth and compactly supported. Then we have
\begin{align*}
\sum_{\substack{c\geq 1 \\ c\equiv 0 \bmod{N}}} \frac{S(m,n;c)}{c} \phi
& \bigg( 4\pi \frac{\sqrt{mn}}{c}\bigg) = \sum_{j=1}^{\infty} \frac{\overline{\rho_j}(m)\rho_j(n) \sqrt{mn}}{\cosh(\pi \kappa_j)}\phi_+(\kappa_j)\\
&+\frac{1}{4\pi}\sum_{ \c_{\chi}^2\mid M \mid N} \int_{\Bbb{R}} \rho_{\chi, M, N}(n, t) \overline{ \rho_{\chi, M, N}(m, t)} \phi_+ (t) dt\\
& + \sum_{ \substack{\ell \geq 2 \mbox{\scriptsize{ \upshape{even}}} \\ 1 \leq j \leq \theta_{\ell}(N)} } (\ell-1)! \sqrt{mn} \, \overline{\psi_{j,\ell}}(m) \psi_{j,\ell} (n) \phi_h(\ell),
\end{align*}
where the Bessel transforms $\phi_+$ and $\phi_h$ are defined by
$$
\phi_+(r):=\frac{2\pi i}{\sinh(\pi r)} \int_0^{\infty} (J_{2ir}(\xi) - J_{-2ir} (\xi) ) \phi(\xi) \,\frac{d\xi}{\xi}
$$
and
$$
\phi_h(\ell) := 4 i^\ell \int_0^{\infty} J_{\ell - 1}(\xi) \phi(\xi) \,\frac{d\xi}{\xi}.
$$
\end{lem}

We next state bounds for the transforms $\phi_+$ and $\phi_h$ that appear in Kuznetsov's formula. These bounds are consequences of \cite[Lemma 1]{BHM}. 
\begin{lem} \label{lem:boundforcoefficientafterKuznetsov}
\begin{enumerate}
\item Let $\phi(x)$ be a smooth function supported on $x \asymp X$ such that $\phi^{(j)}(x) \ll_{j} X^{-j}$ for all integers $j \geq 0$. For $t \in \mathbb R$, we have
$$ \phi_+(t), \ \  \phi_h(t) \ll_C \frac{1 + |\log X|}{1 + X} \left( \frac{1 + X}{1 + |t|}\right)^C$$
for any constant $C \geq 0.$
\item  Let $\phi(x)$ be a smooth function supported on $x \asymp X$ such that $\phi^{(j)}(x) \ll_{j} (X/Z)^{-j}$ for all integers $j \geq 0$. For $t \in (-1/4, 1/4)$, we have
$$ \phi_+ (it) \ll \frac{1 + (X/Z)^{-2 |t|}}{1 + X/Z}.$$
\item Assume that $\phi(x) = e^{iax} \psi(x)$ for some constant $a$ and some smooth function $\psi(x)$ supported on $x \asymp X$ such that $\psi^{(j)}(x) \ll_j X^{-j}$ for all integers $j \geq 0.$  Then
$$ \phi_+(t), \ \  \phi_h(t) \ll_{C, \epsilon} \frac{1+|\log X|}{F^{1 -\epsilon}} \left( \frac{F}{1 + |t|} \right)^C$$
for any $C \geq 0,$ $\epsilon > 0$ and some $F = F(X, a) < (|a|+1)(X+1).$
\end{enumerate} 
\end{lem}
Lemma~\ref{lem:boundforcoefficientafterKuznetsov}~(3) is a slight generalization of \cite[Lemma~1~(c)]{BHM}. This generalization incorporates the bound in \cite[Lemma~1~(a)]{BHM}. It is convenient for us that Lemma~\ref{lem:boundforcoefficientafterKuznetsov}~(3) holds uniformly for all $a$.

Next, we record the following bounds from \cite[Lemma 3.3]{BCL}. 
\begin{lem} \label{lem:boundforhu}
Suppose that $W$ is a smooth function that is compactly supported on $(0,\infty)$. For real $X>0$ and real numbers $u$ and $\xi$, let 
\begin{align*}
h_u(\xi)  = J_{k-1}(\xi)W\bfrac{\xi}{X} \e{u\xi}.
\end{align*}
Then for all $C\ge 0$,
\begin{enumerate}
    \item $h_{u,+} (r) \ll    \frac{1 + |\log X|}{F^{1 - \epsilon}} \left( \frac{F}{1 + |r|}\right)^C \min \left\{ X^{k - 1}, \frac 1{\sqrt X}\right\} \ \ \ \ \ \ \ \   \textrm{for some} \ F < (|u| + 1)(1 + X). $
    \item If $r\in (- 1/4,1/4)$, then $h_{u,+}(ir) \ll  \left( \frac 1{\sqrt X} + (1 + |u|)^{\frac 12}\right) \min \left\{ X^{k - 1}, \frac 1{\sqrt X}\right\}. $ 
    
\end{enumerate}

\end{lem}

\subsection{Oldforms and newforms}\label{subsec:newformoldform} 
In the application of Petersson's or Kuznetsov's formula, one often encounters an orthonormal basis of Maass forms $\{ u_j \}$ or of holomorphic modular forms $B_\ell(N)$. However, to apply GRH for Hecke $L$-functions in bounding sums over primes, it is necessary to express our basis in terms of newforms. The relevant theory was developed by Atkin and Lehner~\cite{AL}. For further background, we refer the reader to \S 14.7 of \cite{IK}, \S 2 of \cite{ILS}, and \S 5 of \cite{BM}. The information below is taken from \S 3.1 in \cite{BCL}.

We will state this theory for Maass forms only, although the theory applies also to holomorphic modular forms with slight changes in notation.  Let $S(\N)$ denote the space of all Maass forms of level $\N$ and $S^*(\M)$ denote the space orthogonal to all old forms of level $\M$.  By the work of Aktin and Lehner~\cite{AL}, $S^*(\M)$ has an orthonormal basis consisting of primitive Hecke eigenforms, which we denote by $H^{*}(\M)$.  Then, we have the orthogonal decomposition
$$
S(\N) = \bigoplus_{\N = \mL \M} \bigoplus_{f \in H^*(\M)} S(\mL; f),
$$
where $S(\mL; f)$ is the space spanned by $f|_l$ for $l|\mL$, where 
$f|_l(z) = f(lz).$ Let $f$ denote a newform of level $\M|\N$, normalized as a level $\N$ form, which means that the first coefficient satisfies 
$ |\rho_{ f}(1)|^2 =  (\N \tau_f)^{o(1)} / \N , $
where $\tau_f$ is the spectral parameter of $f$.

Blomer and Mili\'cevi\'c showed in \cite[Lemma 9]{BM} that there is an orthonormal basis for $S(\mL; f)$ of the form   
\begin{equation}\label{eqn:BMorthonormalform}
f^{(g)} = \sum_{d|g} \xi_g(d) f|_d
\end{equation}
for $g|\mL$, where $\xi_g(d)$ is defined in (5.6) of \cite{BM}. It satisfies $ \xi_1 (1) = 1$ and $$\xi_g(d) \ll g^\epsilon \bfrac{g}{d}^{\theta - 1/2} \ll d^{\epsilon} \bfrac{g}{d}^{\theta - 1/2 + \epsilon}.$$ Since $\theta < 1/2$, this implies the bound
\begin{equation}\label{eqn:xigdbdd2}
	\xi_g(d) \ll d^\epsilon.
\end{equation}
Also, \cite[Lemma 2]{BM} implies that the Fourier coefficients of $f^{(g)}$ satisfy
\begin{equation} \label{eqn:fouriercoeffbdd}
\sqrt{n} \rho_{f^{(g)}}(n) \ll (n\N)^\epsilon n^\theta (\N, n)^{1/2 - \theta} |\rho_{f}(1)| \ll \N^\epsilon n^{1/2+\epsilon} |\rho_{f}(1)|.
\end{equation}
This bound is somewhat crude, but will suffice for our purposes. Note that $f^{(g)}$ is an eigenfunction of the Hecke operator $T(n)$ for all $(n, \mathcal N)  =1$. Indeed, the $n$th Fourier coefficient of $f|_d$ is nonzero only if $d|n$. Since $g|\N$ in \eqref{eqn:BMorthonormalform}, it follows for $(n, \N) = 1$ that we may take only the $d=1$ term and deduce that
\begin{align*}
\sqrt{n}\rho_{f^{(g)}}(n) = \xi_g(1) \sqrt{n}\rho_f(n) = \xi_g(1) \rho_f(1) \lambda_f(n) = \rho_{f^{(g)}}(1) \lambda_f(n).
\end{align*}
This implies that $f^{(g)}$ is a Hecke eigenform with $
\lambda_{f^{(g)}}(n) = \lambda_f(n)$
for $(n, \N) = 1$. From (5.3) of \cite{BM}, we can write $n = n_0n'$, where $(n_0, \N) = 1$, $(n_0, n') = 1$ and 
\begin{equation} \label{eqn:coeffBM}
	\sqrt{n} \rho_{f^{(g)}}(n) = \lambda_f(n_0) \sqrt{n'} \rho_{f^{(g)}}(n').
\end{equation}

\subsubsection*{Remark}
For the rest of the paper, we will always take our orthonormal basis of cusp Maass forms $\{u_j\}$ and orthonormal basis of holomorphic forms $B_l(\N)$ to be these Hecke bases defined above.

\subsection{An explicit formula and some consequences of GRH}

 The following lemma is the first step in our proof of Theorem~\ref{thm:main}.

\begin{lem}\label{lem:Explicit} Let $\Phi$ be an even Schwartz function whose Fourier transform has compact support. We have
\es{\label{eqn:explicit} \sum_{j} \Phi\left( \frac{\gamma_{j, f}}{2\pi}\log Q\right) - \widehat{\Phi} (0) - \frac{\Phi (0)}{2}= \mathfrak{M}_{\Phi , f} (Q) + O \left(   \frac{ \log \log Q}{ \log Q} \right) , }
where 
\es{\label{def:Mi}  \mathfrak{M}_{\Phi , f} (Q) := -\frac{2}{\log Q}\sum_{\substack{p \\ p \nmid q}}\frac{\log p \lambda_f(p)}{\sqrt{p}} \widehat\Phi\left(\frac{\log p}{\log Q}\right) .} 
 
\end{lem}
The lemma holds by \cite[Lemmas 2.5 and 4.1]{BCL} and Lemma \ref{lem:boundforSymL-functions}.
The following lemma is \cite[Lemma 2.7]{BCL}.

{\lem\label{lem:CLee3.5} Assume GRH for $L(s,\chi)$ with a primitive Dirichlet character $\chi$ mod $q$ and for $L(s, f)$, where $f$ is a primitive holomorphic Hecke eigenform or a primitive Maass Hecke eigenform of level $q$ and weight $k$. Let $X>0$ be a real number, and let $\Psi$ be a smooth function that is compactly supported on $[0,X]$. Suppose that, for each positive integer $m$, there exists a constant $A_m$ depending only on $m$ such that
$$
|\Psi^{(m)}(x)| \leq \frac{A_m}{\min\{\log(X+3), X/x\}x^{m}}
$$
for all $x>0$. Write $z = \frac 12 +it$ with $t$ real, and let $N$ be a positive integer. If $\chi$ is a non-principal character, then
$$\sum_{(p, N) = 1}\frac{\chi(p)\log(p)\Psi(p)}{p^z}\ll A_3 \log^{1 + \epsilon}( X + 2) \log(q+|t| ) + \log N \max_{0 \leq x \leq X} |\Psi(x)|,$$
with absolute implied constant. Similarly, 
$$\sum_{(p, N) = 1}\frac{\lambda_f(p)\log(p)\Psi(p)}{p^z}\ll A_3 \log^{1 + \epsilon}( X + 2) \log (q+k+|t| ) + \log N \max_{0 \leq x \leq X} |\Psi(x)|,$$
with absolute implied constant.}

\noindent \textit{Remark}: If $c$ is a fixed constant and $\Upsilon$ is a smooth function compactly supported on $[0,c]$, then the function $\Psi(x)=\Upsilon(cx/X)$ satisfies the conditions in Lemma~\ref{lem:CLee3.5} since $X^{-m} \ll x^{-m} (x/X)$ for positive integers $m$. Also, if $\Upsilon$ is a smooth function compactly supported on $(-\infty,c]$, then the function $\Psi(x)=\Upsilon(\frac{c\log x}{\log X})$ satisfies the conditions in the lemma.

\begin{lem} \label{lem:sumprimeDiagonal}
  Let $F(t)$ be a smooth compactly supported function on $(-\sigma, \sigma)$, and $\log q \sim \log Q$. Then 

\est{\frac{1}{\log Q} \sum_{\substack{p \\ (p,q) = 1}} \frac{\log p}{p} F\left( \frac{\log p^2}{\log Q}\right) = \frac{1}{4} \int_{-\infty}^{\infty}  F(u) \> du + O\left( \frac{\log \log Q}{\log Q}\right), \ \ \ \ \textrm{and}}

\est{\frac{1}{(\log Q)^2} \sum_{\substack{p \\ (p,q) = 1}} \frac{\log^2 p}{p} F\left( \frac{\log p}{\log Q}\right) = \frac{1}{2} \int_{-\infty}^{\infty} |u| F(u) \> du + O\left( \frac{\log \log Q}{\log Q}\right).}

\end{lem}

This lemma follows immediately from the prime number theorem.  The following lemma is similar to \cite[Lemma 2.6]{BCL}.

\begin{lem} \label{lem:sym2bdd} Assume GRH for $L(s, \sym^2 f)$, where $f$ is a primitive Maass Hecke eigenform of level $q$ and spectral parameter $\tau_f$. For $\frac 12 < \sigma \leq \frac 54$ we have
 \begin{equation*}
-\frac{L'}{L}(\sigma+it , \sym^2 f) \ll \frac{\big( \log(q + \tau_f + |t|)\big)^{^{\frac 43 - \frac{2\sigma}{3}} }}{2\sigma - 1}.
\end{equation*}
\end{lem}

The proof of this lemma is essentially the same as \cite[Lemma 2.6]{BCL}, and we refer the reader there for the proof.  The next two lemmas help us bound sums over prime squares.  We begin with \cite[Lemma 2.11]{REUandMiller}.

\begin{lem} \label{lem:boundforSymL-functions}
Let $f$ be a primitive holomorphic cuspidal newform of weight $k$ and level $q$. Assume GRH for $L(s, sym^2f)$ and let $F$ be a smooth compactly supported function on $(-\sigma, \sigma)$.  Then for $q \sim Q$,
$$\frac{1}{(\log Q)^2}\sum_{(p, q) = 1} \frac{\lambda_f(p^2) \log^2 p}{p} F\left( \frac{\log p}{\log Q}\right) \ll \frac{\log \log Q}{\log Q}. $$
\end{lem}

Next we need a bound for a similar quantity where $f$ is a Maass form and the Ramanujan bound is not available unconditionally.
\begin{lem} \label{lem:boundforSymL-functions 2}
Let $f$ be a primitive Maass form with spectral parameter $\tau_f$ and level $q$. Assume GRH for $L(s, sym^2f)$ and let $F$ be a smooth compactly supported function on $(0, \infty)$ and $P \ge 1$.  Then for $\mathfrak C := q (1+|\tau_f|)^2$ and any positive integer $M $,
$$\sum_{(p, M) = 1} \frac{\lambda_f(p^2) \log^2 p}{p^{1+iu}} F\left( \frac{p}{P}\right) \ll  \log (\mathfrak C+ |u|) P^{-1/2 +\epsilon} + \log M,$$
{where the implied constant depends on $\epsilon.$}
\end{lem}
\begin{proof}
    By, for instance, the work of Kim and Sarnak \cite{KimS}, 
$\lambda_f(p^2) \ll p^{1/2}$, whence
$$\sum_{p|M} \frac{\lambda_f(p^2) \log^2 p}{p^{1+iu}} F\left( \frac{p}{P}\right) \ll \sum_{p|M} \log p \le \log M.
$$
Thus, 
    \est{
        \sum_{(p, M) = 1} \frac{\lambda_f(p^2) \log^2 p}{p^{1+iu}} F\left( \frac{p}{P}\right)
        &= \sum_{p} \frac{\lambda_f(p^2) \log^2 p}{p^{1+iu}} F\left( \frac{p}{P}\right) + O(\log M) \\
        &= \log P \sum_{p} \frac{\lambda_f(p^2) \log p}{p^{1+iu}} G\left( \frac{p}{P}\right) + O(\log M),
    }
    where $G(x) = F(x) \left(1+ \frac{\log x}{\log P}\right)$.  Since $G$ is in the Schwartz class, its Mellin transform satisfies  $\widetilde{G}(s) \ll \frac{1}{1+|s|^A}$ for any $A>0$.

    When $n = p^k$ with $k\ge 2$, we may bound $\lambda_f(n^2) \ll p^{\frac{7}{32}k}$ by \cite{KimS}. By Mellin inversion, we have
    \est{
        \sum_{p} \frac{\lambda_f(p^2) \log p}{p^{1+iu}} G\left( \frac{p}{P}\right)
        &= \sum_{n} \frac{\lambda_f(n^2) \Lambda(n)}{n^{1 + iu}} G\left( \frac{n}{P}\right) + O(1)\\
        &= \frac{1}{2\pi i} \int_{(c)} -\frac{L'}{L}(1+iu+s, \sym^2 f) P^s \widetilde{G}(s) ds + O(1)  
    }
    for $ c>0$. Under GRH, we can shift the contour to $c=-1/2 + \epsilon$ without residues.  By Lemma \ref{lem:sym2bdd} and the rapid decay of $\widetilde{G}(s)$, we have
\est{
    \sum_{p} \frac{\lambda_f(p^2) \log p}{p^{1+iu}} G\left( \frac{p}{P}\right) &\ll \frac{(\log (\mathfrak C+ |u|))^{4/3 - 2(1+c)/3}}{2c+1} P^c  \ll \log (\mathfrak C+ |u|) P^{-1/2 +\epsilon}.
}
\end{proof}

The following is \cite[Lemma 3.6]{CLee}. The minus sign in $ F (- i \mathcal U (1-\alpha)) \log Q $ is different from \cite{CLee}, since the Fourier transform is defined differently there.  

\begin{lemma}\label{lemma:prime sum}
    Assume RH and that $ F: \R \to \R $ is smooth and rapidly decreasing with $\widehat{F} $ compactly supported. 
    Let $\mathcal U = \frac{\log Q}{2\pi}$ and define
    $$ R (\alpha, F) := \sum_p \frac{\log p }{p^\alpha } \widehat{F}  \left( \frac{ \log p }{ \log Q} \right) - F (- i \mathcal U (1-\alpha)) \log Q . $$
    Then 
    $$ R( \alpha, F) = - \log Q \int_{ - \infty}^0 \widehat{F} ( w) Q^{(1-\alpha) w } dw +O\left( 1+ \left(|\alpha| + \tfrac{1}{\log Q}\right)\left( \tRe(\alpha) - \tfrac 12 \right)^{-3}\right) $$
    for $\tfrac 12 + \tfrac{10}{\log Q}  \leq \tRe (\alpha)  $, and
    $$  R( \alpha, F) = O( ( \log Q)^2 )  $$
    for $ | \tRe (\alpha) - \tfrac12 | \leq \tfrac{10}{\log Q} $.
                        \end{lemma}

\subsection{Combinatorial Sieve} \label{sec:combsieve}
We will apply the combinatorial sieve to express sums over distinct ordered elements as unrestricted sums.  This sieving also appeared in \cite{RudnickSarnak} and \cite{CLee}.

\begin{lem} \label{lem:cs} 
Let $ f_1 , \ldots, f_n $ be functions defined on the set of primes.     Then we have
\begin{align*}
    \sum_{p_1 , \ldots , p_n } f_1 ( p_1 ) \cdots f_n ( p_n)  &= \sum_{ \underline{G}\in \Pi_n }   \  \sumsharp_{ p_1 , \ldots , p_\nu   } f_{G_1} ( p_1 ) \cdots f_{G_\nu} ( p_\nu), \\
    \sumsharp_{ p_1 , \ldots , p_n   } f_1 ( p_1 ) \cdots f_\nu ( p_n )  & = \sum_{\underline{G} \in \Pi_{n}} \mu^* ( \underline G)   \sum_{ p_1 , \ldots , p_\nu  } f_{G_1} ( p_1 ) \cdots f_{G_\nu} ( p_\nu)  , 
\end{align*}
where $\sum^\sharp$ denotes sums over distinct primes, $ f_{G_j} := \prod_{ i \in G_j } f_i  $ for $ j \leq \nu$ and $ \mu^*  ( \underline G) = \prod_{G_j  \in \underline G}   (-1)^{|G_j|-1} ( |G_j |-1)! $ for $\underline G  \in \Pi_n $.

\end{lem}

\begin{proof}
Let
$$ R_{\underline G} := \sumsharp_{ p_1 , \ldots , p_\nu  } f_{G_1} ( p_1 ) \cdots f_{G_\nu} ( p_\nu), \qquad C_{\underline G} := \sum_{p_1 , \ldots , p_\nu } f_{G_1} ( p_1 ) \cdots f_{G_\nu} ( p_\nu) $$
for $ \underline G = \{ G_1 , \ldots, G_\nu \} \in \Pi_n $. 
By combinatorial sieving (e.g. see \cite[Lemma 2.1]{CLee}) we find that
\begin{equation}\label{eqn:combinatorial sieve}
   C_{ \underline O}   = \sum_{ \underline{G}\in \Pi_n } R_{ \underline G}, \qquad     R_{\underline O}   =  \sum_{\underline{G} \in \Pi_{n}} \mu^* ( \underline G) C_{\underline G},
   \end{equation}
where $\underline O = \{ \{1\} , \ldots , \{ n \} \} $. By rewriting the above identities in terms of prime sums, we obtain the identities in the lemma. 
\end{proof}

For clarity, we record the following bounds for prime sums.  These are essentially applications of the combinatorial sieve and GRH.

\begin{lemma}\label{lem:primesumremovedistinct}
Assume GRH. Let $\chi$ be a non-principal character modulo $M$.  Fix $  P_1,\ldots ,P_\kappa \ge 1$ with $\log P_i \ll \log Q$, $N_1 , \ldots , N_\kappa \leq N $, real $v_1,\ldots ,v_\kappa $ and $ \kappa' \leq \kappa $. Let $V$ be a smooth function compactly supported on $[1/2,3]$.  Then
\es{\label{eqn:primecharacterbdd}
 &\sumsharp_{\substack{ p_1,...,p_\kappa  \\  (  p_r  , N_r) = 1 \textup{ for } r \leq \kappa   } } \prod_{r = 1}^{\kappa' }  \frac{ \overline{\chi(p_r)} \log p_r}{p_r^{1/2-it}} V\bfrac{p_r}{P_r}  \e{v_r\frac{p_r}{P_r}} \prod_{r = \kappa' +1}^{\kappa } \frac{ \chi(p_r) \log p_r}{p_r^{1/2+it}} V\bfrac{p_r}{P_r}  \e{v_r\frac{p_r}{P_r}}  \\ 
 &\ll (MNQ(1+|t|))^\epsilon Y(\vb)^3 
}
for any $\epsilon>0$, where $ \vb = ( v_1 , \ldots  , v_\kappa ) $ and 
\begin{equation}\label{def:Yv}
Y(\vb) := \prod_{j=1}^\kappa ( 1+|v_j|) .
\end{equation} 
    
\end{lemma}

\begin{proof}
    
By Lemma \ref{lem:cs}, the sum in \eqref{eqn:primecharacterbdd} is a linear combination of $ \prod_{j=1}^\nu \mathcal{P}_1(G_j)$ over $ \underline G \in \Pi_\kappa $, where $\mathcal P_1 (   G_j )$ is defined by
$$
       \sum_{\substack{p \\ ( p  , N_r) = 1 \\
   \textup{for } r \in G_j}} \left( \prod_{\substack{ r \in G_j \\  r \le \kappa' }} \frac{ \overline{\chi( p )} \log p }{ p^{1/2-it}} V\bfrac{ p}{P_r}  \e{v_r\frac{p}{P_r}} \right)\left( \prod_{\substack{r \in G_j \\    r > \kappa' }} \frac{ \chi(p) \log p}{p^{1/2+it}} V\bfrac{ p}{P_r}  \e{v_r\frac{p}{P_r}} \right).
$$
If $|G_j|>1$, then we have
\begin{align}
  \mathcal P_1 (G_j)  \ll \sum_{p  \leq Q^A } \frac{(\log p )^m}{p^{m/2}} \ll 
    \begin{cases}
        (\log Q)^{2} &\textup{ if } |G_j| =2 ,\\
        1 &\textup{ if } |G_j| >2.
    \end{cases}
\end{align}
If $ G_j  = \{r\}$ for some $ \kappa' < r \leq \kappa $, then
$$ \mathcal P_1 (G_j) = \sum_{\substack{p \\ (p , N_r ) = 1}}  \frac{ \chi(p ) \log p }{p^{1/2+it}} V\bfrac{ p}{P_r}  \e{v_r\frac{ p}{P_r}} \ll (\log Q)^{1+\epsilon} \log (M + |t|) (1+|v_r|)^3 + \log N
$$
by Lemma \ref{lem:CLee3.5} with $\Psi(x) = V\bfrac{x}{P_r}  \e{v_r\frac{ x}{P_r}}$, so that $\Psi^{(3)}(x) \ll \frac{1+|v_r|^3}{1+|x|^3}$.  
We can find the same bound for $G_j = \{r \}$ with $ r \leq \kappa'$.
Combining these bounds we obtain
$$\prod_{j=1}^\nu \mathcal{P}_1(G_j) \ll (MNQ(1+|t|))^\epsilon Y(\vb)^3      
$$
for any $\epsilon >0 $ and each $ \underline G \in \Pi_\kappa$, which proves the lemma.
\end{proof}

\begin{lem}\label{lem:heckebasisfouriersumbdd}
Assume GRH.	Let $u$ be an element of the Atkin-Lehner basis of level $\N$ so $u = f^{(g)}$ for some primitive Hecke form $f$ of level $\M|\N$, and some $g|\frac{\N}{\M}$.  We also set
    \begin{eqnarray*}
        \mathcal C = 
        \begin{cases}
            k &\textup{ if  $f$ is a holomorphic form of weight $k$,}\\
            1+|\tau_f| &\textup{ if  $f$ is a Maass form with spectral parameter $\tau_f$.}
        \end{cases}
    \end{eqnarray*}
   For $1\le r\le \kappa$, let $\Psi_r$ be smooth functions supported in $(a,b)$, where $0 < a < b$, and  let $X_r >0$ and $t_r$ be real numbers.  Moreover, we assume that
$    \max( \log \alpha, \log X_1,...,\log X_\kappa) \ll \log Q$. Then
	\est{\mathscr B_{u} &:= \sumsharp_{\substack{p_1,...,p_{\kappa} \\  ( \mathfrak p (\kappa) , \alpha) = 1   } } 
   \rho_u( \mathfrak p (\kappa)  )\prod_{r = 1}^{\kappa} \left[\frac{ \log p_r}{p_r^{ i t_r}} \Psi_r\left(\frac{p_r}{X_r}\right) \right]  \ll_{\Psi} |\rho_f(1)| (\mathcal C \N Q) ^{\epsilon} \prod_{r=1}^\kappa \log(2+|t_r| ).}
\end{lem}

 \begin{proof} 
We first split the sum to distinguish primes dividing $  \N$ or not. Then $   \mathscr B_{u} = \sum_{ R  \sqcup R'  = [\kappa ] } \mathscr B_{u,R}$, where
$$  \mathscr B_{u,R } =     \sumsharp_{\substack{p_1,...,p_\kappa \\ (\mathfrak p (\kappa) , \alpha) = 1   \\ ( \mathfrak p (R)  , \N) = 1 , ~\mathfrak p (R')  | \N     
  }}\sqrt{ \mathfrak p (\kappa)  }\rho_u( \mathfrak p (\kappa) )\prod_{r = 1}^{\kappa} \left[\frac{ \log p_r}{p_r^{1/2+it_r }}\Psi_r\left(\frac{p_r}{X_r}\right) \right] .$$
Without loss of generality, we only consider the case $R=  [ \gamma] \subset [ \kappa ]  $. By \eqref{eqn:coeffBM}
\est{\sqrt{p_1 \cdots p_{\kappa}}\rho_u(p_1 \cdots p_{\kappa})  = \lambda_f(p_{ 1}) \cdots \lambda_f(p_{\gamma}) \sqrt{p_{\gamma + 1} \cdots p_{\kappa}} \, \rho_u(p_{\gamma + 1} \cdots p_{\kappa}),}
and by  \eqref{eqn:fouriercoeffbdd} we have  
\est{
 \rho_u( p_{\gamma+ 1} \cdots p_{\kappa } ) \ll |\rho_f(1)|   (\N p_{\gamma+ 1} \cdots p_{\kappa } ) ^\epsilon \ll |\rho_f(1)|  (\N Q) ^\epsilon .}  Hence, by \eqref{eqn:coeffBM} we find that
 \begin{equation}\label{eqn:BuR}
  \mathscr B_{u,R } \ll      |\rho_f(1)|  (\N Q)^\epsilon \left|  \ \, \sumsharp_{\substack{p_{1},.., p_{\gamma} \\ ( \mathfrak p (\gamma), \alpha \N) = 1 }} \prod_{r =   1}^{\gamma} \left[\frac{\lambda_f(p_r) \log p_r}{p_r^{1/2+it_r}}\Psi_{r}\left(\frac{p_r}{X_r}\right)\right]    \right|.  
\end{equation}
 
By Lemma \ref{lem:cs} the sum in \eqref{eqn:BuR} is a linear combination of $ \prod_{j=1}^\nu \mathcal{P}_2 (G_j) $ over $ \underline G = \{ G_1 , \ldots , G_\nu\} \in \Pi_R $, where
$$ \mathcal{P}_2 (G_j ) :=  \sum_{ \substack{ p \\ (p, \alpha \N )=1}}  \prod_{r \in G_j  } \left[\frac{\lambda_f(p ) \log p }{p^{1/2+it_r}}\Psi_{r}\left(\frac{p }{X_r}\right)\right]  . $$
Hence, by \eqref{eqn:BuR}  it is enough to prove that
\begin{equation}\label{eqn:bound P2}
\prod_{j=1}^\nu \mathcal{P}_2 (G_j)\ll (\mathcal C \N Q)^\epsilon \log(2+|t_1| ) \cdots \log(2+|t_{\gamma}| )
\end{equation}
for each $ \underline G \in \Pi_R $.

For $G_j = \{r\}$, we have by Lemma \ref{lem:CLee3.5} that
$$\mathcal{P}_2 (G_j )  \ll \log^{1+\epsilon}(X_r+2) \log (\N+\mathcal C +|t_r|) + \log \N 
\ll (\N \mathcal C Q )^\epsilon \log (2 +|t_r|),
$$
where in applying Lemma \ref{lem:CLee3.5}, we have used that
	\begin{equation*}
		\frac{d^m}{dx^m} \Psi_r\left( \frac{x}{X}\right) = \frac{1}{X^m}\Psi_r^{(m)} \left( \frac{x}{X}\right) \ll \frac{1}{Xx^{m-1}} \max_{a<y<b} |\Psi_r^{(m)}(y)|
	\end{equation*}
	for each positive integer $m$.
For $G_j = \{k, l\}$ with $X_k \le X_l$, we have 
\begin{align*}
   \mathcal{P}_2 (G_j )
    &= \sum_{\substack{p\\ (p, M) = 1}}\frac{\lambda_f(p)^2 \log^2 p}{p^{1+i(t_k + t_l)}} \Psi_{k}\left(\frac{p}{X_k}\right) \Psi_{l}\left(\frac{p}{X_l} \right) = \sum_{\substack{p\\ (p, M) = 1}} \frac{(\lambda_f(p^2) +1)\log^2 p}{p^{1+i(t_k + t_l)}} F\left(\frac{p}{X_k}\right), 
\end{align*}
where $F(x) := \Psi_{k}(x) \Psi_{l}\left(x \frac{X_k}{X_l} \right)$. 
By Lemma \ref{lem:boundforSymL-functions 2}
\begin{align*}
   \mathcal{P}_2 (G_j )  &\ll \log(\mathcal C \N + |t_k| + |t_l|) X_k^{-1/2+\epsilon} + \log M + \sum_{p\le X_k} \frac{\log^2 p}{p} \ll (\mathcal C \N Q)^\epsilon, 
\end{align*}when $\log M  = \log (\alpha \N p_1...p_\gamma) \ll (Q\N)^\epsilon \log (2+|t_l|) \log (2+|t_k|)$, since $\log p_i \ll Q^\epsilon$ and $\log \alpha \ll \log Q \ll Q^\epsilon$, and $\sum_{p\le X_k} \frac{\log^2 p}{p} \ll X_k^\epsilon \ll Q^\epsilon$.
 Finally, when $  |G_j|\geq 3$, we have that 
$$\mathcal{P}_2 (G_j ) \ll \sum_{p} \left( \frac{p^{\frac{7}{64}  } \log  p}{p^{1 /2}} \right)^{|G_j|}\ll 1.$$
One can deduce \eqref{eqn:bound P2} by combining the above bounds of $ \mathcal P_2 (G_j)$ and concludes the proof of the lemma.
 \end{proof}

\subsection{Other lemmas}

We begin with stating the Hecke relations. 

\begin{lem} \label{lem:multoflambdaf}Let $f$ be a newform of weight $k$ and level $q$ in  $  \mathcal H_k(q) $. Then 
$$ \lambda_f(m) \lambda_f(n) = \sum_{\substack{d | (m, n)  \\ (d, q) = 1}} \lambda_f\left( \frac{mn}{d^2}\right).   $$
If $(p, q) = 1,$ then 
\est{ \lambda_f(p)^{2m} =  \sum_{r = 0}^m \left( {2m  \choose m-r} - {2m \choose m-r-1} \right) \lambda_f(p^{2r})}
and 
\est{  \lambda_f(p)^{2m + 1} =  \sum_{r = 0}^m \left( {2m + 1  \choose m-r} - {2m + 1 \choose m-r-1} \right) \lambda_f(p^{2r + 1}).}
\end{lem}
The above formulas can be found in  \cite{Guy} and \cite{HM}.  
The next two lemmas collect some well known properties and formulas for the $J$-Bessel function.
{\lem\label{jbessel} Let $J_{k-1}$ be the $J$-Bessel function of order $k-1$. We have
\est{ J_{k-1}(2\pi x) =\frac{1}{2\pi \sqrt{x}}\left(W_k(2\pi x) \e{x-\frac{k}{4}+\frac{1}{8}} + \overline{W}_k(2\pi x) \e{-x+\frac{k}{4}-\frac{1}{8} }\right) }
for $ x >0$, where 
$$W_k(x) = \frac{1}{\Gamma(k- \frac12 )} \int_0^\infty e^{-u} u^{k - \tfrac 32} \left( 1 + \frac{iu}{4\pi x}\right)^{k - \tfrac 32} \> du.$$ 
Note that $W_k^{(j)}(x)   \ll_{j,k} x^{-j} $ as $x \to \infty$. 
Moreover, 
\begin{equation*}  J_{k-1}(2x) =\sum_{\ell = 0}^{\infty} (-1)^{\ell} \frac{x^{2\ell+k-1}}{\ell! (\ell+k-1)!}
\end{equation*}
and $$J_{k-1}(x)\ll \textup{min} \{x^{-1/2}, x^{k-1}\}.$$
}

The proof of the first three claims of Lemma~\ref{jbessel} can be found in \cite[p. 206]{Watt}, and the statement of the last claim is modified from Equation 16 of Table 17.43 in \cite{GR}.  

\begin{lem} \label{lem:MellinofprodofJbessel} For $-\tRe(\mu + \nu) < \tRe(s) < 1$
\est{ \int_0^{\infty} J_\mu(x) J_{\nu}(x) x^{s - 1} \> dx = 2^{s-1} \Gamma(1-s) \mathcal G_{\mu, \nu}(s) ,}
where 
\est{\mathcal G_{\mu, \nu}(s) = \frac{ \Gamma \left( \frac{\mu}{2} + \frac{\nu}{2}  + \frac{s}{2}\right) }{\Gamma \left( \frac{\mu}{2} - \frac{\nu}{2}  - \frac{s}{2} + 1\right)\Gamma \left( \frac{\nu}{2} - \frac{\mu}{2}  - \frac{s}{2} + 1\right)\Gamma \left( \frac{\mu}{2} + \frac{\nu}{2}  - \frac{s}{2} + 1\right)}.}
 Moreover, let $\delta, \sigma$ and $\nu$ be fixed real numbers. Then for $\mu = \delta + it$ and $s = \sigma + iy$, 
 \est{ \Gamma(1-s) \mathcal G_{\mu, \nu}(s) \ll (1 + |y|)^{\sigma - \frac 52 } (1 + |t|)^{2\sigma - 2}e^{\frac{\pi}{2}|t|}.}

\vskip 0.1in
For a positive even number $k$, 
we have 
\es{\label{eqn:diffJbesselat0}\mathcal G_{ \mu, k-1}(0) - \mathcal G_{-\mu, k-1}(0) = 0,}
 and for non-integer $v$,
 \es{ \label{eqn:diffJbesselatvplus1}  
 \mathcal G_{ v, k-1}(v + 1) - \mathcal G_{- v, k-1}(v + 1) = -(-1)^{k/2}\frac{\Gamma\left( v + \tfrac k2 \right) \sin (\pi  v  )}{\pi \Gamma \left( -v + \frac{k}{2}   \right) } .}
\end{lem}
\begin{proof}
The first equation in the lemma is from Equation (33) on p.331 in \cite{Bateman} or Equation (2) on p.403 in \cite{Watt}. The bound comes from the Stirling formula of the Gamma function.

Now we will prove \eqref{eqn:diffJbesselat0} - \eqref{eqn:diffJbesselatvplus1}. From the definition of $\mathcal G_{\mu, k-1}(s)$, we derive
\es{\label{eqn:diffGamma1}& \mathcal G_{\mu, k-1}(s) - \mathcal G_{-\mu, k-1}(s)\\
&=   \frac{ \Gamma \left( \frac{\mu}{2} + \frac{k - 1}{2} + \frac s2  \right)\Gamma \left( -\frac{\mu}{2} - \frac{k - 1}{2}  - \frac s2 + 1\right) - \Gamma \left( -\frac{\mu}{2} + \frac{k - 1}{2} + \frac s2  \right) \Gamma \left( \frac{\mu}{2} - \frac{k - 1}{2}  - \frac s2 + 1\right)}{\Gamma \left( \frac{\mu}{2} - \frac{k-1}{2}  - \frac s2 + 1\right) \Gamma \left( -\frac{\mu}{2} - \frac{k - 1}{2}  - \frac s2 + 1\right)\Gamma \left( \frac{k - 1}{2} - \frac{\mu}{2}  - \frac s2 + 1 \right)\Gamma \left(\frac{\mu}{2} + \frac{k - 1}{2}  - \frac s2 + 1\right)}. }
Using the identity 
\es{\label{Gammaidbw1-zandz} \Gamma(1-z) \Gamma(z) = \frac{\pi}{\sin(\pi z)}}
the numerator of the expression above becomes
\es{\label{eqn:diffGammanumer} &\frac{\pi}{ \sin\left( \pi \left( \frac{\mu}{2} + \frac s2 + \frac{k - 1}{2}\right)\right)} - \frac{\pi}{ \sin\left( \pi \left( -\frac{\mu}{2} + \frac s2 + \frac{k - 1}{2}\right)\right)} \\
&= \pi (-1)^{k/2 + 1} \left[\frac{1}{\cos\left(\pi\left( \frac{\mu}{2} + \frac s2 \right) \right)} -  \frac{1}{\cos\left(\pi\left( -\frac{\mu}{2} + \frac s2 \right) \right)}\right], }
where we have used that $k$ is even. Note that the expression vanishes when $s = 0$, which proves \eqref{eqn:diffJbesselat0}.

To prove \eqref{eqn:diffJbesselatvplus1}, we first use \eqref{eqn:diffGamma1} and the fact that $\frac{1}{\Gamma(1 - k/2)} = 0$ when $k$ is an even natural number to obtain that
\est{&\mathcal G_{ v, k-1}(v + 1) - \mathcal G_{-v, k-1}(v + 1) 
= - \frac{ 1 }{\Gamma \left( -v - \frac{k}{2}  + 1\right)\Gamma \left(-v + \frac{k }{2}   \right)}. }
 Next, using \eqref{Gammaidbw1-zandz} for $z =  - v - \frac k2 + 1$, where $v $ is not an integer and $k$ is an even integer, we have 
\est{ - \frac{ 1 }{\Gamma \left( -v - \frac{k}{2}  + 1\right)\Gamma \left(-v + \frac{k }{2}   \right)} 
&= -(-1)^{k/2} \frac{\Gamma\left( v + \tfrac k2 \right) \sin (\pi  v  )}{\pi \Gamma \left( -v + \frac{k}{2}   \right) }. }
\end{proof}
\vskip 0.1in
\begin{lem} \label{lem:asympforNQ} Let $\Psi(x)$ be a smooth function compactly supported in $(a, b),$ where $a$ and $b$ are fixed positive constants with $a < b $. Then we have 
$$ N_0 (Q) = Q \widetilde{\Psi}(1) T(1) + O\left( Q^{\epsilon}\right) $$
for any $ \epsilon>0$, where 
$$T(s) := \sum_{L_1, L_2} \frac{\mu(L_1L_2) \zeta_{L_1}(2)}{L_1^{1 + 2s} L_2^{1 + s}}   \sum_{ \ell_1  | L_1} \frac{\mu( \ell_1 )}{ \ell_1 ^{2 + s}} \sum_{ \ell_2  | L_2} \frac{\mu( \ell_2 )}{ \ell_2 ^{s}} $$
and $ \zeta_{L_1} (2)$ is defined in \eqref{def:zetaNs}.
In particular, $T(s)$ is absolutely convergent for $\Rep(s)>0$. 
\end{lem}

\begin{proof}
By Lemma \ref{lem:PeterssonNg} we have
\begin{equation*}
N_0 (Q) = \sum_q \Psi\bfrac{q}{Q} \sum_{\substack{q=L_1L_2d \\ L_1|d,  \, \,  (L_2, d) = 1}} \frac{\mu(L_1L_2)}{L_1L_2}  \prod_{\substack{p|L_1 \\ p^2 \nmid d}}  \left( 1-\frac{1}{p^2} \right)^{-1} \sum_{\elltwo|L_2^{\infty} }\frac{\Delta_d(1,\elltwo^2)}{\elltwo}.
\end{equation*}
From Lemma \ref{lem:petertruncate}, 
\est{N_0 (Q) &= \sum_q \Psi\bfrac{q}{Q} \sum_{\substack{q=L_1L_2d \\ L_1|d,  \, \,  (L_2, d) = 1}} \frac{\mu(L_1L_2)}{L_1L_2}  \prod_{\substack{p|L_1 \\ p^2 \nmid d}}  \left( 1-\frac{1}{p^2} \right)^{-1}   +O( Q^\epsilon ).}
The main term is 
\est{N_1 (Q) =  \sum_{\substack{L_1, L_2, d \\ L_1|d, \, \,  (L_2, d) = 1}} \frac{\mu(L_1L_2)}{L_1L_2}  \prod_{\substack{p|L_1 \\ p^2 \nmid d}}  \left( 1-\frac{1}{p^2} \right)^{-1}  \Psi\bfrac{L_1L_2d}{Q}.}

We will do changes of variables similar to \S 6 in \cite{BCL}.
Since $L_1|d$, we write $d = L_1 m$ and have
\begin{equation}\label{eqn:pL1d product to sum}
\prod_{\substack{p|L_1 \\ p^2 \nmid d}}  \left( 1-\frac{1}{p^2} \right)^{-1}
  = \prod_{ p|L_1  }  \left( 1-\frac{1}{p^2} \right)^{-1}  \prod_{\substack{p|L_1 \\ p  | m}}  \left( 1-\frac{1}{p^2} \right)  = \zeta_{L_1}(2)  \sum_{ \ell_1  |(L_1,m)}\frac{\mu( \ell_1  )}{ \ell_1 ^2}.
\end{equation}
Using this and substituting  $q=L_1L_2d = L_1^2 L_2m$ and $m =  \ell_1  n$ in the above expression for $N_1 (Q)$ gives 
\begin{align*}
N_1 (Q)
= & \sum_{L_1,L_2  } \frac{\mu(L_1L_2)\zeta_{L_1}(2) }{L_1L_2}    \sum_{ \ell_1  |L_1}\frac{\mu( \ell_1 )}{ \ell_1 ^2}      \sum_{\substack{n  \\ (n,L_2)=1}}  \Psi\left(\frac{L_1^2L_2 \ell_1 n}{Q}\right). 
\end{align*}
Next, we use M\"{o}bius inversion to detect the condition $(n,L_2)=1$ and deduce that
\est{ 
N_1(Q)
&=  \sum_{ L_1,L_2 } \frac{\mu(L_1L_2)\zeta_{L_1}(2)}{L_1L_2}   \sum_{ \ell_1 |L_1}\frac{\mu( \ell_1 )}{ \ell_1 ^2} \sum_{ \ell_2 |L_2}\mu( \ell_2 )     \sum_{n}   \Psi\left(\frac{L_1^2L_2 \ell_1  \ell_2 n}{Q}\right).  
}
By the inverse Mellin transform 
\begin{equation} \label{inverseMellinPsi} \Psi(x) = \frac{1}{2\pi i} \int_{(\sigma)} \widetilde{\Psi}(s) x^{-s} \> ds ,
\end{equation}
we have 
$$  N_1 (Q) = \frac{1}{2\pi i} \int_{(\sigma)} \widetilde{\Psi}(s) Q^{s} \zeta(s) T(s)\> ds$$
for $\sigma > 1$. 
The Dirichlet series $T(s)$ is absolutely convergent for $ \tRe(s)>0$.
Moving the contour integration to the line $\tRe(s) = \epsilon >0$, we pick up a simple pole at $s = 1$, and the residue is $Q\widetilde \Psi(1)T(1)$. By the fast decay of $\widetilde \Psi (s)$ along the vertical line, the remaining integral is  $O(Q^{\epsilon})$.
\end{proof}

Below is \cite[Lemma 2.7]{Li} from the third author's paper.
\begin{lem}\label{lem:testfunc}
Let $c_0$ and $c_1$ be any fixed positive real numbers.  Then there exists a smooth non-negative even Schwartz class function $F$ such that $F(x) \ge 1$ for all $x \in [-c_1, c_1]$ and $\widehat{F}(x)$ is even and compactly supported on $[-c_0, c_0]$.  
\end{lem}

Next, we have a standard result for the Fourier transform. We quote it from the beginning of \S 3 in \cite{CLee}.

\begin{lem} \label{lem:FourierResult}
Let $F$ be a Schwartz class function on $\mathbb{R}$ with $ \mathrm{supp} ~\widehat{F} \subset [ -\kappa_0 , \kappa_0 ] $. Then $F$ has an extension to the complex plane that is entire. 
Moreover, for any integer $A_1 \geq 0$, 
\begin{equation}\label{eqn: F bound}
  F( v+ iy  ) = \int_{\mathbb{R}} \widehat{F}(w) e^{-2 \pi wy}  e^{2 \pi i w v }     dw \ll_{A_1}  \new{\min\left\{ 1 ,  \frac{1+|y|^{A_1}}{ |v|^{A_1}} \right\}  e^{2\pi \kappa_0|y|}}
\end{equation}
for $v , y\in \R  $.
\end{lem}

\begin{proof}
    The first part of the lemma is Theorem 3.3 in \cite[p.122]{SS}.  The second assertion contained in \eqref{eqn: F bound} follows from integration by parts many times.
\end{proof}

\section{ Setup  of the Proof of Theorem \ref{thm:main}}\label{sec:outline of proof main theorem}

We begin by applying Lemma \ref{lem:Explicit} to \eqref{def:nthcenterNT}. Our first task is to show that the contribution from the error term $O\left( \frac{\log \log Q}{\log Q}\right)$ in \eqref{eqn:explicit} is negligible.
\begin{prop}\label{prop:SnQ sum MiQ}
Assume GRH. Let $\Phi_i$ be an even Schwartz function with $\widehat \Phi_i$ compactly supported in $(-\sigma_i, \sigma_i),$ where $\sum_{i = 1}^n \sigma_i < 4$.   Define
\begin{equation} \label{def:SnQ} 
S_n(Q) = \sum_q \Psi\bfrac{q}{Q}  \sumh_{f \in \mathcal{H}_k(q)} \prod_{i = 1}^n \mathfrak M_{\Phi_i,f} (Q),
\end{equation}
where $\mathfrak M_{\Phi_i,f} (Q)$ is defined in \eqref{def:Mi}. 
Then we have
$$ \mathscr{L}_n(Q) = \frac{S_n (Q)}{N_0 (Q)}  + O \left( \frac{ \log \log Q}{ \log Q} \right)   $$
as $Q \to \infty$.
\end{prop}

To compute $S_n (Q)$, we write it as a linear combination of sums over distinct primes by Lemma \ref{lem:cs} such as 
\est{S_n(Q) &=  \sum_{\underline A \in \Pi_n} S_n(Q; \underline A), }
where $\underline A = \{ A_1 , \ldots , A_\nu\} $, $a_j := |A_j |$,
\est{ \label{def:SnQai} S_n(Q; \underline A ) :=  \sum_q \Psi\left( \frac qQ \right)\sumh_{f \in \mathcal H_k(q)} \ \sumsharp_{\substack{p_1, p_2,..., p_{\nu} \\ p_i \nmid q }} \prod_{j = 1}^\nu  \left(-\frac{2}{\log Q} \frac{\log p_j \lambda_f(p_j)}{\sqrt{p_j}} \right)^{a_j} H_{A_j} \left( \frac{\log p_j}{\log Q}\right) }
and 
 $$ H_{A_j}(x) = \prod_{k \in  A_j } \widehat\Phi_k\left(x\right) . $$
Next, we show that the main contribution arises from set partitions where $a_j \leq 2$ for all $j$, and the number of sets $A_j$ with $a_j = 1$ is not $1$. 
\begin{prop}\label{prop: when SnQA is small}
 Assume GRH. Let $\underline A = \{ A_1 , \ldots , A_\nu\} \in \Pi_n $ and $a_j = |A_j |$ for $j \leq \nu$.  If $a_j \geq 3$ for some $j$, then 
\begin{equation}\label{eqn:ajatleast3} S_n(Q; \underline A) = O\left( \frac{Q}{(\log Q)^3 }\right). 
\end{equation}
If exactly one of the $a_j$ equals $1$ and all others equal $2$, then we have
\begin{equation} \label{eqn:onlyoneajis1}
 S_{n}(Q; \underline{A}) = O\left(  \frac{ Q\log \log Q}{\log Q} \right).
 \end{equation}
\end{prop}
 Hence, we are left to calculate
 \begin{equation}\label{eqn: SnQ SnQKh}
 S_n(Q) = \frac{1}{N_0 (Q)}\sum_{\substack{K \sqcup K_0 = [n] \\   |K| \neq 1 }} \sum_{\underline G \in \Pi_{K_0 , 2 } }   S_n( Q ; \underline G \sqcup  \pi_{K,1}     )  + O\left( \frac{\log \log Q}{\log Q}\right),
 \end{equation}
where $\Pi_{K_0, 2}$ and $ \pi_{K,1}$ are defined in Definition \ref{def:set partition}. Then Theorem \ref{thm:main} follows from Proposition \ref{prop:SnQ sum MiQ}, \eqref{eqn: SnQ SnQKh}, the following proposition and Theorem \ref{thm:Cn}. 
\begin{prop} \label{prop:SnQCn}  
 Assume GRH. Let $C_0 (n)$ and $C_2 (n)$ be defined as in Theorem \ref{thm:Cn}.  We have
\begin{equation} \label{prop:eqn C0(n)case}
\lim_{Q \to \infty} \frac{1}{N_0 (Q)}\sum_{  \underline G \in \Pi_{n,2} } S_n(Q; \underline G  ) = C_0 (n) ,
\end{equation}
and
\begin{equation} \label{prop:eqn C2(n)case}
\lim_{Q \to \infty} \frac{1}{N_0 (Q)}\sum_{\substack{K \sqcup K_0 = [n] \\   |K| \geq 2 }} \sum_{\underline G \in \Pi_{K_0 , 2 } }   S_n( Q ; \underline G \sqcup \pi_{K,1}    )   = C_2 (n)  .
\end{equation}
\end{prop}

We will prove Propositions \ref{prop:SnQ sum MiQ} - \ref{prop:SnQCn} in \S \ref{section:proof props}.

\section{Proof of Propositions \ref{prop:SnQ sum MiQ} - \ref{prop:SnQCn}} \label{section:proof props}

\subsection{Proof of Proposition \ref{prop:SnQ sum MiQ}} \label{section:proof prop 1}

By \eqref{def:nthcenterNT} and Lemmas \ref{lem:Explicit} and \ref{lem:asympforNQ}, it is enough to show that
\begin{equation}\label{eqn:MifQ absolute bound}
  \sum_q \Psi\bfrac{q}{Q} \sumh_{f \in \mathcal H_{k}(q)} \prod_{i = 1}^m  | \mathfrak M_{\Phi_i, f} (Q)| \ll Q 
  \end{equation}
for all $ 1 \leq m \leq n $.
By \eqref{eqn:explicit}, it is equivalent to prove that 
\begin{equation}\label{eqn:equiv to prop3.1} 
  \sum_q \Psi\bfrac{q}{Q} \sumh_{f \in \mathcal H_{k}(q)} \prod_{i = 1}^m \left|  \sum_{j} \Phi_i\left( \frac{\gamma_{j, f}}{2\pi}\log Q\right) \right|  \ll Q   
  \end{equation}
  for all $ 1 \leq m \leq n $. For all $ i \leq n$, we have
\begin{align*}
    \sum_{j} \left| \Phi_i\left( \frac{\gamma_{j, f}}{2\pi}\log Q\right) \right|  & \ll \sum_{ \ell=1}^\infty 
\# \left\{ \frac{2\pi (\ell-1) }{\log Q} \leq |\gamma_{j,f}| < \frac{2\pi \ell }{\log Q} \right\} \frac{1}{\ell^{10n}} \\
& \ll \sum_{ \ell=1}^\infty   \frac{1}{\ell^{10n}} \sum_j  H \left( \frac{\gamma_{j,f} \log Q}{ 2 \pi \ell} \right) ,
\end{align*}
where $H(x)$ is an even Schwartz function such that $H(x) \geq 1 $ for $ |x| \leq 1 $ and $ H(x) \geq 0 $ for all $ x \in \mathbb{R}$. 
We also assume that $\widehat H $ is even and compactly supported on $[ - \tfrac1{2n} , \tfrac1{2n} ] $. Such function $H$ exists by Lemma \ref{lem:testfunc}. 
Hence, \eqref{eqn:equiv to prop3.1} is justified if we prove that 
\begin{equation}\label{eqn:equiv2 to prop3.1} \begin{split}
   \sum_q & \Psi\bfrac{q}{Q} \sumh_{f \in \mathcal H_{k}(q)}   \left( \sum_{ \ell=1}^\infty   \frac{1}{\ell^{10n}} \sum_j  H \left( \frac{\gamma_{j,f} \log Q}{ 2 \pi \ell} \right) \right)^m  \\
 & \ll  \sum_{ \ell=1}^\infty   \frac{1}{\ell^{10n}} \sum_q \Psi\bfrac{q}{Q} \sumh_{f \in \mathcal H_{k}(q)}   \left( \sum_j  H \left( \frac{\gamma_{j,f} \log Q}{ 2 \pi \ell} \right) \right)^m \ll Q  
 \end{split} \end{equation}
for every $ m \leq n $, where the first inequality holds by H\"older's inequality.

By \eqref{eqn:explicit} with $ H$ in place of $\Phi$, it is enough to show that
\begin{equation}\label{eqn:equiv3 to prop3.1}
  \sum_{ \ell=1}^\infty   \frac{1}{\ell^{10n}} \sum_q \Psi\bfrac{q}{Q} \sumh_{f \in \mathcal H_{k}(q)}    \left(   \sum_{\substack{p \\ p \nmid q}}\frac{\log p \lambda_f(p)}{\sqrt{p}} \widehat H\left(\frac{\ell \log p}{\log Q}\right) \right)^{m} \ll Q (\log Q)^{ m}   
  \end{equation}
  for all $ m \leq n $.
We will prove  
\begin{equation}\label{eqn:equiv4 to prop3.1}
  \sum_{ \ell=1}^\infty   \frac{1}{\ell^{10n}} \sum_q \Psi\bfrac{q}{Q} \sumh_{f \in \mathcal H_{k}(q)}    \bigg|   \sum_{\substack{p \\ p \nmid q}}\frac{\log p \lambda_f(p)}{\sqrt{p}} \widehat H\left(\frac{\ell \log p}{\log Q}\right) \bigg|^{2m} \ll Q (\log Q)^{ 2m}   
  \end{equation}
for all $ m \leq n $. It then follows from \eqref{eqn:equiv4 to prop3.1} and Cauchy's inequality that \eqref{eqn:equiv3 to prop3.1} holds.

Let $ B_k (q)$ be an orthogonal basis of $S_k(q)$ containing $ \mathcal H_k (q)$. Since $\lambda_f(p)$ is real, 
\begin{align*}
\sumh_{f \in \mathcal H_{k}(q)}    &  \bigg|  \sum_{\substack{p \\ p \nmid q}}\frac{\log p \lambda_f(p)}{\sqrt{p}} \widehat H\left(\frac{\ell \log p}{\log Q}\right) \bigg|^{2m}   \leq     \sumh_{f \in  B_{k}(q)}    \bigg(   \sum_{\substack{p \\ p \nmid q}}\frac{\log p \lambda_f(p)}{\sqrt{p}} \widehat H\left(\frac{\ell \log p}{\log Q}\right) \bigg)^{2m} \\
  = &  \sumh_{f \in  B_{k}(q)} \sum_{\substack{ p_1 , \ldots, p_{2m}   \\ ( \mathfrak p(2m) , q ) =1 }} \bigg(\prod_{i =1}^{ 2m}     \frac{ \lambda_f (p_i) \log p_i  }{\sqrt{p_i}} \widehat H\left(\frac{\ell \log p_i}{\log Q}\right)     \bigg)     .
  \end{align*}
  By Lemma \ref{lem:cs}, the above equals  $\displaystyle \sum_{ \underline G \in \Pi_{2m}}  R_1 (\underline G ), $ 
  where
$$
R_1 ( \underline G) :=   \sumsharp_{\substack{ p_1 , \ldots, p_{\nu}   \\ ( \mathfrak p(\nu) , q ) =1 }} \prod_{j=1}^\nu   \left(      \frac{  \log p_j  }{\sqrt{p_j}} \widehat H\left(\frac{\ell \log p_j}{\log Q}\right)     \right)^{|G_j|}  \sumh_{f \in  B_{k}(q)} \prod_{j=1}^\nu  \lambda_f (p_j)^{|G_j|}   .$$
By Lemma \ref{lem:multoflambdaf}, $R_1 (\underline G)$ is a linear combination of 
\begin{equation}\label{eqn p sums f Bkq}
\sumsharp_{\substack{ p_1 , \ldots, p_{\nu}   \\ (\mathfrak p(\nu) , q ) =1 }} \prod_{j=1}^\nu   \left(      \frac{  \log p_j  }{\sqrt{p_j}} \widehat H\left(\frac{\ell \log p_j}{\log Q}\right)     \right)^{|G_j|}  \sumh_{f \in  B_{k}(q)}    \lambda_f (p_1^{k_1} \cdots p_\nu^{k_\nu}  )  
\end{equation} 
over $ 0 \leq  k_j \leq |G_j|$ for all $j \leq \nu$.

\new{We apply Lemma \ref{lem:petertruncate} to the $h$-sum in \eqref{eqn p sums f Bkq}. If $ k_1 = \cdots = k_\nu = 0 $, then the $h$-sum is bounded by $ 1$. In this case, $ |G_j|$ must be even by Lemma \ref{lem:multoflambdaf}, so the worst case should be $ |G_j|= 2$ for all $ j \leq \nu = m $. Hence, \eqref{eqn p sums f Bkq} is $ O( ( \log Q)^{2m}) $ when $ k_1 = \cdots = k_\nu = 0 $.  
If at least one of the $k_j $ is nonzero, then  \eqref{eqn p sums f Bkq} is 
$$  \ll    \frac{Q^{\epsilon}}{Q} \bigg(\prod_{j = 1}^{\nu} \sum_{  p_j \leq Q^{1/2n\ell}  }  \frac{ \log p_j }{p_j^{|G_j|/2}} p_j^{k_j/4} \bigg) \ll \frac{Q^{\epsilon}}{Q} \prod_{j = 1}^{\nu} Q^{\tfrac{3}{8n\ell}}  \ll Q^{-\tfrac 14 + \epsilon} $$
 by Lemma \ref{lem:petertruncate} and the support of $\widehat{H}$. Thus, we have
 $$ R_1 ( \underline{G})    \ll Q (\log Q)^{2m} ,$$
   which implies \eqref{eqn:equiv4 to prop3.1}. This concludes the proof of the proposition. }

\subsection{Proof of Proposition \ref{prop: when SnQA is small} - set partitions with small contribution} \label{section:proof prop 2}

  The collection $\Pi_n$ of all set partitions of $[n]$ forms a lattice with the partial ordering given by $\underline{A}  \preceq \underline{G} $ if every set $G_i $ in $\underline{G}$ is a union of sets in $\underline{A}$.
Then by Lemma \ref{lem:cs}, $S_n ( Q ; \underline A ) $ is a linear combination of 
\begin{equation}\label{eqn:frac q Phi h sum f P3}
     \frac{(- 2)^n}{(\log Q)^n }   \sum_q \Psi\left( \frac qQ \right)\sumh_{f \in \mathcal H_k(q)}  \prod_{ G_j \in \underline G} \mathcal{P}_3 ( G_j) \end{equation}
over $ \underline{G} \in \Pi_n$ with $ \underline A \preceq \underline G$, where
\begin{equation}\label{def:P3Gj}
  \mathcal{P}_3 ( G_j) :=  \sum_{ p \nmid q }    \left(   \frac{\log p  \lambda_f(p )}{\sqrt{p }} \right)^{|G_j| } H_{G_j} \left( \frac{\log p }{\log Q}\right). 
  \end{equation}
By Lemma \ref{lem:CLee3.5} for $ |G_j|=1 $ and by $|\lambda_f(p)| \leq 2$ and the prime number theorem for $ |G_j| \geq 2$, we find that
\begin{equation}\label{eqn:bound P3Gj}
\mathcal P_3 (G_j )  \ll 
\begin{cases}
 ( \log Q)^{2+\epsilon} & \textup{ if } |G_j|=1 \\
 ( \log Q)^2 & \textup{ if } |G_j |= 2, \\
 1 & \textup{ if } |G_j |>2 .    
\end{cases}
\end{equation}
  
 Suppose that $ |A_\ell | \geq 3 $ for some $A_\ell \in \underline A$ and $ \underline A \preceq \underline G \in \Pi_n $. Then $ |G_\ell |\geq 3 $ for some $ G_\ell \in \underline G$.  By \eqref{eqn:bound P3Gj} for $|G_j| \geq 2 $ and \eqref{eqn:MifQ absolute bound}, we find that \eqref{eqn:frac q Phi h sum f P3} is 
 $$  \ll   \frac{1}{(\log Q)^3 }   \sum_q \Psi\left( \frac qQ \right)\sumh_{f \in \mathcal H_k(q)}  \prod_{\substack{ G_i \in \underline{G} \\     G_i =\{ g_i \}  } }  \frac{ | \mathcal P_3 (\{ g_i \} ) | }{ \log Q} \ll  \frac{Q}{ ( \log Q)^3} .
 $$ 
 This proves \eqref{eqn:ajatleast3}.

 Next, we need \cite[Proposition 4.2]{BCL} to prove \eqref{eqn:onlyoneajis1}  and we state it here for the completeness. Note that we changed $\log q$ in the proposition to $\log Q$, but the proof remains the same.
\begin{prop}[Baluyot, Chandee and Li \cite{BCL}] \label{prop:boundSigma_1} 
Assume GRH. Let $\Phi$ be an even Schwartz function with $\widehat \Phi$ compactly supported in $(-4, 4) $. Then 
$$   \sum_q \Psi\bfrac{q}{Q} \sumh_{f \in \mathcal H_{k}(q)}\sum_{p\nmid q} \frac{\lambda_f(p) \log p}{\sqrt{p}} \widehat \Phi\left(\frac{\log p}{\log Q} \right)  \ll  Q . $$
\end{prop}
Without loss of generality, we only consider $\underline A = \{   A_1, \ldots , A_\nu  \} \in \Pi_n$ with $ |A_i | = 2 $ for all $ i \leq \nu-1$ and $ A_\nu = \{ 1 \} $.  
If $ \underline A \preceq \underline G$ and $ \underline A \neq \underline G$, then $ \underline G$ contains $ G_j $ with $|G_j |> 2 $.  By Lemma \ref{lem:cs} and \eqref{eqn:ajatleast3}, we have
$$  S_n ( Q; \underline A) = \frac{(- 2)^n}{(\log Q)^n }   \sum_q \Psi\left( \frac qQ \right)\sumh_{f \in \mathcal H_k(q)}  \prod_{ A_j \in \underline A} \mathcal{P}_3 ( A_j)  + O \left( \frac{ Q}{ ( \log Q)^3}\right) .$$ 
By  Lemma \ref{lem:multoflambdaf}, we see that $ \lambda_f (p)^2 = \lambda_f (p^2 ) + 1 $ for $ p \nmid q $. 
  By Lemmas \ref{lem:sumprimeDiagonal} and  \ref{lem:boundforSymL-functions}, we have 
\begin{equation}\label{eqn:P3AjlogQ asymp}
\frac{4 \mathcal{P}_3 ( A_j) }{ (\log Q)^2}  = \sum_{ p_j \nmid q} \frac{4(\log p_j  )^2 }{ (\log Q)^2} \frac{ (\lambda_f(p_j^2) + 1)}{p_j }  H_{A_j} \left( \frac{\log p_j}{\log Q}\right)  =  \mathscr I_2 (A_j) + O \left( \frac{ \log \log Q}{ \log Q} \right) 
\end{equation} 
for $ |A_j |=2$. 
  Hence, we find that
 \est{ S_n(Q; \underline A) &= \frac{ - 2 }{ \log Q  }  \sum_q \Psi\left( \frac qQ\right)\sumh_{f \in \mathcal H_k(q)}\left( \sum_{  p  \nmid q }  \frac{\log p  \lambda_f(p )}{\sqrt{p }}  \widehat \Phi_{1} \left( \frac{\log p }{\log Q}\right)\right)\\
&\times   \left( \prod_{j = 1}^{\nu-1}\mathscr I_2 (A_j) + O \left( \frac{ \log \log Q}{ \log Q} \right) \right) + O \left( \frac{Q}{ ( \log Q)^3} \right) .}
By Proposition \ref{prop:boundSigma_1} and \eqref{eqn:MifQ absolute bound}, we obtain \eqref{eqn:onlyoneajis1}.


\subsection{Proof of Proposition \ref{prop:SnQCn} - Main contribution} \label{section:proof prop 3}

We first prove \eqref{prop:eqn C0(n)case} similarly to the proof of \eqref{eqn:onlyoneajis1}.  Let $ \underline A \in \Pi_{n,2}$. We apply Lemma \ref{lem:cs} to remove the condition that the primes are distinct, and bound the remaining terms using \eqref{eqn:ajatleast3}. Thus we have
$$  S_n ( Q; \underline A) = \frac{(- 2)^n}{(\log Q)^n }   \sum_q \Psi\left( \frac qQ \right)\sumh_{f \in \mathcal H_k(q)}  \prod_{ A_j \in \underline A} \mathcal{P}_3 ( A_j)  + O \left( \frac{ Q}{ ( \log Q)^3}\right) .$$

By \eqref{eqn:P3AjlogQ asymp} and Lemma \ref{lem:asympforNQ} we have
\begin{align*}
   S_n ( Q; \underline A) =&     \sum_q \Psi\left( \frac qQ \right)\sumh_{f \in \mathcal H_k(q)} \left( \prod_{A_j \in \underline A }\mathscr I_2 (A_j) + O \left( \frac{ \log \log Q}{ \log Q} \right) \right) + O \left( \frac{ Q}{ ( \log Q)^3}\right) \\
   =&   N_0 (Q)    \prod_{A_j \in \underline A }\mathscr I_2 (A_j) + O \left( \frac{Q \log \log Q}{ \log Q} \right) .
   \end{align*}
 By \eqref{def:C0(n)} we have
  $$ \sum_{ \underline A \in \Pi_{n,2}} \frac{ S_n (Q; \underline A ) }{ N_0 (Q) } =  C_0 (n)  + O\left(  \frac{\log \log Q}{\log Q}\right) . $$
This proves \eqref{prop:eqn C0(n)case}.

 
 Next, we compute \eqref{prop:eqn C2(n)case}.
 Let   $ K \sqcup K_0 = [n] $ and $K= \{ k_1 , \ldots , k_\kappa \} $ for some $\kappa \geq 2 $. 
 Then  by \eqref{eqn:ajatleast3} and Lemma \ref{lem:cs}, we have
  \begin{multline*}
S_n   (Q;    \underline G \sqcup \pi_{K,1}   )  
=  \frac{  (-2)^\kappa}{ ( \log Q)^\kappa } \sum_{  \substack{ \underline A \in \Pi_K \\ |A_j | \leq 2 \textup{ for all } j }} \mu^* (\underline A ) \\
 \times \sum_q \Psi\left( \frac qQ \right) \sumh_{f \in \mathcal H_k(q)}  \prod_{ A_j \in \underline A} \mathcal P_3 (A_j )  \prod_{ G_j \in \underline G} \frac{4\mathcal P_3 ( G_j ) }{( \log Q)^2} + O \left( \frac{Q}{ ( \log Q)^3} \right)  
\end{multline*}
for $ \underline G   \in \Pi_{K_0 ,2 } $, where $\pi_{K,1} $ is defined in Definition \ref{def:set partition}.  
By \eqref{eqn:P3AjlogQ asymp}, \eqref{eqn:MifQ absolute bound} and \eqref{eqn:bound P3Gj}, we find that
 \begin{multline*}
S_n   (Q;    \underline G \sqcup \pi_{K,1}  )  
= \bigg( \prod_{G_j \in \underline{G}} \mathscr I_2 (G_j) \bigg)  \frac{  (-2)^\kappa}{ ( \log Q)^\kappa } \sum_{  \substack{ \underline A \in \Pi_K \\ |A_j | \leq 2 \textup{ for all } j }} \mu^* (\underline A ) \\
\times  \sum_q \Psi\left( \frac qQ \right) \sumh_{f \in \mathcal H_k(q)}  \prod_{ A_j \in \underline A} \mathcal P_3 (A_j )  + O \left( \frac{Q \log \log Q }{   \log Q } \right)  .
\end{multline*}
We once again apply Lemma \ref{lem:cs} and use the bound in \eqref{eqn:bound P3Gj} to convert the sum over $\underline{A}$ back into a sum over distinct primes, and obtain that
\begin{equation}\label{eqn:SnQGk 1}
S_n   (Q;    \underline G \sqcup \pi_{K,1}  ) =      \bigg( \prod_{G_j \in \underline{G}} \mathscr I_2 (G_j)  \bigg) S_\kappa ( Q; \pi_{K,1} ) + O \left( \frac{Q \log \log Q }{   \log Q } \right) .
\end{equation} 
It remains to compute $ S_\kappa (Q ; \pi_{K,1} )$.
 \begin{prop}\label{prop:main proposition}
Assume GRH. Let $ K  $ be a finite set of positive integers with $ |K| = \kappa   \geq 2 $. Then  we have 
$$   
  \lim_{Q \to \infty} \frac{ S_\kappa (Q ; \pi_{K,1} ) }{N_0 (Q)}  =  \sum_{ \substack{      K'     \sqcup K'' = K   \\    |K'| = 2   } }            \mathscr{V} ( K' , K'' ), $$
  where the function $\mathscr V$ is in \eqref{def:V}.
  \end{prop}
 We will prove Proposition \ref{prop:main proposition} in \S \ref{section:proof of main proposition}. Then it is easy to see that  \eqref{def:C2(n)},  \eqref{eqn:SnQGk 1} and  Proposition \ref{prop:main proposition} imply  
  \eqref{prop:eqn C2(n)case}.  This concludes the proof of the proposition.

 \section{Initial steps toward the proof of Proposition \ref{prop:main proposition}} \label{section:proof of main proposition}

   Let $K = \{ k_1 , \ldots, k_\kappa \}$ with $\kappa \geq 2$, then we have
$$  S_\kappa (Q ; \pi_{K,1} ) =   \frac{  (-2)^\kappa}{ ( \log Q)^\kappa } \sum_q \Psi\left( \frac qQ \right)  \sumsharp_{\substack{ p_1 , \ldots , p_\kappa \\ ( \mathfrak{p} (\kappa)  , q) =1 }}  \prod_{ j=1}^\kappa \left( \frac{ \log p_j }{ \sqrt{ p_j }} \widehat{\Phi}_{k_j} \left( \frac{ \log p_j }{ \log Q} \right)        \right)\sumh_{f \in \mathcal H_k(q)}  \lambda_f ( \mathfrak p (\kappa) ) .    $$
By Lemma \ref{lem:PeterssonNg}, we have
\begin{multline*}  
	  S_\kappa (Q ; \pi_{K,1} ) =   \frac{  (-2)^\kappa}{ ( \log Q)^\kappa }     \sum_{\substack{ L_1 , L_2 , d \\ L_1|d, \ (L_2 , d)=1 }} \frac{\mu(L_1L_2)}{L_1L_2}  \prod_{\substack{p|L_1 \\ p^2 \nmid d}}  \left( 1-\frac{1}{p^2} \right)^{-1}   \Psi\left( \frac{L_1L_2d}{Q} \right)    \\
\times \sumsharp_{\substack{ p_1 , \ldots , p_\kappa \\  ( \mathfrak p (\kappa ), L_1L_2d )=1 }}  \prod_{ j=1}^\kappa \left( \frac{ \log p_j }{ \sqrt{ p_j }} \widehat{\Phi}_{k_j} \left( \frac{ \log p_j }{ \log Q} \right)        \right) \sum_{\elltwo |L_2^{\infty} }\frac{\Delta_d( \mathfrak p (\kappa), \elltwo^2)}{\elltwo}.
\end{multline*}
We first show that only small $L_1 L_2 $ and $ \elltwo$ contribute to the main term. 
\begin{lemma} \label{lemma:Skappa truncate LE} Assume GRH. For $ |K|= \kappa \geq 2$, we have
\begin{multline*}  
	S_\kappa (Q ; \pi_{K,1} ) =   \frac{  (-2)^\kappa}{ ( \log Q)^\kappa }     \sum_{\substack{ L_1 , L_2 , d \\ L_1|d, \ (L_2 , d)=1  \\
      L_1 L_2 < \LL_{ \kappa+4 }  }} \frac{\mu(L_1L_2)}{L_1L_2}  \prod_{\substack{p|L_1 \\ p^2 \nmid d}}  \left( 1-\frac{1}{p^2} \right)^{-1}   \Psi\left( \frac{L_1L_2d}{Q} \right)    \\
\times \sumsharp_{\substack{ p_1 , \ldots , p_\kappa \\   ( \mathfrak p (\kappa ), L_1L_2d )=1  }}  \prod_{ j=1}^\kappa \left( \frac{ \log p_j }{ \sqrt{ p_j }} \widehat{\Phi}_{k_j} \left( \frac{ \log p_j }{ \log Q} \right)        \right) \sum_{\substack{\elltwo|L_2^{\infty} \\ \elltwo < \LL_{\kappa+2 }  } }\frac{\Delta_d(\mathfrak p (\kappa), \elltwo^2)}{\elltwo} + O\left( \frac{Q}{\log Q}\right),
\end{multline*}
where
$$\LL_{m} : = (\log Q)^{m}   .  $$ 
\end{lemma}

\begin{proof}[Proof of Lemma \ref{lemma:Skappa truncate LE}]

After rewriting $\Delta_d(\mathfrak p(\kappa), \elltwo^2)$ as in \eqref{def:Deltaq(m,n)}, 
it suffices to show that
\begin{multline*}  
    \summany_{ \substack{ L_1 , L_2 , d  , \elltwo \\ L_1|d, \ (L_2 , d)=1 , \ \elltwo |L_2^\infty  \\
      L_1 L_2 \geq \LL_{\kappa+4}  \textup{ or } \elltwo \geq \LL_{\kappa+2} }  }  \frac{\mu(L_1L_2)}{L_1L_2 \elltwo}  \prod_{\substack{p|L_1 \\ p^2 \nmid d}}  \left( 1-\frac{1}{p^2} \right)^{-1}   \Psi\left( \frac{L_1L_2d}{Q} \right)    \\
\times \sumsharp_{\substack{ p_1 , \ldots , p_\kappa \\   ( \mathfrak p (\kappa ), L_1L_2d )=1  }}  \prod_{ j=1}^\kappa \left( \frac{ \log p_j }{ \sqrt{ p_j }} \widehat{\Phi}_{k_j} \left( \frac{ \log p_j }{ \log Q} \right)        \right)  \sumh_{f \in B_k(d)} \lambda_f(\mathfrak p(\kappa)) \lambda_f(\elltwo^2) \ll Q( \log Q)^{\kappa -1} . 
\end{multline*}
Since $\lambda_f(p_1 \cdots p_\kappa) = \lambda_g(p_1 \cdots p_\kappa) = \lambda_g(p_1) \cdots \lambda_g(p_\kappa)$ 
 for some Hecke newform $g$ of level dividing $d$ and $  \lambda_g (\elltwo^2) \ll \tau(\elltwo^2)$, the above sum is
 \begin{multline*}
     \ll     \summany_{ \substack{ L_1 , L_2 , d  , \elltwo \\ L_1|d, \ (L_2 , d)=1 , \ \elltwo|L_2^\infty  \\
      L_1 L_2 \geq \LL_{\kappa+4}   \textup{ or } \elltwo \geq \LL_{\kappa+2} }  }  \frac{\tau(\elltwo^2)}{L_1L_2 \elltwo}    \Psi\left( \frac{L_1L_2d}{Q} \right)   \\
      \times  \sumh_{f \in B_k (d) } \left| \ \,   \sumsharp_{\substack{ p_1 , \ldots , p_\kappa \\   ( \mathfrak p (\kappa ), L_1L_2d )=1  }}  \prod_{ j=1}^\kappa \left( \frac{\lambda_g ( p_j)  \log p_j }{ \sqrt{ p_j }} \widehat{\Phi}_{k_j} \left( \frac{ \log p_j }{ \log Q} \right)        \right)    \right| .
 \end{multline*}
 By Lemma \ref{lem:cs}, \eqref{def:P3Gj} and \eqref{eqn:bound P3Gj} with $ g$ and $ L_1 L_2 d $ in place of $ f$ and $q $, respectively, we find that
$$  \sumsharp_{\substack{ p_1 , \ldots , p_\kappa \\  ( \mathfrak p (\kappa ), L_1L_2d )=1  }}  \prod_{ j=1}^\kappa \left( \frac{\lambda_g ( p_j)  \log p_j }{ \sqrt{ p_j }} \widehat{\Phi}_{k_j} \left( \frac{ \log p_j }{ \log Q} \right)        \right) \ll  \sum_{\underline G \in \Pi_\kappa }  \left| \prod_{ G_j \in \underline G } \mathcal{P}_3 ( G_j) \right| \ll ( \log Q)^{2 \kappa + \epsilon}   $$
for any $\epsilon>0$. Moreover, by Lemma \ref{lem:petertruncate} we have
$$ \sumh_{f \in B_k (d) } 1 = \Delta_d ( 1,1) \ll 1 + \frac{ \tau(d)}{d^{3/2}} \ll 1  . $$
Hence, it suffices to show that
\begin{equation}\label{eqn:L1L2de}
    \summany_{ \substack{ L_1 , L_2 , d  , \elltwo \\ L_1|d, \ (L_2 , d)=1 , \ \elltwo|L_2^\infty  \\
      L_1 L_2 \geq \LL_{\kappa+4}   \textup{ or } \elltwo \geq \LL_{\kappa+2} }  }  \frac{\tau(\elltwo^2)}{L_1L_2 \elltwo}    \Psi\left( \frac{L_1L_2d}{Q} \right)     
  \ll Q ( \log Q)^{-\kappa - 1 -\epsilon_0} 
\end{equation}
for some $ \epsilon_0 >0$.

The sum in \eqref{eqn:L1L2de} is less than 
\begin{equation}\label{eqn:L1L2de 2}
       \summany_{ \substack{ L_1 , L_2 , d  , \elltwo \\ L_1|d  , \ \elltwo|L_2^\infty  \\
      L_1 L_2 \geq \LL_{\kappa+4}     }  }  \frac{\tau(\elltwo^2)}{L_1L_2 \elltwo}    \Psi\left( \frac{L_1L_2d}{Q} \right)     
+     \summany_{ \substack{ L_1 , L_2 , d  , \elltwo \\ L_1|d  , \ \elltwo|L_2^\infty  \\    \elltwo \geq \LL_{\kappa+2} }  }  \frac{\tau(\elltwo^2)}{L_1L_2 \elltwo}    \Psi\left( \frac{L_1L_2d}{Q} \right)    .  
      \end{equation}
Since
$$ \sum_{\elltwo | L_2^{\infty}} \frac{\tau(\elltwo^2)}{\elltwo} \ll \prod_{p | L_2} \left( 1 + \frac 3p\right) \ll \tau(L_2),$$
 the first sum in \eqref{eqn:L1L2de 2} is 
 \begin{align*}
\ll       \summany_{ \substack{ L_1 , L_2 , d    \\ L_1|d, \ 
      L_1 L_2 \geq \LL_{\kappa+4}     }  }    \frac{\tau(L_2)}{L_1L_2 }    \Psi\left( \frac{L_1L_2d}{Q} \right) \leq 
         \frac{1}{\LL_{\kappa+4} }   \sum_{   L_1 , L_2 , m       } \tau(L_2)  \Psi\left( \frac{L_1^2 L_2 m}{Q} \right) \ll \frac{Q  }{(\log Q)^{\kappa+2}}.
 \end{align*}
Since 
$$ \sum_{\substack{\elltwo|L_2^{\infty} \\ \elltwo\ge \LL_{\kappa+2} }} \frac{ \tau(\elltwo^2)}{\elltwo} \ll \sum_{\substack{\elltwo|L_2^{\infty} \\ \elltwo\ge \LL_{\kappa+2} }} \frac{1}{\elltwo^{1-\epsilon}}
\leq  \frac{1}{\LL_{\kappa+2} ^{1-2\epsilon}} \sum_{\substack{\elltwo|L_2^{\infty} }} \frac{1}{\elltwo^\epsilon}
\ll \frac{\tau(L_2)}{\LL_{\kappa+2} ^{1-2\epsilon}}, 
$$
the second sum in \eqref{eqn:L1L2de 2} is
\begin{align*}
  \ll  \frac{1}{\LL_{\kappa+2}^{1-2\epsilon}}   \summany_{ \substack{ L_1 , L_2 , d   \\ L_1|d   }  }  \frac{\tau(L_2 )}{L_1L_2 }    \Psi\left( \frac{L_1L_2d}{Q} \right)     \ll  \frac{ 1 }{\LL_{\kappa+2}^{1-2\epsilon}}   \sum_{  L_1 , L_2 , m         }  \frac{\tau(L_2 )}{L_1L_2 }    \Psi\left( \frac{L_1^2L_2 m}{Q} \right) \ll   \frac{Q }{\LL_{\kappa+2}^{1-2\epsilon}} 
\end{align*}
 for any $\epsilon>0$. This proves \eqref{eqn:L1L2de} and the lemma follows. 
\end{proof}

Next, we do changes of variables for the sum in Lemma \ref{lemma:Skappa truncate LE} similarly to \S 6 in \cite{BCL} and the proof of Lemma \ref{lem:asympforNQ}. 
Since $L_1 | d $, we let $ d = L_1 m$ and apply \eqref{eqn:pL1d product to sum} to obtain
\begin{multline*}  
	S_\kappa (Q ; \pi_{K,1} ) =   \frac{  (-2)^\kappa}{ ( \log Q)^\kappa }     \sum_{\substack{ L_1 , L_2 , m \\   (L_2 , m)=1  \\
      L_1 L_2 < \LL_{\kappa+4} }} \frac{\mu(L_1L_2)\zeta_{L_1} (2)}{L_1L_2}    \sum_{ \ell_1  |(L_1,m)}\frac{\mu( \ell_1 )}{ \ell_1 ^2}   \Psi\left( \frac{L_1^2 L_2m}{Q} \right)    \\
\times \sumsharp_{\substack{ p_1 , \ldots , p_\kappa \\   ( \mathfrak p (\kappa ), L_1L_2 m )=1  }}  \prod_{ j=1}^\kappa \left( \frac{ \log p_j }{ \sqrt{ p_j }} \widehat{\Phi}_{k_j} \left( \frac{ \log p_j }{ \log Q} \right)        \right) \sum_{\substack{\elltwo|L_2^{\infty} \\ \elltwo < \LL_{\kappa+2} } }\frac{\Delta_{L_1m} ( \mathfrak p (\kappa) , \elltwo^2)}{\elltwo} + O\left( \frac{Q}{\log Q}\right),
\end{multline*}
where $\zeta_{L_1}(2)$ is defined in \eqref{def:zetaNs}.
We change the condition $  \ell_1  |m $ to $ m= \ell_1  n$ and then change the condition $(n, L_2 )=1$ by putting $\sum_{ \ell_2  |(n,L_2) } \mu( \ell_2 ) $. Then we find that
\begin{multline*}  
	S_\kappa (Q ; \pi_{K,1} ) =   \frac{  (-2)^\kappa}{ ( \log Q)^\kappa }     \sum_{\substack{ L_1 , L_2 , n   \\
      L_1 L_2 < \LL_{\kappa+4} }} \frac{\mu(L_1L_2)\zeta_{L_1} (2)}{L_1L_2}    \sum_{ \ell_1 | L_1 }\frac{\mu( \ell_1 )}{ \ell_1 ^2} \sum_{ \ell_2 |(n,L_2) } \mu( \ell_2 )    \Psi\left( \frac{L_1^2 L_2 \ell_1  n}{Q} \right)    \\
\times \sumsharp_{\substack{ p_1 , \ldots , p_\kappa \\   ( \mathfrak p (\kappa ), L_1L_2n )=1  }}  \prod_{ j=1}^\kappa \left( \frac{ \log p_j }{ \sqrt{ p_j }} \widehat{\Phi}_{k_j} \left( \frac{ \log p_j }{ \log Q} \right)        \right) \sum_{\substack{\elltwo|L_2^{\infty} \\ \elltwo < \LL_{\kappa+2}  } }\frac{\Delta_{L_1 \ell_1  n} (  \mathfrak p (\kappa) , \elltwo^2)}{\elltwo} + O\left( \frac{Q}{\log Q}\right).
\end{multline*}
By removing  the condition $  \ell_2 |n$, replacing $n$ by $ \ell_2  n$ and changing the order of sums, we find that 
\begin{multline*}  
	S_\kappa (Q ; \pi_{K,1} ) =   \frac{  (-2)^\kappa}{ ( \log Q)^\kappa }     \sum_{\substack{ L_1 , L_2   \\
      L_1 L_2 < \LL_{\kappa+4} }} \frac{\mu(L_1L_2)\zeta_{L_1} (2)}{L_1L_2}      \sum_{ \ell_1  | L_1 }\frac{\mu( \ell_1 )}{ \ell_1 ^2}  \sum_{ \ell_2 |L_2 } \mu( \ell_2 )   \sum_{\substack{\elltwo|L_2^{\infty} \\ \elltwo < \LL_{\kappa+2} } }\frac{1}{\elltwo}   \\
\times \sum_n \sumsharp_{\substack{ p_1 , \ldots , p_\kappa \\   ( \mathfrak p (\kappa ), L_1L_2 n  )=1  }}  \prod_{ j=1}^\kappa \left( \frac{ \log p_j }{ \sqrt{ p_j }} \widehat{\Phi}_{k_j} \left( \frac{ \log p_j }{ \log Q} \right)        \right)  
  \Psi\left( \frac{L_1^2 L_2 \ell_1  \ell_2 n}{Q} \right) \Delta_{L_1 \ell_1  \ell_2 n} ( \mathfrak p (\kappa) , \elltwo^2) + O\left( \frac{Q}{\log Q}\right) .
\end{multline*}

We want to remove the condition $ ( \mathfrak p (\kappa ), n ) = ( p_1 \cdots p_\kappa , n ) =1  $ in the above sum to apply Kuznetsov's formula. After relabeling $p_j$ with $ p_{k_j}$ for $ j = 1, \ldots , \kappa$, we split the $\#$ sum as 
$$   \sumsharp_{\substack{ p_{k_1} , \ldots , p_{k_\kappa} \\ ( \mathfrak p (K) , L_1L_2 )=1  
\\ (n, \mathfrak p (K) ) =1  }}  =   \sumsharp_{\substack{ p_{k_1} , \ldots , p_{k_\kappa} \\ ( \mathfrak p (K) ,  L_1L_2)=1     }}  -   \sumsharp_{\substack{ p_{k_1} , \ldots , p_{k_\kappa} \\  ( \mathfrak p (K), L_1L_2 ) =1  \\ (n, \mathfrak p (K) ) \neq 1  }}   =   \sumsharp_{\substack{ p_{k_1} , \ldots , p_{k_\kappa} \\ ( \mathfrak p (K),  L_1L_2 )=1  }}  - \sum_{ \substack{K_1 \sqcup K_2 = K \\ K_1 \neq \emptyset }}    \sumsharp_{\substack{ p_{k_1} , \ldots , p_{k_\kappa} \\ ( \mathfrak p (K) ,  L_1L_2 )=1  \\ 
\mathfrak p (K_1 ) |n   \\
 ( \mathfrak p (K_2) ,  n  ) =1 }}.   $$
Hence we obtain the decomposition
\begin{equation}\label{eqn:SkappaQpiK1 CKQ decomp}
	S_\kappa (Q ; \pi_{K,1} ) =  
\mathscr C_K   (Q )  -  \sum_{ \substack{K_1 \sqcup K_2 = K \\ K_1 \neq \emptyset }}  \mathscr C_{K_1, K_2} (Q )   + O\left( \frac{Q}{\log Q}\right) , 
\end{equation}
where the main term $\mathscr C_K(Q)$ corresponds to the full sum with coprimality condition $(\mathfrak p(K), L_1L_2) = 1$, and each $\mathscr C_{K_1, K_2}(Q)$ captures the contribution when $\mathfrak p(K_1)$ divides $n$ but $\mathfrak p(K_2)$ does not. More precisely,
\begin{equation}  \label{def:CKQ}
\mathscr C_K (Q) :=       \frac{  (-2)^\kappa}{ ( \log Q)^\kappa }     \sumprime_{\mathbb{L}} \frac{\mu(L_1L_2)\zeta_{L_1} (2) }{L_1L_2}    \frac{\mu( \ell_1   \ell_2 )}{ \ell_1 ^2 \elltwo} C_K   (Q; \mathbb{L} ) , 
      \end{equation}
      where the prime sum is over 
           \begin{equation}\label{def:mathbbL}
      \mathbb{L} := ( L_1 , L_2 , \ell_1 , \ell_2 , \elltwo ) 
      \end{equation}
satisfying the conditions
\begin{equation} \label{conditions mathbbL}
 \ell_1  | L_1 , ~ \ell_2  | L_2 ,~  L_1 L_2 < \LL_{\kappa+4} , ~\elltwo|L_2^{\infty} ~\textup{ and }~ \elltwo < \LL_{\kappa+2}  .
 \end{equation}
  $  \mathscr C_{K_1 , K_2 }(Q)  $ is defined by replacing $C_K$ to $C_{K_1, K_2}$ in \eqref{def:CKQ},
\begin{equation} \label{def:CKQLre} \begin{split}
 C_{K} & (Q; \mathbb{L} )  := \sum_n  \Psi\left( \frac{L_1^2 L_2 \ell_1  \ell_2 n}{Q} \right)\\
 &  \hskip 1in \times   \sumsharp_{\substack{ p_{k_1} , \ldots , p_{k_\kappa} \\ ( \mathfrak p (K) , L_1L_2 )=1   }}
  \prod_{ j=1}^\kappa \left( \frac{ \log p_{k_j} }{ \sqrt{ p_{k_j} }} \widehat{\Phi}_{k_j} \left( \frac{ \log p_{k_j} }{ \log Q} \right)        \right)  
   \Delta_{L_1 \ell_1  \ell_2 n} ( \mathfrak p (K), \elltwo^2),
  \end{split}\end{equation}
  and $   C_{K_1, K_2}   (Q; \mathbb{L} ) $ is defined by adding the conditions $\mathfrak p (K_1 ) |n   $ and $
 ( \mathfrak p (K_2) ,  n  ) =1$  
to the $\#$- sum in \eqref{def:CKQLre}. Furthermore, we split $ \mathscr C_{K_1 , K_2 } (Q)$   depending on the contribution of $\mathfrak p ( K_1) = \prod_{ k_j \in K_1 }p_{k_j} <  \LL_{3\kappa} $ or $ \geq \LL_{3\kappa} $  
such that 
$$ \mathscr C_{K_1 , K_2 } (Q) = \mathscr C_{K_1 , K_2 , <  } (Q   )  + \mathscr C_{K_1 , K_2 , \geq  } (Q) .$$
 Then Proposition \ref{prop:main proposition} follows by applying the following lemmas to \eqref{eqn:SkappaQpiK1 CKQ decomp}.
\begin{lem} \label{lem:CMQ(1,1)} 
Assume GRH.
Let $K$ be a set of positive integers such that $|K| \geq 2 $. Then we have
\est{\frac{\mathscr C_K (Q)}{N_0(Q)}  =   \sum_{\substack{K' \sqcup K'' = K \\ |K'| = 2}} \mathscr V(K', K'')  + O\left( (\log Q )^{-1+\epsilon}\right)}
for any $\epsilon>0$, where the function $\mathscr V$ is in \eqref{def:V}.
\end{lem}
\begin{lem} \label{lem:CMQ(R)_fill prime}
Assume GRH. Let $K$ be a set of positive integers such that $|K| \geq 2 $, $K_1 \sqcup K_2 = K$ and $ K_1 \neq \emptyset $. Then
\est{\mathscr C_{K_1 , K_2 , <} (Q)  \ll \frac{Q}{\log Q}.}

\end{lem}
\begin{lem} \label{lem:boundtruncatedprimetR} 
Assume GRH. Let $K$ be a set of positive integers such that $|K| \geq 2 $, $K_1 \sqcup K_2 = K$ and $ K_1 \neq \emptyset $. Then
\est{\mathscr C_{K_1 , K_2 , \geq } (Q)  \ll \frac{Q}{\log Q}.}

\end{lem}

We will prove Lemma \ref{lem:CMQ(1,1)} in \S \ref{sec:applyKuznetsov} - \S \ref{sec:ctnrescalc} and Lemma \ref{lem:CMQ(R)_fill prime} in \S \ref{sec:fillprimes} by modifying the arguments of the proof of Lemma \ref{lem:CMQ(1,1)}.
 We end this section with a proof of Lemma \ref{lem:boundtruncatedprimetR}.

\begin{proof}[Proof of Lemma \ref{lem:boundtruncatedprimetR}] Without loss of generality and to simplify notation, we only consider the case when $K = [\kappa]  $ for $ \kappa \geq 2 $, $K_1 = [\kappa_1 ]$ for $ 1 \leq \kappa_1 \leq \kappa$ and $ K= K_1 \sqcup K_2 $. Then by the definitions below \eqref{def:CKQ}, we find that
\begin{equation}
 \label{def:CK1K2Q}   
\mathscr C_{K_1, K_2, \geq  }  (Q)  
=     \frac{  (-2)^\kappa}{ ( \log Q)^\kappa }     \sumprime_{\mathbb{L}} \frac{\mu(L_1L_2)\zeta_{L_1} (2) }{L_1L_2}    \frac{\mu( \ell_1   \ell_2 )}{ \ell_1 ^2   \elltwo}  C_{K_1,K_2, \geq   }   (Q; \mathbb{L} ) , 
\end{equation}
 where     
\begin{equation} \label{def:CK1K2QLre} \begin{split}
 & C_{K_1, K_2 , \geq  }  (Q; \mathbb{L} ) := \sum_n  \Psi\left( \frac{L_1^2 L_2 \ell_1  \ell_2 n}{Q} \right)\\
 &  \hskip 1in \times  \sumsharp_{\substack{ p_1 , \ldots , p_\kappa \\ ( \mathfrak p (K) , L_1L_2 )=1   \\
 \mathfrak p (K_1) | n  , ~ \mathfrak p (K_1) \geq \LL_{3\kappa } \\
 (\mathfrak p ( K_2 )  , n ) = 1  
  }}
  \prod_{ j=1}^\kappa \left( \frac{ \log p_j }{ \sqrt{ p_j }} \widehat{\Phi}_{j} \left( \frac{ \log p_j }{ \log Q} \right)        \right)  
   \Delta_{L_1 \ell_1  \ell_2 n} ( \mathfrak p (K) , \elltwo^2).
  \end{split}\end{equation}
By replacing $n$ with $\mathfrak p (K_1 ) n $, we eliminate the condition $ \mathfrak p ( K_1 ) |n $. Then applying the definition of $ \Delta_q (m,n)$, we have
 \begin{align*}
     & C_{K_1, K_2 , \geq }  (Q;  \mathbb{L} )   = \sum_n    \sumsharp_{\substack{ p_1 , \ldots , p_\kappa \\ ( \mathfrak p (K) , L_1L_2 )=1   \\ 
 (\mathfrak p ( K_2 )  , n ) = 1  \\
 \mathfrak p (K_1) \geq \LL_{3\kappa }  }}
  \prod_{ j=1}^\kappa \left( \frac{ \log p_j }{ \sqrt{ p_j }} \widehat{\Phi}_{j} \left( \frac{ \log p_j }{ \log Q} \right)        \right)\\
    &      
  \hskip 1.5in \times \Psi\left( \frac{L_1^2 L_2 \ell_1  \ell_2  \mathfrak p (K_1) n}{Q} \right) \sumh_{f \in B_k (L_1 \ell_1  \ell_2  \mathfrak p (K_1)  n) }
   \lambda_f ( \mathfrak p (K))\lambda_f (  \elltwo^2) ,
  \end{align*}
where we have taken the $h$-sum over $f$ above to be over an Atkin-Lehner basis so that 
$$f = f^{*(g)} = \sum_{\ell|g} \xi_{g}(\ell) f^*|_{\ell}
$$
for some newform $f^*$ of level dividing $L_1 \ell_1  \ell_2  \mathfrak p (K_1) n$, and some $g|L_1 \ell_1  \ell_2  \mathfrak p (K_1) n$. By comparing Fourier coefficients we find that
$$ \lambda_f (\mathfrak p ( K )) = \sum_{ \ell | ( g , \mathfrak p(K) ) } \xi_g (\ell) \lambda_{f^*} \left(  \frac{ \mathfrak{p}(K)}{\ell} \right) . $$
Since $ ( g , \mathfrak p(K) ) | ( L_1 \ell_1  \ell_2  \mathfrak p (K_1) n , \mathfrak p(K) ) = \mathfrak p(K_1)  $, there is $ K_3 \subset K_1 $ such that $ ( g , \mathfrak p(K) ) = \mathfrak p ( K_3 ) $. Moreover, $ \ell |  \mathfrak p ( K_3) $  is equivalent to $ \ell = \mathfrak p (K_4) $ for some $ K_4 \subset K_3 $. Note that $ \mathfrak p ( \emptyset ) = 1 $. Thus, we have
$$ \lambda_f (\mathfrak p ( K )) =\lambda_{f^*} ( \mathfrak p (K_2) )  \sum_{ K_4 \subset K_3 } \xi_g (\mathfrak p (K_4) ) \lambda_{f^*} \left(  \mathfrak{p}(K_1 \setminus K_4 )  \right) . $$
Hence
 \begin{align*}
     & C_{K_1, K_2 , \geq  }  (Q;\mathbb{L})   = \sum_n   \sumsharp_{\substack{ p_1 , \ldots , p_{\kappa_1} \\ ( \mathfrak p (K_1) , L_1L_2 )=1      \\
 \mathfrak p (K_1) \geq \LL_{3\kappa } }}
  \prod_{ j=1}^{\kappa_1} \left( \frac{ \log p_j }{ \sqrt{ p_j }} \widehat{\Phi}_{j} \left( \frac{ \log p_j }{ \log Q} \right)        \right)  \\
    &     \times 
  \sumh_{f \in B_k (L_1 \ell_1  \ell_2  \mathfrak p (K_1)  n) }
  \Psi\left( \frac{L_1^2 L_2 \ell_1  \ell_2  \mathfrak p (K_1) n}{Q} \right) \lambda_f (  \elltwo^2)   \sum_{ K_4 \subset K_3 } \xi_g (\mathfrak p (K_4) ) \lambda_{f^*} \left(  \mathfrak{p}(K_1 \setminus K_4 )  \right)\\
  &  \times \sumsharp_{\substack{ p_{\kappa_1+1} , \ldots , p_\kappa \\ ( \mathfrak p (K_2 ) , L_1L_2  \mathfrak p(K_1) n )=1        }}
  \prod_{ j=\kappa_1+1}^\kappa \left( \frac{ \lambda_{f^*} ( p_j ) \log p_j }{ \sqrt{ p_j }} \widehat{\Phi}_{j} \left( \frac{\log p_j }{ \log Q} \right)        \right) .
  \end{align*}
  
The last $ \sharp$-sum is by Lemma \ref{lem:cs} 
\begin{align*}
\sumsharp_{\substack{ p_{\kappa_1+1} , \ldots , p_\kappa \\ ( \mathfrak p (K_2 ) , L_1L_2  \mathfrak p(K_1) n )=1        }}
  \prod_{ j=\kappa_1+1}^\kappa \left( \frac{ \lambda_{f^*} ( p_j ) \log p_j }{ \sqrt{ p_j }} \widehat{\Phi}_{j} \left( \frac{\log p_j }{ \log Q} \right)        \right) =\sum_{ \underline G \in \Pi_{K_2} } \mu^*(\underline{G}) \prod_{G_j \in \underline{G}} \mathcal P_3 ( G_j ),
  \end{align*}
where $ \mathcal P_3 (G_j)$ is defined in \eqref{def:P3Gj} with $ q= L_1L_2  \mathfrak p(K_1) n $. Since $ n $ can be any positive integer, we have an upper bound for $ \mathcal P_3 (G_j)$ depending on $ n, Q$. 
By Lemma \ref{lem:CLee3.5} and the prime number theorem, we find that
\begin{equation}\label{eqn:bound P3Gj 2}
\mathcal P_3 (G_j )  \ll 
\begin{cases}
 ( \log Q)^{2+\epsilon} +    \log n & \textup{ if } |G_j|=1 ,\\
 ( \log Q)^2 & \textup{ if } |G_j |= 2, \\
 1 & \textup{ if } |G_j |>2    .  
\end{cases}
\end{equation}
 Hence, we have 
$$  \sumsharp_{\substack{ p_{\kappa_1+1} , \ldots , p_\kappa \\ ( \mathfrak p (K_2 ) , L_1L_2  \mathfrak p(K_1) n )=1        }}
  \prod_{ j=\kappa_1+1}^\kappa \left( \frac{ \lambda_{f^*} ( p_j ) \log p_j }{ \sqrt{ p_j }} \widehat{\Phi}_{j} \left( \frac{\log p_j }{ \log Q} \right)        \right)   \ll ( \log Q)^{2\kappa_2 + \epsilon} +  ( \log n)^{\kappa_2} , $$
  where $ \kappa_2 := |K_2| = \kappa - \kappa_1$.
Since $\xi_g(\ell) \ll \ell^{\epsilon}$ by \eqref{eqn:xigdbdd2}, $ \lambda_{f^*} ( \mathfrak p ( K_1 \setminus K_4 )) \ll  \mathfrak p ( K_1 \setminus K_4 )^{ \frac{7}{64} + \epsilon}  $ and $ \lambda_f(\elltwo^2) \ll \elltwo^{ \frac{7}{32} +\epsilon } $ by \eqref{eqn:KimSbdd}, we find that
 \begin{align*}
       C_{K_1, K_2 , \geq   }  (Q;  \mathbb{L} )   & \ll \elltwo^{ \frac{7}{32} +\epsilon }   \sum_{\substack{ p_1 , \ldots , p_{\kappa_1}     \\   \mathfrak p (K_1) \geq \LL_{3\kappa }  }}  \mathfrak{p}(K_1)^{  \frac{7}{64} - \frac12 +  \epsilon} 
  \prod_{ j=1}^{\kappa_1}    \left| \widehat{\Phi}_{j} \left( \frac{ \log p_j }{ \log Q} \right) \right|            \\
    &      \times \sum_n 
  \Psi\left( \frac{L_1^2 L_2 \ell_1  \ell_2  \mathfrak p (K_1) n}{Q} \right)    \left( ( \log Q)^{2\kappa_2 + \epsilon} +  ( \log n)^{\kappa_2} \right)  \\ 
  & \ll     Q( \log Q)^{2\kappa_2 + 1+ \epsilon}     \frac{\elltwo^{ \frac{7}{32} +\epsilon } }{L_1^2 L_2 \ell_1  \ell_2   }   \sum_{\substack{ p_1 , \ldots , p_{\kappa_1}     \\
 \mathfrak p (K_1) \geq \LL_{3\kappa } }}
  \mathfrak{p}(K_1)^{  \frac{7}{64} - \frac32 +  \epsilon} \prod_{ j=1}^{\kappa_1}     \left|\widehat{\Phi}_{j} \left( \frac{ \log p_j }{ \log Q} \right)   \right|\\
  & \ll    Q( \log Q)^{2\kappa_2 + 1+ \epsilon} \LL_{3\kappa } ^{\frac{7}{64} - \frac12 +  \epsilon }   \frac{\elltwo^{ \frac{7}{32} +\epsilon } }{L_1^2 L_2 \ell_1  \ell_2   } .
  \end{align*}
By applying this bound to 
\eqref{def:CK1K2Q} and the fact that $ \kappa - \kappa_2 = \kappa_1 \geq 1 $, we have
\begin{align*}
\mathscr C_{K_1, K_2, \geq   }  (Q) 
\ll  &  Q( \log Q)^{2\kappa_2 - \kappa + 1+ \epsilon} \LL_{3\kappa }^{-\frac{25}{64}     }     \sum_{L_1, L_2}  \frac{ 1 }{L_1^3 L_2^2 } \sum_{ \ell_1  |L_1 ,  \ell_2 |L_2}\frac{1}{ \ell_1 ^3 \ell_2   }  \sum_{ \elltwo|L_2^{\infty}   }\frac{1}{\elltwo^{1- \frac{7}{32} -\epsilon } }  \\
\ll &  Q( \log Q)^{2\kappa_2 - \kappa + 1+ \epsilon} \LL_{3\kappa }^{ - \frac{25}{64}     } \ll \frac{Q}{\log Q} . 
      \end{align*} 
This proves the lemma.
\end{proof}

\section{Applying Kuznetsov's formula to $\mathscr C_K (Q)$} \label{sec:applyKuznetsov}

In this section, we prove Lemma \ref{lem:CMQ(1,1)}. First, we estimate the sum $   C_{K}  (Q ;\mathbb{L} ) $ in \eqref{def:CKQLre} for $ \mathbb{L}=  ( L_1, L_2,  \ell_1  ,  \ell_2 , \elltwo) $ satisfying \eqref{conditions mathbbL}. By Petersson's formula (Lemma \ref{lem:usualPetersson}) we write
\begin{align*}
 C_{K}  (Q; \mathbb L ) =2& \pi i^{-k}  \sumsharp_{\substack{ p_{k_1} , \ldots , p_{k_\kappa} \\ ( \mathfrak p (K) , L_1L_2 )=1   }}
  \prod_{ j=1}^\kappa \left( \frac{ \log p_{k_j} }{ \sqrt{ p_{k_j} }} \widehat{\Phi}_{k_j} \left( \frac{ \log p_{k_j} }{ \log Q} \right)        \right)  
   \\
 &  \times    \sum_{c \geq 1} \sum_{\substack{n  }} \frac{S( \elltwo^2, \mathfrak p (K) ; cL_1 \ell_1  \ell_2 n)}{cL_1 \ell_1  \ell_2 n}   \Psi \left( \frac{L_1^2 L_2  \ell_1   \ell_2  n }{Q}  \right) 
 J_{k-1} \left( \frac{4\pi \elltwo \sqrt{\mathfrak p (K) }}{cL_1 \ell_1  \ell_2 n}\right)  .
  \end{align*}

Next we introduce a smooth partition of unity. Let $V$ be a smooth function compactly supported on $[1/2, 3]$ satisfying $\sumd_P V\bfrac{x}{P} = 1$ for all $x \geq 1$, where $\sumd_{P}$ denotes a sum over $P = 2^j$ for $j\geq 0$. Moreover, let $V_0$ be a smooth function that is compactly supported in $(\alpha_1 , \beta_1 )$ for some $0< \alpha_1< 1/2 $ and $\beta_1 > 3$ such that $V_0 (\xi) = 1$ when $\xi \in [1/2, 3]$.  By introducing the partition of unity to the prime sums, we find that
\begin{multline}\label{eqn:CKQ 1}
 C_{K}  (Q; \mathbb{L}) =   2  \pi i^{-k}  \sumd_{P_1, \ldots , P_\kappa }  \sum_{c \geq 1}  \sumsharp_{\substack{ p_{k_1} , \ldots , p_{k_\kappa} \\ ( \mathfrak p (K) , L_1L_2 )=1   }}
  \prod_{ j=1}^\kappa \left( \frac{ \log p_{k_j} }{ \sqrt{ p_{k_j} }} V \left( \frac{ p_{k_j}}{P_j} \right)      \right)  
   \\
\times   \sum_{n}   \frac{S( \elltwo^2, \mathfrak p (K) ; cL_1 \ell_1  \ell_2 n )}{ cL_1 \ell_1  \ell_2 n }  
 J_{k-1} \left( \frac{4\pi \elltwo \sqrt{\mathfrak p (K) }}{cL_1 \ell_1  \ell_2 n}\right)  H \left( \frac{4\pi  \elltwo \sqrt{\mathfrak p (K)  }}{cL_1 \ell_1  \ell_2 n} , \frac{ p_{k_1} }{P_1 } , \ldots , \frac{p_{k_\kappa}}{P_\kappa }   \right)   ,
  \end{multline}
  where 
\begin{equation}\label{eqn:Hdef}
 H(\xi, \lambb ) := 
\Psi\left(\frac{X}{\xi}\sqrt{\lambda_1\cdots \lambda_\kappa}\right)\left[\prod_{j = 1}^{\kappa} \widehat \Phi_{k_j}\left(\frac{\log (\lambda_j P_j)}{\log Q}\right)V_{0}(\lambda_j)\right]  
\end{equation}
for $ \xi \in \mathbb{R}$, $\lambb = ( \lambda_1 , \ldots , \lambda_\kappa ) \in \mathbb{R}^\kappa $ and
\begin{equation} \label{eqn: Xdefinition}
X   := \frac{4\pi L_1 L_2 \elltwo \sqrt{P_1 \cdots P_{\kappa}}}{cQ}.
\end{equation}
\begin{rem} \label{remark: support 1}
The dyadic sum over $P_i$ in \eqref{eqn:CKQ 1} is supported on
\begin{equation}\label{eqn:Pibdd}
    P_1 \cdots P_\kappa \le Q^{4-\delta}
\end{equation}
for some $\delta>0$ by the support of the $\widehat \Phi_{j}$.
 $H (\xi, \lambb) $ is nonzero only if $\lambda_j \asymp 1$ for all $j \leq \kappa $ by the support of $V_0$ and $\alpha_2 \leq \frac{\xi}{X} \leq \beta_2 $ for some $ 0 < \alpha_2 < \beta_2 $ by the support of $\Psi$. Let $ W $ be a smooth function that is compactly supported in $ ( \alpha_3 , \beta_3 ) $ for some $ 0 < \alpha_3 < \alpha_2 $ and $ \beta_3 > \beta_2 $, and $W(x)=1$ for $\alpha_2 \leq x \leq \beta_2 $. Then we can multiply $ W \left(   \frac{4\pi  \elltwo \sqrt{\mathfrak p (K)  }}{cL_1 \ell_1  \ell_2 n} \frac{1}{X}     \right)$ to the right hand side of \eqref{eqn:CKQ 1} with no harm.
\end{rem}

We want to apply the Fourier inversion of $H$. For $ u \in \mathbb{R}$ and $\vb = (v_1,\ldots , v_{\kappa}) \in \mathbb{R}^\kappa $, we let 
\begin{equation}\label{eqn:Hhat}
\widehat{H}(u, \vb) = \int_{-\infty}^\infty \cdots \int_{-\infty}^\infty H(\xi, \lambb) \e{-\xi u - \lambda_1 v_1 -\cdots - \lambda_{\kappa}v_{\kappa}} d\xi d\lambda_1 \cdots d\lambda_{\kappa}  
\end{equation}
 be the usual Fourier transform of $H$.  For reference later, we record the following bounds on $\widehat H$.
\begin{lem}\label{lem:hatHbdd}
With notations as above, we have that for any integers $A_1, A_2 \geq 0 $ 
\begin{equation*}
\widehat{H}(u, \vb) \ll_{A_1, A_2  }   \frac{X}{(1+|u|X)^{A_1}}     \frac{1}{ Y(\vb)^{A_2}},
\end{equation*}
where $Y(\vb)$ is defined in \eqref{def:Yv}.
\end{lem}

\begin{proof}
If $\lambda_j \asymp 1$ for all $ j \leq \kappa $ and $\xi \asymp X $, then we have
\begin{align*}
\frac{\partial^{n_0}}{\partial^{n_0} \xi} \frac{\partial^{n_1}}{\partial^{n_1} \lambda_1} \cdots \frac{\partial^{n_{\kappa}}}{\partial^{n_{\kappa}} \lambda_\kappa}  H(\xi, \lambb  ) \ll \frac{1}{X^{n_0} }
\end{align*}
for any nonnegative integers $ n_0 , \ldots, n_\kappa $. 
The lemma then follows from repeated integration by parts.
\end{proof}

 Define
\begin{multline}\label{def:SigmaDis} 
 \Sigma_{\mathfrak T}  := 2\pi i^{-k}   \sumd_{P_1, \ldots , P_{\kappa}}  \sum_{c\geq 1} \int_{-\infty}^\infty \cdots \int_{-\infty}^\infty  \widehat{H}(u, \vb) \\     \times \sumsharp_{\substack{ p_{k_1} , \ldots , p_{k_\kappa} \\ ( \mathfrak p (K) , L_1L_2 )=1   }}
  \prod_{ j=1}^\kappa \left( \frac{ \log p_{k_j} }{ \sqrt{ p_{k_j} }} V \left( \frac{ p_{k_j}}{P_j} \right)   \e{ \frac{p_{k_j}}{P_j} v_j }   \right)    \mathfrak T (c,\mathfrak p(K); u)  \> du \> dv_1 \cdots \> dv_{\kappa}.  
\end{multline}
By the Fourier inversion, \eqref{eqn:CKQ 1} and Remark \ref{remark: support 1}, we find that
\begin{equation}\label{eqn:CKQ 2}
 C_{K}  (Q; \mathbb{L}) = \Sigma_{\cS},
 \end{equation}
 where
 $$ \cS (c, \mathfrak p(K); u) :=  \sum_{n}   \frac{S( \elltwo^2, \mathfrak p (K) ; cL_1 \ell_1  \ell_2 n )}{ cL_1 \ell_1  \ell_2 n }  
 h_u  \left( \frac{4\pi \elltwo \sqrt{\mathfrak p (K) }}{cL_1 \ell_1  \ell_2 n}\right)  $$
 and
 \begin{equation} \label{def:huxi}
h_u(\xi) := J_{k-1}(\xi)W\bfrac{\xi}{X} \e{u\xi}.
\end{equation}

 By Kuznetsov's formula (Lemma~\ref{lem:kuznetsov}) with $ N=cL_1  \ell_1   \ell_2  $, $m=\elltwo^2$, $n=\mathfrak p(K) $ and $ \phi = h_u $, we find that
\es{ \label{eqn:sumoverpseparateintoDisCtnHol}
 \cS (c, \mathfrak p(K); u)   = \Dis(c, \mathfrak p(K); u) + \Ctn(c, \mathfrak p(K); u) + \Hol(c, \mathfrak p(K); u) ,
}
where
\begin{equation} \label{def:DisCtnHol}\begin{split}
\textup{Dis}(c,\mathfrak p(K); u) &:=  \sum_{j=1}^{\infty} \frac{\overline{\rho_j}(\elltwo^2)\rho_j(\mathfrak p(K)) \sqrt{\mathfrak p(K)\elltwo^2} }{\cosh(\pi \kappa_j)} h_{u,+}(\kappa_j) , \\
\Ctn(c, \mathfrak p(K); u) &:= \frac{1}{4\pi}\sum_{\mathfrak c_{\chi}^2 | M | N}  \int_{-\infty}^{\infty}  \rho_{\chi, M, N}(\mathfrak p(K), t) \overline{\rho_{\chi, M, N}(\elltwo^2, t)} h_{u, +}(t) \> dt, \textup{ \;\;\;and}\\
\Hol(c,\mathfrak p(K); u) &:= \frac{1}{2\pi} \sum_{ \substack{\ell\geq 2 \mbox{\scriptsize{ even}} \\ 1 \leq j \leq \theta_{\ell}(N)} } (\ell - 1)! \sqrt{\mathfrak p(K)\elltwo^2} \, \overline{\psi_{j,\ell}}(\elltwo^2) \psi_{j,\ell} (\mathfrak p(K)) h_{u, h}(\ell).
\end{split}\end{equation}
Here, $ h_{u, +}$ and $ h_{u,h}$ are the Bessel transforms of $h_u$ defined in Lemma \ref{lem:kuznetsov}.  Note that the forms appearing in \eqref{def:DisCtnHol} are of level $cL_1 \ell_1  \ell_2 $. 
By \eqref{def:SigmaDis}, \eqref{eqn:CKQ 2} and \eqref{eqn:sumoverpseparateintoDisCtnHol}, we find that
\begin{equation}\label{eqn:CKQ Sigma Dis Ctn Hol}
 C_K (Q ; \mathbb{L}) = \Sigma_{\Dis} +  \Sigma_{\Ctn} + \Sigma_{\Hol}. 
 \end{equation}
Then we have the following propositions.
\begin{prop}\label{prop:DisHol}
Assume GRH, \eqref{conditions mathbbL} and \eqref{eqn:Pibdd}. 
For any $\epsilon>0$, we have 
$$ \Sigma_{\Dis} \ll  Q^{1 -  {\delta}/{2} + \epsilon} , \qquad \Sigma_{\Hol} \ll  Q^{1 -  {\delta}/{2} + \epsilon}.$$
\end{prop}

\begin{prop}\label{prop:Ctn}  Assume GRH, \eqref{conditions mathbbL} and \eqref{eqn:Pibdd}. Let $K$ be a set of positive integers such that $ |K| = \kappa \geq 2 $. Then we have
\begin{equation*}
    \Sigma_{\Ctn} = \frac{ Q ( \log Q)^{\kappa}  \widetilde{\Psi} ( 1)\delta_{\elltwo=1} }{ (-2)^{\kappa}  L_1^2 L_2  \ell_1 \ell_2  }    \sum_{ \substack{ K'\sqcup K'' = K \\ |K'|=2}  } \mathscr V ( K' , K'') + O\left(  \frac{ Q ( \log Q)^{\kappa -1+\epsilon }  }{ L_1 L_2 \elltwo}  \right) 
\end{equation*}
for any $ \epsilon>0$, where the function $\mathscr V$ is in \eqref{def:V}.
\end{prop}

The above propositions imply Lemma \ref{lem:CMQ(1,1)} as follows.

\subsection{Proof of Lemma \ref{lem:CMQ(1,1)}}
\label{sec:summaryproofofCMQ11}

By \eqref{def:CKQ}, \eqref{eqn:CKQ Sigma Dis Ctn Hol} and Propositions \ref{prop:DisHol} and \ref{prop:Ctn}, we have
$$\mathscr C_K (Q)  =        Q  \widetilde{\Psi} ( 1)    \sum_{\substack{ L_1 , L_2  ,  \ell_1 ,  \ell_2  \\   \ell_1  | L_1 , \  \ell_2  | L_2 \\
      L_1 L_2 < \LL_{\kappa+4} }} \frac{\mu(L_1L_2)\zeta_{L_1} (2) }{L_1^3L_2^2}    \frac{\mu( \ell_1   \ell_2 )}{ \ell_1 ^3 \ell_2}  
           \sum_{ \substack{ K'\sqcup K'' = K \\ |K'|=2}  } \mathscr V ( K' , K'') + O\left(   \frac{ Q}{ ( \log Q)^{  1-\epsilon } }    \right) 
$$
for any $\epsilon>0$. By Lemma \ref{lem:asympforNQ}, the sum over $L_1 , L_2 , \ell_1 , \ell_2 $ is asymptotically $T(1)$ and we have
$$\mathscr C_K (Q)  =        N_0(Q)  
           \sum_{ \substack{ K'\sqcup K'' = K \\ |K'|=2}  } \mathscr V ( K' , K'') + O\left(   \frac{ Q}{ ( \log Q)^{  1-\epsilon } }    \right) .
$$
This concludes the proof of the lemma.

\section{Proof of Proposition \ref{prop:DisHol} - Contribution from Holomorphic forms and Maass forms} \label{sec:dispart}

In this section we bound $\Sigma_{\Dis}$ only, since bounding the contribution of $\Hol(c, p_1....p_{\kappa})$ is similar and easier by the Ramanujan-Petersson bound. 
We recall that in the sum $\Sigma_{\Dis}$, $\sum_{j=1}^{\infty} $ denotes a sum over the spectrum of level $cL_1 \ell_1  \ell_2 $, where $\{u_j\}_{j=1}^\infty$ is the orthonormal basis for the Maass forms of level $cL_1 \ell_1  \ell_2 $ described in \S \ref{subsec:newformoldform}, and $\rho_j(n)$ denotes the Fourier coefficients of $u_j$.  In addition, each $u_j$ is of the form $f^{(g)}$ where $f$ is a Hecke newform of level $M$ with $M |cL_1 \ell_1  \ell_2 $, and $g|\frac {cL_1 \ell_1  \ell_2 }{M}$, and $\tau_f $ is the spectral parameter of $f$, i.e., $ \lambda_j = \frac{1}{4} + \kappa_j^2  = \tau_f (1-\tau_f)   $. 
By \eqref{eqn:BMorthonormalform}
$$ \rho_j ( 1) =  \xi_g(1) \rho_f (1)  . $$

\begin{lem}\label{lem:primesumMaassbdd}
Assume GRH, \eqref{conditions mathbbL} and \eqref{eqn:Pibdd}.  With notation as above, we have
 $$
 \sumsharp_{\substack{p_{k_1},\ldots ,p_{k_\kappa} \\ (\mathfrak p(K), L_1L_2) = 1  } } 
 \rho_j(\mathfrak p(K))\prod_{r = 1}^{\kappa} \left(   \log p_{k_r}  V\bfrac{p_{k_r}}{P_r}  \e{v_r\frac{p_{k_r}}{P_r}} \right) 
	 \ll  |\rho_f(1)| (c Q )^\epsilon (1+|\kappa_j |)^\epsilon Y(\vb )^2
$$ 
for any $\epsilon>0$, where $ Y(\vb)$ is defined in \eqref{def:Yv}.
\end{lem}

\begin{proof}
Let $V_0$ be a smooth function that is compactly supported on $(0, \infty)$ such that $V_0(x)=1$ whenever $ V(x) \neq 0 $. Then we multiply $ \prod_{r=1}^\kappa V_0  \bfrac{p_{k_r}}{P_r}$ to the $\sharp$-sum in the lemma without any changes. By the Mellin inversion we find that
$$ \mathcal W_{v_r}(x) := \e{v_r x} V(x) = \frac{1}{2 \pi  } \int_{-\infty}^\infty  \widetilde{\mathcal W}_{v_r}( i t_r ) x^{-it_r}   dt_r  $$
for each $ r \leq \kappa $, where $\widetilde{\W} $ is the Mellin transform of $\W$. Since  
$ \W_{v_r}^{(l)}(x) \ll (1+|v_r |)^l $ for every $l \geq 0  $ and $ \W$ is compactly supported, we have $
\widetilde{\W}_{v_r} (it_r) \ll   \frac{(1+|v_r|)^2}{(1+|t_r |)^2}
$ by integration by parts. 
Thus, the $\sharp$-sum in the lemma is bounded by
\begin{align*}
\ll & \int_{-\infty}^\infty \cdots \int_{-\infty}^\infty  \left| \ \,\sumsharp_{\substack{p_{k_1},\ldots ,p_{k_\kappa} \\ (\mathfrak p(K), L_1L_2) = 1  } } 
 \rho_j(\mathfrak p(K))\prod_{r = 1}^{\kappa} \frac{ \log p_{k_r}}{ p_{k_r}^{ i t_r }} V_0 \bfrac{p_{k_r}}{P_r}    \right| \prod_{r=1}^\kappa \bfrac{1+|v_r|}{1+|t_r |}^2  dt_1 \cdots dt_\kappa \\
 \ll & \int_{-\infty}^\infty \cdots \int_{-\infty}^\infty |\rho_f(1)| (cL_1 \ell_1  \ell_2   Q )^\epsilon (1+|\tau_f|)^\epsilon Y(\vb)^2 \prod_{r=1}^\kappa \frac{ \log (2+|t_r|) }{(1+|t_r |)^2}   dt_1 \cdots dt_\kappa \\
 \ll & |\rho_f(1)| (c   Q )^\epsilon (1+|\kappa_j |)^\epsilon Y(\vb)^2,
\end{align*}
where the second inequality holds by Lemma
\ref{lem:heckebasisfouriersumbdd}.

\end{proof}

We have $ \elltwo \overline{\rho_j}(\elltwo^2) \ll (cL_1 \ell_1  \ell_2 )^{\epsilon} \elltwo^{1 + \epsilon} |\rho_f(1)| $ by \eqref{eqn:fouriercoeffbdd}. By this inequality,  \eqref{def:SigmaDis}, \eqref{def:DisCtnHol} and Lemmas \ref{lem:hatHbdd} and \ref{lem:primesumMaassbdd}, we find that
\es{\label{eqn:SigmaDis bound1} 
  \Sigma_\Dis \ll   Q^\epsilon \sumd_{P_1, \ldots , P_\kappa} \sum_c c^\epsilon \int_{-\infty}^\infty \sum_{j=1}^\infty \frac{   |   \rho_f (1) |^2 (1+ |\kappa_j |)^\epsilon  }{  \cosh ( \pi \kappa_j )  }   \frac{X |h_{u,+} (\kappa_j )| }{(1+|u|X)^{A_1}}   du }
for any $ A_1 \geq 0 $ and $X $ defined in \eqref{eqn: Xdefinition}. We choose $ A_1 \geq 3 $ for later uses. 
We write
\es{ \label{eqn:Discsum}
\Sigma_{\Dis}  = \Sigma_{\Dis, \tRe} + \Sigma_{\Dis, \tIm},
 } where $ \Sigma_{\Dis, \tRe} $ is the contribution of the real $\kappa_j $,  and $\Sigma_{\Dis, \tIm}$ is the contribution of the imaginary $\kappa_j$ corresponding to exceptional eigenvalues.

For $ \Sigma_{\Dis, \tIm}$, we first have
 \begin{equation}\label{eqn:bound Xhupluskappaj}
 \int_{-\infty}^\infty \frac{X |h_{u,+} (\kappa_j )| }{(1+|u|X)^{A_1}}   du \ll \left(1+ \frac{1}{\sqrt{X}} \right) \min \left\{ X^{k-1} , \frac{1}{ \sqrt{X}} \right\} \ll \frac{\min\{ 1, X^{k-1} \} }{\sqrt{X}} 
 \end{equation}
by Lemma \ref{lem:boundforhu} (2).
Next we need the spectral large sieve bound
\begin{equation}\label{eqn:DI spectral bound}
\sum_{|\kappa_j| \le x} \frac{|\rho_j(1)|^2 }{\cosh(\pi \kappa_j)} \ll x^2
\end{equation}
from Deshoulliers and Iwaniec~\cite[Theorem 2]{DI}.  Hence, we find that
\begin{align*}
    \Sigma_{\Dis,\tIm}  \ll &    Q^\epsilon \sumd_{P_1, \ldots , P_\kappa} \sum_c  c^\epsilon \sum_{\substack{ \kappa_j = ir \\ |r| < 1/2 }   } \frac{   |   \rho_f (1) |^2 (1+ |\kappa_j |)^\epsilon  }{  \cosh ( \pi \kappa_j )  }  \frac{\min\{ 1, X^{k-1}\} }{\sqrt{X}} \\   
     \ll &      Q^\epsilon \sumd_{P_1, \ldots , P_\kappa} \sum_c  c^\epsilon    \frac{\min\{ 1, X^{k-1}\} }{\sqrt{X}}
\end{align*}
by \eqref{eqn:SigmaDis bound1} - \eqref{eqn:DI spectral bound}. By \eqref{conditions mathbbL}, \eqref{eqn: Xdefinition} and \eqref{eqn:Pibdd}, we have
$$ \sumd_{P_1, \ldots , P_\kappa} \sum_c  c^\epsilon    \frac{\min\{ 1, X^{k-1}\} }{\sqrt{X}} \ll \sumd_{P_1, \ldots , P_\kappa} \left(  \frac{ L_1 L_2 \elltwo \sqrt{ P_1 \cdots P_\kappa }}{Q}  \right)^{1+\epsilon} \ll    Q^{1-\frac{\delta}2 +\epsilon} .   $$
Hence, for any $\epsilon>0$, we have
\begin{equation}\label{eqn:bound Sigma Dis Im}
\Sigma_{\Dis,\tIm} \ll      Q^{1- \frac{\delta}2  +\epsilon}.
\end{equation}

For $\Sigma_{\Dis, \tRe}$,
by \eqref{eqn:SigmaDis bound1}, \eqref{eqn:Discsum} and Lemma \ref{lem:boundforhu} (1), we find that
\begin{align*}
     \Sigma_{\Dis, \tRe} \ll &   Q^\epsilon \sumd_{P_1, \ldots , P_\kappa} \sum_c c^\epsilon \int_{-\infty}^\infty \sum_{\kappa_j \in \mathbb{R}} \frac{   |   \rho_f (1) |^2 (1+ |\kappa_j |)^\epsilon  }{  \cosh ( \pi \kappa_j )  }  \\
    &  \times \frac{1}{F^{1-\epsilon}} \left(  \frac{F}{ 1+|\kappa_j|} \right)^{C(j)} \min \left\{ X^{k-1} , \frac{1}{ \sqrt{X}} \right\}     \frac{X (1+|\log X|) }{(1+|u|X)^{A_1}}   du
\end{align*}
for some $ F < (|u|+1)(1+X)$ and for any choice of $ C(j) \geq 0 $. Since $F$ depends on $u$ and $X$, we first estimate the $j$-sum
$$ \Sigma_{\Dis, \tRe,1} := \sum_{\kappa_j \in \mathbb{R}} \frac{   |   \rho_f (1) |^2 (1+ |\kappa_j |)^\epsilon  }{  \cosh ( \pi \kappa_j )  F^{1-\epsilon}} \left(  \frac{F}{ 1+|\kappa_j|} \right)^{C(j)} .$$
Each newform $f$ appears at most $\ll (cL_1 \ell_1  \ell_2 )^{\epsilon}$ times in the above $j$-sum as $ u_j = f^{(g)}$ with $g  | \frac{cL_1 \ell_1  \ell_2 }{M}$. When $ g=1$, we have $ u_j = f^{(1)} = f$ and so $ \rho_j (1) = \rho_f (1)$. 
Then we have
$$ \Sigma_{\Dis, \tRe,1}  \ll (cQ)^\epsilon \sum_{\substack{\kappa_j \in \mathbb{R} \\  u_j = f^{(1)} \textup{ for some } f   }} \frac{   |   \rho_j (1) |^2 (1+ |\kappa_j |)^\epsilon  }{  \cosh ( \pi \kappa_j )  F^{1-\epsilon}} \left(  \frac{F}{ 1+|\kappa_j|} \right)^{C(j)}. $$
We can change the above sum to the sum over all real $\kappa_j$ by adding more positive terms. By splitting the sum dyadically and applying \eqref{eqn:DI spectral bound}, we find that 
\begin{align*} 
  \Sigma_{\Dis, \tRe,1} 
 &\ll (cQ)^{\epsilon} \left(\sum_{|\kappa_j| \leq F}  \frac{|\rho_j(1)|^2(1+|\kappa_j|)^\epsilon}{\cosh(\pi \kappa_j) F^{1- \epsilon}}  +  \sumd_{\ell} \sum_{\ell F < |\kappa_j| \leq 2 \ell F}  \frac{|\rho_j(1)|^2   }{\cosh(\pi \kappa_j)  }   \frac{F^{2+2\epsilon} }{(1 + |\kappa_j|)^3}    \right)  \\
&\ll (cQ)^{\epsilon} F^{1 + \epsilon} \ll (cQ)^{\epsilon} (1+|u|)^{1 + \epsilon}(1+X)^{1 + \epsilon}
\end{align*}
 for any $\epsilon>0$, where we have chosen $ C(j)=0 $ for $ |\kappa_j | \leq F $ and $ C(j) = 3+\epsilon$ otherwise.

 Since
 $$\int_{-\infty}^\infty (1+|u|)^{1 + \epsilon}     \frac{X  }{(1+|u|X)^{A_1}}   du \ll \left( 1 + \frac1X \right)^{1 + \epsilon}    $$
 for $ A_1 \geq 3$, 
  we have
\begin{align*}
     \Sigma_{\Dis, \tRe}  
     \ll &   Q^\epsilon \sumd_{P_1, \ldots , P_\kappa} \sum_c c^\epsilon   \frac{ (1+X)^{2 + 2\epsilon}   }{X^{1+\epsilon}}   \min \left\{ X^{k-1} , \frac{1}{ \sqrt{X}} \right\}      (1+|\log X|)  .
\end{align*}
\new{Since $ k \geq 4 $, the $c$-sum is convergent} and bounded by
$ \left( \frac{L_1L_2 \elltwo \sqrt{P_1 \cdots P_{\kappa}}}{Q} \right)^{1+\epsilon}$ for any $\epsilon>0$.  This may be verified by dividing the sum into two depending on $ c \leq  \frac{4 \pi L_1L_2 \elltwo \sqrt{P_1 \cdots P_{\kappa}}}{Q}$ or not. By \eqref{eqn:Pibdd}, we have
\begin{align*}
     \Sigma_{\Dis, \tRe}  
     \ll & \elltwo^{1+\epsilon} ( L_1  \ell_1   \ell_2  Q)^\epsilon \sumd_{P_1, \ldots , P_\kappa}  \left( \frac{L_1L_2 \elltwo \sqrt{P_1 \cdots P_{\kappa}}}{Q} \right)^{1+\epsilon}  \ll      Q^{1-\frac{\delta}2 +\epsilon}  ,
\end{align*}
which has the same bound as $ \Sigma_{\Dis, \tIm} $ in \eqref{eqn:bound Sigma Dis Im}. Thus, $\Sigma_{\Dis}$ has the same bound as well, which proves the first inequality of the proposition. As we mentioned in the beginning of this section, we omit the proof of the holomorphic case, since it is similar and easier.

\section{Proof of Proposition \ref{prop:Ctn} - Contribution from Eisenstein series}\label{sec:ctn}

Recall that $\Sigma_{\Ctn}$ is defined in \eqref{def:SigmaDis} and \eqref{def:DisCtnHol}. 
Let $\Ctn_0 $ and $ \Ctn_\non$ be the contribution of the trivial character $ \chi_0$ and the nontrivial characters, respectively, in \eqref{def:DisCtnHol}. Then we have
\begin{equation} \label{def:Ctn0}
\Ctn_0 (c, \mathfrak p(K); u) := \frac{1}{4\pi}\sum_{  M | N }  \int_{-\infty}^{\infty}  \rho_{\chi_0 , M, N }(\mathfrak p(K), t) \overline{\rho_{\chi_0 , M, N }(\elltwo^2, t)} h_{u, +}(t) \> dt 
\end{equation}
with $N= cL_1  \ell_1  \ell_2  $ and $\Ctn_\non$ is the same as $ \Ctn$ defined in \eqref{def:DisCtnHol} except for $ \c_\chi \neq 1 $.  Since $ \Ctn= \Ctn_0 + \Ctn_\non$ by definition, we see that
$$ \Sigma_{\Ctn} = \Sigma_{\Ctn_0} + \Sigma_{ \Ctn_ \non} $$
by \eqref{def:SigmaDis}. Then it is easy to see that Proposition \ref{prop:Ctn} follows from the two propositions below.
\begin{prop}\label{prop:Ctn non}
 Assume GRH, \eqref{conditions mathbbL} and \eqref{eqn:Pibdd}.
With notation as above and for any $\epsilon>0$, we have 
$$ \Sigma_{\Ctn_\non} \ll  Q^{\tfrac12 - \tfrac {\delta}{4} + \epsilon} .$$
\end{prop}
\begin{prop}\label{prop:Ctn0} 
 Assume RH, \eqref{conditions mathbbL} and \eqref{eqn:Pibdd}. Let $K$ be a set of positive integers such that $ |K|= \kappa \geq 2 $. With notation as above and for any $\epsilon>0$, we have
$$     \Sigma_{\Ctn_0} =  \frac{ Q ( \log Q)^{\kappa}  \widetilde{\Psi} ( 1)\delta_{\elltwo=1} }{ (-2)^{\kappa}  L_1^2 L_2  \ell_1 \ell_2  }    \sum_{ \substack{ K'\sqcup K'' = K \\ |K'|=2}  } \mathscr V ( K' , K'') + O\left(  \frac{ Q ( \log Q)^{\kappa -1+\epsilon }  }{ L_1 L_2 \elltwo}  \right) , $$
where the function $\mathscr V$ is in \eqref{def:V}.
\end{prop}
We will prove Proposition \ref{prop:Ctn non} in \S \ref{sec:Eisenwithnontrivialchar} and Proposition \ref{prop:Ctn0} in \S \ref{sec:trivialchar off diag main terms} and \S \ref{sec:ctnrescalc}.

\subsection{Proof of Proposition \ref{prop:Ctn non}: bounding the contributions of the non-trivial character} \label{sec:Eisenwithnontrivialchar}

We begin by proving the following lemma.

\begin{lemma}\label{lem: eisensteinprimesumbdd}
    Suppose that $\chi \bmod c_\chi$ is non-trivial such that $c_{\chi}^2 | M |N $ and $ N = cL_1 \ell_1  \ell_2  $. Under the same assumptions as in Proposition \ref{prop:Ctn non}, we have
\begin{multline}\label{eqn:eisensteinprimesum}
       \sumsharp_{\substack{p_{k_1},...,p_{k_\kappa} \\ (\mathfrak p(K), L_1L_2) = 1  } } \left( \prod_{j = 1}^{\kappa} \frac{ \log p_{k_j}}{\sqrt{p_{k_j}}} V\bfrac{p_{k_j}}{P_j}  \e{v_j\frac{p_{k_j}}{P_j}} \right) \rho_{\chi, M, N}(\mathfrak p(K), t) \overline{\rho_{\chi, M, N}(\elltwo^2, t)}   \\
     \ll    \frac{1}{ \sqrt{N}}  (NQ(1+|t|))^\epsilon Y(\vb)^3. 
\end{multline}

\end{lemma}

\begin{proof}
    By \eqref{eqn:Eisensteincoeff} and the fact that $ |\widetilde{C}(\chi, M, t)|=1 $, we find that
    \begin{align*}
        &\rho_{\chi, M, N}(\mathfrak p(K), t) \overline{\rho_{\chi, M, N}(\elltwo^2, t)} \\
        &= \frac{M_1  \zeta_{(M, \frac NM )}(1) \elltwo^{-2it}\overline{\rho'_{\chi, M, N}(\elltwo^2, t)} }{M_2 N |L^{(N)}(1+2it, \chi^2)|^2}   
         \sum_{m_2 |M_2} m_2  \mu\bfrac{M_2}{m_2 } \bar{\chi}(m_2 )\sum_{\substack{n_1 n_2 = \frac{\mathfrak p(K)}{M_1 m_2}  \\ (n_2 , N/M)=1}} \frac{\mathfrak p(K)^{it} \chi(n_2) \bar \chi(n_1)  }{n_2^{2it}}
    \end{align*}
    where
    $M = \c_\chi M_1 M_2$, $(M_2, \c_\chi)=1$ and $M_1 | \c_\chi^{\infty}$. Note that 
       \new{$ \rho'_{\chi, M, N}(\elltwo^2 , t)   \ll  \elltwo^{2+\epsilon}N^\epsilon$} and $ \zeta_{(M, N/M)}(1) \leq \tau(N) \ll N^\epsilon $ for any $\epsilon>0$. We also need a well-known bound
       $$ \frac1{L^{(N)}(1+2it, \chi^2)} = \frac1{L(1+2it, \chi^2 \chi_{N,0} )} \ll N^\epsilon (1+|t|)^\epsilon   $$
       for any $ \epsilon>0 $, where $ \chi_{N,0} $ is the principal character modulo $N$. See \S 11 of \cite{MV} for a proof. Hence, the $\sharp$-sum in \eqref{eqn:eisensteinprimesum} is
    \begin{equation}\label{eqn: bound sharp sum 8.3}
        \ll  \frac{M_1    \elltwo^{2+\epsilon}   (1+|t|)^\epsilon  }{M_2 N^{1-\epsilon}    }            \sum_{ m_2 |M_2} m_2    
          \bigg| ~~ \sumsharp_{\substack{p_{k_1},...,p_{k_\kappa} \\ (\mathfrak p(K), L_1L_2) = 1  \\    \mathfrak p(K) = n_1 n_2 M_1 m_2 \\ (n_2, N/M)=1 } }  \left(\prod_{j = 1}^{\kappa} \frac{ \log p_{k_j}}{p_{k_j}^{\frac12 - it }} V\bfrac{p_{k_j}}{P_j}  \e{v_j\frac{p_{k_j}}{P_j}}  \right) \frac{\bar \chi(n_1) \chi(n_2)  }{n_2^{2it}}\bigg|.
    \end{equation}

   Since $\mathfrak p(K) = n_1 n_2 M_1 m_2 $ and the $p_{k_j}$ are distinct primes, we write $n_1 = \mathfrak p (K_1)$, $ n_2 = \mathfrak p (K_2)$, $M_1 = \mathfrak p (K_3)$ and $ m_2  = \mathfrak p (K_4)$ for $K_1 \sqcup \cdots \sqcup K_4 = K $.
By Lemma \ref{lem:primesumremovedistinct}, the $\sharp$-sum in \eqref{eqn: bound sharp sum 8.3} is
\begin{align*}
    \leq &  \sum_{ K_1 \sqcup \cdots \sqcup K_4 = K}\bigg| ~~ \sumsharp_{\substack{p_{k_1},...,p_{k_\kappa} \\ (\mathfrak p(K), L_1L_2) = 1  \\ (\mathfrak p ( K_2 ) , N/M)=1 \\ \mathfrak p(K_3)= M_1 , \mathfrak p(K_4)= m_2  } }  \prod_{j = 1}^{\kappa} \frac{ \log p_{k_j}}{p_{k_j}^{\frac12 - it }} V\bfrac{p_{k_j}}{P_j}  \e{v_j\frac{p_{k_j}}{P_j}}   \frac{\chi(\mathfrak p(K_2)) \bar \chi(\mathfrak p(K_1))  }{\mathfrak p(K_2)^{2it}}\bigg| \\
    \ll &    \sum_{ K_1 \sqcup \cdots \sqcup K_4 = K} (M_1 m_2)^{- \frac12 +\epsilon}  \bigg| ~~   \sumsharp_{\substack{p_k \textup{ for } k \in K_1 \sqcup K_2  \\ (\mathfrak p(K_1  ), L_1L_2 M_1 m_2 ) = 1 \\ (\mathfrak p(  K_2 ), L_1L_2 M_1 m_2  N/M ) = 1   } }  \prod_{j \in K_1   } \frac{ \bar \chi(p_{k_j})\log p_{k_j}}{p_{k_j}^{\frac12 - it }} V\bfrac{p_{k_j}}{P_j}  \e{v_j\frac{p_{k_j}}{P_j}}  \\
    & \qquad \times \prod_{j \in   K_2 } \frac{ \chi (p_{k_j}) \log p_{k_j}}{p_{k_j}^{\frac12 + it }} V\bfrac{p_{k_j}}{P_j}  \e{v_j\frac{p_{k_j}}{P_j}}  \bigg|\\
     \ll &\frac{1}{ \sqrt{M_1 m_2 } } (NQ(1+|t|))^\epsilon Y(\vb)^3
    \end{align*}
 for any $ \epsilon>0$.
Therefore, by combining the above inequalities, the $\sharp$-sum in \eqref{eqn:eisensteinprimesum} is 
$$
    \ll     \frac{M_1     }{M_2 N  }   
         \sum_{m_2 | M_2} m_2   \frac{1}{ \sqrt{M_1 m_2 } } (NQ(1+|t|))^\epsilon Y(\vb)^3   \ll \frac{1}{ \sqrt{N}}  (NQ(1+|t|))^\epsilon Y(\vb)^3 .
$$       
 \end{proof}

By \eqref{def:SigmaDis} and Lemma \ref{lem: eisensteinprimesumbdd}, we have
\begin{multline*}
 \Sigma_{\Ctn_\non} \ll \sumd_{P_1, \ldots, P_\kappa} \sum_c \int_{-\infty}^\infty \cdots \int_{-\infty}^\infty |\widehat H (u, \vb) |  \\
\times \sum_{ 1 \neq \c_\chi^2 |M|N} \int_{-\infty}^\infty  \frac{1}{ \sqrt{N}}  (NQ(1+|t|))^\epsilon Y(\vb)^3  |h_{u,+}(t)|dt du dv_1 \cdots dv_\kappa
\end{multline*}
for $ N= cL_1  \ell_1   \ell_2  $. 
By Lemma \ref{lem:hatHbdd} and Lemma \ref{lem:boundforhu} (1), we have
\begin{multline*}
 \Sigma_{\Ctn_\non} \ll \sumd_{P_1, \ldots, P_\kappa} \sum_c  \min \left\{ X^{k - 1}, \frac 1{\sqrt X}\right\}  \\
 \times \int_{-\infty}^\infty  \int_{-\infty}^\infty   \frac{X}{(1+|u|X)^A}   
  \frac{1}{ \sqrt{N}}  (NQ(1+|t|))^\epsilon  \frac{1 + |\log X|}{F^{1 - \epsilon}} \left( \frac{F}{1 + |t|}\right)^C    dt du
\end{multline*}
for some $F < (|u|+1)(1+X)$ and $ A,C \geq 0 $ and for any $\epsilon>0$, where $X$ is defined in \eqref{eqn: Xdefinition}. By choosing $A=2$ and $ C= 1+2 \epsilon$, we see that
$$ \Sigma_{\Ctn_\non} \ll Q^\epsilon \sumd_{P_1, \ldots, P_\kappa} \sum_c   \frac{1}{  c^{1/2-\epsilon}} \min \left\{ X^{k - 1}, \frac 1{\sqrt X}\right\}   (1 + |\log X|)  (1+X)^{3\epsilon}  .
$$ 
The $c$-sum is bounded by $ \left(\frac{  L_1 L_2 \elltwo \sqrt{P_1 \cdots P_\kappa}}{Q} \right)^{1/2+\epsilon} $, which may be verified by dividing the sum into two depending on $ c \leq \frac{4 \pi L_1 L_2 \elltwo \sqrt{P_1 \cdots P_\kappa}}{Q}$. Since the $d$-sum is supported on $P_1 \cdots P_\kappa \ll Q^{4-\delta}$, we have
$$ \Sigma_{\Ctn_\non} \ll Q^\epsilon \sumd_{P_1, \ldots, P_\kappa} \left(\frac{  L_1 L_2 \elltwo \sqrt{P_1 \cdots P_\kappa}}{Q} \right)^{1/2+\epsilon}  \ll  Q^{\tfrac12 - \tfrac {\delta}{4} + \epsilon} .
$$

\section{Contribution from the trivial character -- Off-diagonal main terms} \label{sec:trivialchar off diag main terms}

In this section, we start to compute $\Sigma_{\Ctn_0} $ defined in \eqref{def:SigmaDis} and \eqref{def:Ctn0} assuming GRH and \eqref{conditions mathbbL} for $\mathbb L$. By Lemmas \ref{lem:boundforhu} and \ref{lem:hatHbdd}, we can change the order of the sums and the integrals, so that
\begin{multline}  \label{Section9 eqn 1}
 \Sigma_{\Ctn_0 }  = \frac{ i^{-k}}{2}   \sumd_{P_1, \ldots , P_{\kappa}}   \int_{-\infty}^\infty  \sumsharp_{\substack{ p_{k_1} , \ldots , p_{k_\kappa} \\ ( \mathfrak p (K) , L_1L_2 )=1   }}
  \prod_{ j=1}^\kappa \left( \frac{ \log p_{k_j} }{ \sqrt{ p_{k_j} }} V \left( \frac{ p_{k_j}}{P_j} \right)     \right) 
\\
        \times \sum_{c\geq 1}  \sum_{  M | c L_0 }  \rho_{\chi_0 , M, c L_0  }(\mathfrak p(K), t) \overline{\rho_{\chi_0 , M, c L_0 }(\elltwo^2, t)}  \\
         \times \int_{-\infty}^\infty \cdots \int_{-\infty}^\infty    h_{u, +}(t)  \widehat{H}(u, \vb) \prod_{j=1}^\kappa \e{ \frac{p_{k_j}}{P_j} v_j }   dv_1 \cdots \> dv_{\kappa} du  dt      ,
\end{multline}
where
\begin{equation}\label{def:L_0}
L_0 : =  L_1  \ell_1  \ell_2 .
\end{equation}

Next, we apply the Fourier inversion to the $\vb$-integrals and the $u$-integral. Let
$$ \widehat{H_\xi} (\vb) := \int_{-\infty}^\infty \cdots \int_{-\infty}^\infty H ( \xi, \lambb) \e{-\vb \cdot \lambb } d\lambda_1 \cdots d\lambda_\kappa ,$$ 
then we see that $ \widehat{H} ( u, \vb )=  \int_{-\infty}^\infty \widehat{H_\xi} (\vb ) \e{-u\xi} d\xi $. By the Fourier inversion, we have
$ H(\xi , \lambb ) =  \int_{-\infty}^\infty \cdots \int_{-\infty}^\infty  \widehat{H_\xi}(\vb) \e{ \vb \cdot \lambb } dv_1 \cdots dv_\kappa   $. By combining the above and by \eqref{eqn:Hdef} and \eqref{eqn: Xdefinition}, we find that
\begin{align}
     \int_{-\infty}^\infty \cdots \int_{-\infty}^\infty \widehat{H} ( u, \vb )   \prod_{j=1}^\kappa  \e{ \frac{p_{k_j}}{P_j} v_j }  dv_1 \cdots dv_\kappa = \int_{-\infty}^\infty H \left( \xi,  \frac{p_{k_1}}{P_1} , \ldots , \frac{p_{k_\kappa }}{P_\kappa } \right) \e{-u\xi} d\xi  \notag  \\
     =  \left[\prod_{j = 1}^{\kappa} \widehat \Phi_{k_j}\left(\frac{\log p_{k_j}}{\log Q}\right)V_{0}\left( \frac{p_{k_j}}{P_j } \right)\right]   \int_{-\infty}^\infty \Psi\left(\frac{4 \pi L_1 L_2 \elltwo \sqrt{ \mathfrak p (K) }}{\xi cQ} \right)  \e{-u\xi} d\xi . \label{eqn:partial Fourier inversion}
\end{align}
 By \eqref{def:huxi} and Lemma \ref{lem:kuznetsov} we see that
$$  h_{u,+}(t) = \frac{ 2 \pi i }{ \sinh ( \pi t)} \int_0^\infty (J_{2it}(\xi) - J_{-2it} ( \xi) ) J_{k-1} (\xi) W\left( \frac{\xi}{X} \right) \e{u\xi} \frac{ d\xi}{\xi}. $$
Since $W$ is compactly supported, the above integral can be extended to the integral over $ \mathbb{R}$. Hence, by the Fourier inversion, we have
\begin{equation}\label{eqn:Fourier inversion huplus}
\int_{-\infty}^\infty h_{u,+}(t) \e{-u\xi } du =  \frac{ 2 \pi i }{ \sinh ( \pi t)} (J_{2it}(\xi) - J_{-2it} ( \xi) ) J_{k-1} (\xi) W\left( \frac{\xi}{X} \right)  \frac1\xi  .
\end{equation}
Hence, by \eqref{Section9 eqn 1} - \eqref{eqn:Fourier inversion huplus}, we have
\begin{multline*}  
 \Sigma_{\Ctn_0 }  = \frac{ i^{-k}}{2}   \sumd_{P_1, \ldots , P_{\kappa}}   \int_{-\infty}^\infty  \sumsharp_{\substack{ p_{k_1} , \ldots , p_{k_\kappa} \\ ( \mathfrak p (K) , L_1L_2 )=1   }}
  \prod_{ j=1}^\kappa \left( \frac{ \log p_{k_j} }{ \sqrt{ p_{k_j} }}  \widehat \Phi_{k_j}\left(\frac{\log p_{k_j}}{\log Q}\right)   V \left( \frac{ p_{k_j}}{P_j} \right)   \right) 
\\
        \times \sum_{c\geq 1}  \sum_{  M | c L_0 }  \rho_{\chi_0 , M, c L_0 }(\mathfrak p(K), t) \overline{\rho_{\chi_0 , M, c L_0 }(\elltwo^2, t)}  \\
      \times \int_{-\infty}^\infty \Psi\left(\frac{4 \pi L_1 L_2 \elltwo \sqrt{ \mathfrak p (K) }}{\xi cQ} \right)       (J_{2it}(\xi) - J_{-2it} ( \xi) ) J_{k-1} (\xi) W\left( \frac{\xi}{X} \right)  \frac{   d\xi }{\xi}     \frac{ 2 \pi i dt}{ \sinh ( \pi t)}     .
\end{multline*}
\new{Here, the factor $V_0$ has been removed using its definition in the beginning of \S 6}.

By the definition of $W$ in Remark \ref{remark: support 1}, we can change that the $\xi$-integral is over $[0,\infty)$ and remove $W(\xi/X)$. We can also remove the $d$-sum and the factors $V(p_{k_j} / P_j )$ by the fact that $ \sumd_P V(x/P) = 1 $ for $ x \geq 1 $.  Hence, we have 
\begin{multline*}  
 \Sigma_{\Ctn_0 }  = \frac{ i^{-k}}{2}    \int_{-\infty}^\infty  \sumsharp_{\substack{ p_{k_1} , \ldots , p_{k_\kappa} \\ ( \mathfrak p (K) , L_1L_2 )=1   }}
  \prod_{ j=1}^\kappa \left( \frac{ \log p_{k_j} }{ \sqrt{ p_{k_j} }}  \widehat \Phi_{k_j}\left(\frac{\log p_{k_j}}{\log Q}\right)     \right) 
\\
        \times \sum_{c\geq 1}  \sum_{  M | c L_0 }  \rho_{\chi_0 , M, c L_0 }(\mathfrak p(K), t) \overline{\rho_{\chi_0 , M, c L_0 }(\elltwo^2, t)}  \\
      \times \int_{0}^\infty \Psi\left(\frac{4 \pi L_1 L_2 \elltwo \sqrt{ \mathfrak p (K) }}{\xi cQ} \right)       (J_{2it}(\xi) - J_{-2it} ( \xi) ) J_{k-1} (\xi)    \frac{   d\xi }{\xi}     \frac{ 2 \pi i dt}{ \sinh ( \pi t)}     .
\end{multline*}
By the Mellin inversion \eqref{inverseMellinPsi} and changing the order of sums and integrals, which may be justified by Lemma \ref{jbessel},  we have
\begin{align}  
 \Sigma_{\Ctn_0 }  = \frac{ i^{-k}}{2}    \int_{-\infty}^\infty    \int_{0}^\infty  \int_{( -\epsilon_1) } \frac{ \widetilde{\Psi} (s) Q^s }{(4 \pi L_1 L_2 \elltwo)^s }  \sumsharp_{\substack{ p_{k_1} , \ldots , p_{k_\kappa} \\ ( \mathfrak p (K) , L_1L_2 )=1   }}
  \prod_{ j=1}^\kappa \left( \frac{ \log p_{k_j} }{  p_{k_j}^{\frac12(1+s) }}  \widehat \Phi_{k_j}\left(\frac{\log p_{k_j}}{\log Q}\right)     \right)  \notag
\\
\times \widetilde{\varrho}_{L_0 , \mathfrak p(K), \elltwo ; t} (-s )                      (J_{2it}(\xi) - J_{-2it} ( \xi) ) J_{k-1} (\xi)   \xi^{s-1} ds     d\xi     \frac{  dt}{ \sinh ( \pi t)}      \label{Section9 eqn 3}
\end{align}
for $ 0< \epsilon_1  <k-1 $, where
\begin{equation}\label{def:rhoNpKe}
\widetilde{ \varrho}_{L_0, \mathfrak p(K), \elltwo ; t} (s ) := 
\sum_{c\geq 1} \frac{1}{ c^s}   \sum_{  M | c L_0 }  \rho_{\chi_0 , M, c L_0 }(\mathfrak p(K), t) \overline{\rho_{\chi_0 , M, c L_0  }(\elltwo^2, t)}.
\end{equation}

\subsection{Fourier coefficients of Eisenstein series}
We want to show that $\widetilde{ \varrho}_{L_0 , \mathfrak p(K), \elltwo ; t} (s ) $ in  \eqref{def:rhoNpKe} has an analytic continuation. 
It requires the following lemma.
\begin{lemma}\label{lem:eisenteincoeff}
 Let $L_0, \mathfrak p(K), \elltwo$ be as above, then  we have
$$ 	\widetilde{ \varrho}_{L_0, \mathfrak p(K), \elltwo ; t} (s )  
	 =  \frac{ \mathfrak p(K)^{it} }{  L_0  \elltwo^{2it} |\zeta(1+2it)|^2}    
     \sum_{ \substack{ d_1|\mathfrak p (K) \\ d_2|\elltwo^2 } }\frac{\mu(d_1 d_2)}{d_1^{2it} d_2^{-2it}}   \sum_{ \substack{ c_1 | \mathfrak p (K) / d_1  \\  c_2 | \elltwo^2 / d_2  }} \frac{c_2^{2it}}{c_1^{2it}} F_{L_0, d_1  d_2 , \m    }  (s,it ) 
$$ 
for $\tRe (s) >0$, $ t \in \R$ and $ \m  =  \frac{ \mathfrak p(K) \elltwo^2}{c_1c_2d_1d_2}$, where
\begin{equation}\label{def:F in Lemma 9.1}
  F_{\alpha, r  , \m} (s,it)  : =   \sum_{\substack{ c \\  r |c \alpha }}   \frac{  |\zeta_{c\alpha} (1+2it)|^2}{c^{1+s}}  \prod_{ \substack{ p\mid  c\alpha / r  \\ p^2 | c\alpha    } } \bigg(  \frac{ p \delta_{p \nmid \m}   }{p-1}        \bigg)  \prod_{ \substack{ p\mid  c\alpha / r  \\ p^2 \nmid c\alpha    } }  \bigg(  \delta_{p \nmid \m}  + \frac1p       \bigg)  . 
  \end{equation}
\end{lemma}

\begin{proof}
    
Since  $ \c_{\chi}=1$, $ M_1 =1 $ and $ M_2 = M$ in \eqref{eqn:Eisensteincoeff} for $\chi = \chi_0$, we have
\begin{multline}\label{lemma 9.1 eqn 1}
	  \rho_{\chi_0, M, N}(\mathfrak p(K), t) \overline{\rho_{\chi_0, M, N}(\elltwo^2, t)} \\
	 = \frac{  \zeta_{(M, N/M)}(1)  |\zeta_N (1+2it)|^2  }{M N   |\zeta(1+2it)|^2} \frac{\mathfrak p(K)^{it}}{\elltwo^{2it}}  \rho'_{\chi_0, M, N}(\mathfrak p(K), t) \overline{\rho'_{\chi_0, M, N}(\elltwo^2, t)}
\end{multline}
and 
\begin{align*}
  \rho'_{\chi_0, M, N}(n, t)  = &  \sum_{m_1 | M} m_1  \mu\bfrac{M}{m_1}  \sum_{\substack{c_0 | n / m_1  \\ (c_0, N/M)=1}} \frac{1}{c_0^{2it}} \\
  = &   \sum_{m_1 | M} m_1  \mu\bfrac{M}{m_1}  \sum_{ c_0 | n / m_1  } \frac{1}{c_0^{2it}} \sum_{d_1|(c_0, N/M)} \mu (d_1) .
\end{align*}
We replace the condition $ d_1 |c_0$ by a substitution $ c_1 = c_0 / d_1 $. After changing the order of the sums, we find that
\begin{equation}\label{lemma 9.1 eqn 2}
  \rho'_{\chi_0, M, N}(n, t)    =   \sum_{d_1|(n, N/M)} \frac{\mu(d_1)}{d_1^{2it}}   \sum_{ c_1 | n / d_1  } \frac{1}{c_1^{2it}} \sum_{m_1 |(M, n/d_1c_1) } m_1 \mu \left( \frac{M}{m_1} \right). 
  \end{equation}

 By \eqref{lemma 9.1 eqn 1} and \eqref{lemma 9.1 eqn 2}, we have 
 \begin{multline*} 
	  \rho_{\chi_0, M, N}(\mathfrak p(K), t) \overline{\rho_{\chi_0, M, N}(\elltwo^2, t)} 
	 = \frac{  \zeta_{(M, N/M)}(1)  |\zeta_N (1+2it)|^2  }{M N   |\zeta(1+2it)|^2} \frac{\mathfrak p(K)^{it}}{\elltwo^{2it}}  \\ 
     \times \sum_{ \substack{ d_1|(\mathfrak p (K), N/M) \\ d_2|(\elltwo^2, N/M) } }\frac{\mu(d_1)}{d_1^{2it}} \frac{\mu(d_2)}{d_2^{-2it}}   \sum_{ \substack{ c_1 | \mathfrak p (K) / d_1  \\  c_2 | \elltwo^2 / d_2  }} \frac{c_2^{2it}}{c_1^{2it}} \sum_{ \substack{ m_1 |(M, \mathfrak p (K) /d_1c_1)  \\ m_2 |(M, \elltwo^2 /d_2c_2)  } } m_1 m_2 \mu \left( \frac{M}{m_1} \right) \mu \left( \frac{M}{m_2} \right).
\end{multline*}
Since $ ( \mathfrak p(K) , L_2 ) =1 $ and $ \elltwo | L_2^\infty$, we have $ (\mathfrak p (K) , \elltwo^2 ) = ( m_1 , m_2 ) =  (d_1 , d_2 ) = 1 $. Then at most one of $m_1$ and $m_2$ is divisible by $p$  for each prime $ p|M $. If $M$ is not squarefree, then $\mu\left( \frac{M}{m_1}\right)\mu\left( \frac{M}{m_2}\right) = 0$. For a squarefree $M$, we have
 $$   \sum_{m_1 | (M, \m_1 ) } m_1 \mu \left( \frac{M}{m_1} \right)  = \mu(M)  \prod_{ p | ( M, \m_1 ) } (1-p), $$
 so that 
 $$\sum_{ \substack{ m_1 |(M, \mathfrak p (K) /d_1c_1)  \\ m_2 |(M, \elltwo^2 /d_2c_2)  } } m_1 m_2 \mu \left( \frac{M}{m_1} \right) \mu \left( \frac{M}{m_2} \right) = \mu(M)^2   \prod_{ p | ( M, \m ) } (1-p)  $$ 
for $ \m  =  \frac{ \mathfrak p(K) \elltwo^2}{c_1c_2d_1d_2}$. 
By letting $ N = cL_0 $ and changing the order of the sums, we find that
 \begin{multline} \label{Lemma 9.1 equation 1}
	\sum_{ M| cL_0 }  \rho_{\chi_0, M, cL_0}(\mathfrak p(K), t) \overline{\rho_{\chi_0, M, cL_0}(\elltwo^2, t)} 
	   \\ 
      = \frac{  |\zeta_{cL_0} (1+2it)|^2  }{  cL_0   |\zeta(1+2it)|^2} \frac{\mathfrak p(K)^{it}}{\elltwo^{2it}}\sum_{ \substack{ d_1|(\mathfrak p (K), cL_0) \\ d_2|(\elltwo^2, cL_0) } }\frac{\mu(d_1d_2)}{d_1^{2it}d_2^{-2it}}   \sum_{ \substack{ c_1 | \mathfrak p (K) / d_1  \\  c_2 | \elltwo^2 / d_2  }} \frac{c_2^{2it}}{c_1^{2it}}  g \left( cL_0 ; d_1d_2 ,   \m \right),
\end{multline}
where
  $$	g(cL_0 ; d_1d_2 , \mathfrak m ) := \sum_{M|  cL_0 / d_1 d_2 } \frac{ \mu(M)^2    }{M}  \zeta_{ (M, cL_0  / M)}(1)    \prod_{ p | ( M, \m ) } (1-p).$$
By multiplicativity, we find the product formula
\begin{align*}
    g(cL_0 ; d_1d_2 , \mathfrak m ) =&  \prod_{p|  cL_0 / d_1 d_2  } \left(  1 + \frac{ 1 - p\delta_{p|\m }    }{p}  \zeta_{ (p,   cL_0 / p   )   } (1)        \right)  \\
    =&  \prod_{ \substack{ p\mid  cL_0 / d_1d_2  \\ p^2 | cL_0    } } \bigg(  \frac{ p \delta_{p \nmid \m}   }{p-1}        \bigg)  \prod_{ \substack{ p\mid  cL_0 / d_1d_2  \\ p^2 \nmid cL_0    } }  \bigg(  \delta_{p \nmid \m}  + \frac1p       \bigg).
\end{align*}
One can easily complete the proof of the lemma by multiplying $ c^{-s}$ to \eqref{Lemma 9.1 equation 1}, summing it over $c$ and then changing the order of sums.
\end{proof}

Lemma \ref{lem:eisenteincoeff} says that $\widetilde{ \varrho}_{L_0, \mathfrak p(K), \elltwo ; t} (s ) $ is a combination of finitely many $ F_{L_0, d_1d_2, \m }  (s,it)$. Hence, to find an analytic continuation of  $\widetilde{ \varrho}_{L_0, \mathfrak p(K), \elltwo ;t } (s ) $, it is enough to observe $ F_{\alpha, r, \m }  (s,it)$.  
 Here is a product formula for $F_{\alpha, r, \m}(s,it)$ with some conditions applicable to $F_{L_0, d_1d_2, \m} (s, it)$. 

\begin{lem} \label{lem:Fsformula}
	Let $r, \alpha , \m \in \mathbb{N}$, $t \in \R$, $\beta = (r, \alpha)$, $r = r_1 \beta$ and $\alpha = \alpha_1 \beta$. Assume that $r$ is squarefree with $(r, \alpha_1) = 1$  and that every prime $p| \m $ satisfies $p^2 \nmid \alpha_1$. Let $F_{\alpha, r, \m }(s,it) $ be defined in \eqref{def:F in Lemma 9.1}, then we have
 \begin{equation}\label{eqn:Fs formula 1}\begin{split}
 F_{\alpha, r, \m } (s,it)  
	= & \zeta(1+s)   \widetilde{F}(s, it)   \frac{1 }{ r_1^{1 + s}} \prod_{p | r }  \left(   1 - \frac{1}{p^{1+s}}   +  \frac{\delta_{p\nmid \m } }{p^s (p-1)}   \right)   \\    
    & \times \prod_{p ||\alpha_1  }  \left( \frac1p - \frac{1}{p^{2+s}} +   \delta_{p\nmid \m } \left( 1  +   \frac{1 }{p^{s+1} (p-1)}  \right) \right)    \prod_{p^2 |  \alpha_1  }  \left( \frac{p}{p-1}     \right)   \\  
    & \times \prod_{ \substack{  p|\m \\ p \nmid r\alpha_1 }  } \left(  W_p (s, it) - \frac{1}{p^{1+s}} - \frac{1}{p^{2+2s} (p-1) }\right)  \prod_{p | r\alpha_1 \m} W_p ( s, it)^{-1}  ,
    \end{split}\end{equation}
where  
$$  W_p (s, z) =  1 - \frac{1}{p^{1 + 2z}} - \frac{1}{p^{1 - 2z}}   + \frac{1}{p^{2 + 2z + s}} +  \frac{1}{p^{2 - 2z + s}} + \frac{1}{p^{2+ s }}  - \frac{1}{p^{3 + s}} + \frac{1}{p^{3 + 2s}(p-1)} + \frac{1}{p^2}    $$
and
$$ \widetilde{F}  (s,z) =  \prod_{p   }  \left(  \left( 1 - \frac{1}{p^{1+2z}}  \right)^{-1} \left( 1 - \frac{1}{p^{1-2z}}  \right)^{-1}   W_p (s,z) \right)  $$
for $ \tRe( s ) > - \frac12 $ and $ - \frac14 <  \tRe(z) < \frac14$. Moreover, define
\es{\label{def:widetilde{F}_0 (s,z)} \widetilde{F}_0 (s,z)   &:=  \prod_p \bigg( 
 \left( 1 - \frac{1}{p^{1 + 2z}}\right)^{-1} \left( 1 - \frac{1}{p^{1 - 2z}}\right)^{-1}  \\
 &\hskip 1 in\times\left( 1 - \frac{1}{p^{2 + 2z + s}}\right)\left( 1 - \frac{1}{p^{2 - 2z + s}}\right)\left( 1 - \frac{1}{p^{2 + s}}\right) W_p(s, z)  \bigg),}
 then it is convergent when $\tRe(\pm 2z + s) > - \frac{3}{2}$ and  $| \tRe(z) | < \frac 14 $, and bounded when $\tRe(\pm 2z + s) \geq - \frac{3}{2}+\epsilon $ and  $| \tRe(z) | \leq \frac 14 -\epsilon $ for every $\epsilon>0$. It also satisfies
$$ \widetilde{F} (s,z) =  \zeta (2+s) \zeta(2+2z+s) \zeta(2-2z+s) \widetilde{F}_0 (s,z) .  $$
\end{lem}

\begin{proof} 

Since $(r, \alpha) = \beta$, we have $r_1\alpha  = r \alpha_1 = r_1 \alpha_1 \beta $. The condition $ r | \alpha c$ in the definition of $F(s)$ is equivalent to $ r_1 | c $. Then we have 
 $$ F_{\alpha, r, \m } (s,it )  
	=     \sum_{c} \frac{  | \zeta_{r\alpha_1} ( 1+2it)|^2 }{ r_1^{1 + s}c^{1+s}} \prod_{\substack{p|  c \\ (p, r\alpha_1) = 1}}\left|1 - \frac{1}{p^{1+2it}} \right|^{-2}\prod_{\substack{ p|\alpha_1 c  \\  p^2 | r \alpha_1 c    }} \left( \frac{p \delta_{p\nmid \m }   }{p-1}      \right)  \prod_{\substack{ p|\alpha_1 c  \\  p^2 \nmid r \alpha_1 c    }} \left(  \delta_{p\nmid \m }   + \frac1p    \right) . $$
    Since $r$ is squarefree with $(r, \alpha_1) = 1$, we treat four types of primes differently according to $ p^2 |   \alpha_1 $, $ p|| \alpha_1$, $ p | r $ and $ p \nmid r\alpha_1 $.  In case of $ p^2 | \alpha_1 $, $ \delta_{p\nmid \m} =1 $ by an assumption of the lemma. Thus, by multiplicativity, we find that
    \begin{align*}
   F_{\alpha, r, \m } (s,it)  
	= & \frac{  | \zeta_{r\alpha_1} ( 1+2it)|^2 }{ r_1^{1 + s}}     \prod_{p | r }  \left( 1+  \frac{p}{p-1} \delta_{p\nmid \m } \sum_{\ell \geq 1 } \frac{1}{p^{ \ell (1+s) }} \right)    \\    
    & \times \prod_{p ||\alpha_1  }  \left( \delta_{p\nmid \m }   + \frac1p  +   \frac{p}{p-1} \delta_{p\nmid \m } \sum_{\ell \geq 1 } \frac{1}{p^{ \ell (1+s) }} \right)  \prod_{p^2 |  \alpha_1  }  \left( \frac{p}{p-1}   \sum_{\ell \geq 0 } \frac{1}{p^{ \ell (1+s) }} \right)  \\  
    & \times \prod_{p \nmid r\alpha_1   }  \left( 1 +  \left( \left(  \delta_{p\nmid \m }   + \frac1p    \right) \frac{1}{p^{1+s} }  +   \frac{p}{p-1} \delta_{p\nmid \m } \sum_{\ell \geq 2 } \frac{1}{p^{ \ell (1+s) }} \right)\left|1 - \frac{1}{p^{1+2it}} \right|^{-2} \right) .
    \end{align*}
By multiplying $1= \zeta(1+s) \prod_p ( 1- p^{-1-s}) $ and dividing two cases in the last product depending on $ p| \m $, we find that
   \begin{equation}\label{eqn:Fs formula 2} \begin{split}
  F_{\alpha, r, \m } (s,it)  
	= & \frac{  | \zeta_{r\alpha_1} ( 1+2it)|^2 }{ r_1^{1 + s}} \zeta(1+s)      \prod_{p | r }  \left(   1 - \frac{1}{p^{1+s}}   +  \frac{\delta_{p\nmid \m } }{p^s (p-1)}   \right)   \\    
    & \times \prod_{p ||\alpha_1  }  \left( \frac1p - \frac{1}{p^{2+s}} +   \delta_{p\nmid \m } \left( 1  +   \frac{1 }{p^{s+1} (p-1)}  \right) \right)   \prod_{p^2 |  \alpha_1  }  \left( \frac{p}{p-1}     \right)   \\  
    & \times \prod_{ \substack{  p|\m \\ p \nmid r\alpha_1 }  } \left( \left( 1 - \frac{1}{p^{1+s}} \right) \left( 1 + \frac{1}{p^{2+s}} \left|1 - \frac{1}{p^{1+2it}} \right|^{-2} \right)\right) \widetilde{F}_{r \alpha_1 \m} (s,  it ) ,
    \end{split}\end{equation}
 where
    \begin{multline*} 
    \widetilde{F}_{r \alpha_1 \m} (s, z ) := \prod_{p \nmid r\alpha_1  \m  }  \bigg( 1 - \frac{1}{p^{1+s}}  
      +  \frac{\left(1 - \frac{1}{p^{1+s}}  \right) 
      \left( \left(  1   + \frac1p    \right) \frac{1}{p^{1+s} }  +   \frac{p}{p-1}   \sum_{\ell \geq 2 } \frac{1}{p^{ \ell (1+s) }} \right)  }{ \left( 1 - \frac{1}{p^{1+2z}}  \right) \left( 1 - \frac{1}{p^{1-2z}}  \right)   } \bigg)   . 
    \end{multline*}

It is easy to check that 
$$ \widetilde{F}_{r \alpha_1 \m} (s,z) =  \prod_{p \nmid r\alpha_1  \m  }  \left(  \left( 1 - \frac{1}{p^{1+2z}}  \right)^{-1} \left( 1 - \frac{1}{p^{1-2z}}  \right)^{-1}   W_p (s,z) \right) . $$
Since $\widetilde{F} (s,z)$ is the same product as $ \widetilde{F}_{r \alpha_1 \m} (s,z)$ except for the primes $p|r\alpha_1 \m $, we see that
$$  | \zeta_{r\alpha_1} ( 1+2it)|^2 \widetilde{F}_{r \alpha_1 \m} (s,it) =  \widetilde{F}(s, it) \prod_{\substack{p|\m \\ p\nmid r\alpha_1    }} \left| 1 - \frac{1}{p^{1+2it}} \right|^2 \prod_{p | r\alpha_1 \m} W_p ( s, it)^{-1} . $$
By applying the above equation to \eqref{eqn:Fs formula 2}, we obtain \eqref{eqn:Fs formula 1}. It is easy to show the remaining part of the lemma, so we omit the proof.
\end{proof}

\begin{rem} \label{rem:gcdford1d2andalpha} Due to the factor $\mu(L_1L_2) $ and the condition $ \ell_1 | L_1 $ and $ \ell_2 | L_2$ in \eqref{def:CKQ}, $L_1, L_2, \ell_1, \ell_2$ are squarefree and $ (L_1 \ell_1 , \ell_2 ) = 1 $. 
Let $\beta = (d_1d_2, L_0 )$, $d_1d_2 = \beta k_1 $ and $L_0 = \beta \alpha_1 $.  Since $(d_1, L_1L_2) = 1$, $\beta = (d_2, L_0 ) = (d_2, \ell_2)  $, and $\beta | \ell_2$.  Thus $( \alpha_1  , \beta) = (  \ell_2/\beta , \beta) = 1$ and $( \alpha_1 , d_1d_2) = (L_1 \ell_1 \ell_2/ \beta , d_2 ) = (  \ell_2/ \beta , d_2 ) = 1.$ Moreover, every prime $p | \frac{\mathfrak p(K)}{c_1d_1} \frac{\elltwo^2}{c_2d_2}$ satisfies $p^2 \nmid \frac{L_0 }{(\ell_2 , d_2)} = \alpha_1 .$
\end{rem}

\subsection{Combinatorics and computations of sums over primes}\label{Section:combinatorics and computations}
By Remark \ref{rem:gcdford1d2andalpha},  we can apply Lemma \ref{lem:Fsformula} to $ F_{L_0, d_1  d_2 , \m    }  (s, it )$. Recall that $ \mathfrak p (K) $, $L_1$, $L_2 $ are pairwise relatively prime, $ c_1d_1 | \mathfrak p(K)$, $ \ell_1 | L_1$, $ \ell_2 | L_2 $ and $ c_2 d_2 | \elltwo^2 | L_2^\infty$.   In this section, we use the notation that $ a\cdot b $ means the usual multiplication of integers $a$ and $b$ with $ (a,b) = 1 $, which is useful to follow the arguments. 

With $r=d_1 \cdot d_2 $, $\alpha = L_1 \ell_1 \cdot  \ell_2 $, $ r_1 = d_1 \cdot \frac{d_2}{ ( d_2 , \ell_2 )} $, $ \alpha_1 = \ell_1^2 \cdot \frac{ L_1}{\ell_1} \cdot \frac{ \ell_2 }{ ( d_2 , \ell_2 ) } $, $\m_1 = \frac{ \mathfrak p (K)}{c_1 d_1 } $, $\m_2 = \frac{ \elltwo^2}{ c_2 d_2 } $ and $ \m = \m_1 \cdot \m_2    $, we find that 
 \begin{align*}
	F_{L_0, d_1  d_2 , \m    }  (s,it )  =& \zeta(1+s)  \widetilde{F} (s,it)  \frac{(d_2 , \ell_2)^{1+s} }{ (d_1 d_2)^{1 + s}} \prod_{p | d_1 \cdot d_2 }  \left(   1 - \frac{1}{p^{1+s}}   +  \frac{\delta_{p\nmid \m } }{p^s (p-1)}   \right)  \\    
    &  \times \prod_{p  |  \ell_1  }  \left( \frac{p}{p-1}     \right) 
    \prod_{p | \frac{L_1}{\ell_1} \cdot \frac{ \ell_2}{( d_2 , \ell_2 )} }  \left( \frac1p - \frac{1}{p^{2+s}} +   \delta_{p\nmid \m } \left( 1  +   \frac{1 }{p^{s+1} (p-1)}  \right) \right)     \\  
\times    \prod_{ \substack{  p|\m \\ p \nmid d_1 \cdot L_1 \cdot \ell_2 d_2  }  } & \left(  W_p (s, it) - \frac{1}{p^{1+s}} - \frac{1}{p^{2+2s} (p-1) } \right)  \prod_{p | \frac{ \mathfrak p(K) }{c_1} \cdot L_1 \cdot \frac{ \ell_2 \elltwo^2}{ (d_2, \ell_2 ) c_2  } } W_p ( s, it)^{-1}  .
    \end{align*}
 Since $(d_1 , \m ) = 1 $ and $ (L_1 , \m )=1  $, the above products can be split as 
 \begin{align*}
     &      \prod_{p | d_1  }  \left(   1    +  \frac{1}{p^{1+s} (p-1)}   \right)
     \prod_{p |   d_2 }  \left(   1 - \frac{1}{p^{1+s}}   +  \frac{\delta_{p\nmid \m_2 } }{p^s (p-1)}   \right)
      \\    
    & \times \prod_{p  |  \ell_1  }  \left( \frac{p}{p-1}     \right) 
    \prod_{p | \frac{L_1}{\ell_1}  }  \left( 1  + \frac1p  +       \frac{1 }{p^{2+s } (p-1)}    \right)  
    \prod_{p |   \frac{ \ell_2}{( d_2 , \ell_2 )} }  \left( \frac1p - \frac{1}{p^{2+s}} +   \delta_{p\nmid \m_2 } \left( 1  +   \frac{1 }{p^{1+s } (p-1)}  \right) \right)  
     \\  
    & \times \prod_{   p| \m_1    } \left( W_p (s, it) - \frac{1}{p^{1+s}} - \frac{1}{p^{2+2s} (p-1) }\right)    \prod_{ \substack{  p|   \m_2 \\ p \nmid  \ell_2 d_2  }  } \left( W_p (s, it) - \frac{1}{p^{1+s}} - \frac{1}{p^{2+2s} (p-1) }\right) \\
    & \times \prod_{p | \frac{ \mathfrak p(K) }{c_1} \cdot L_1 \cdot \frac{ \ell_2 \elltwo^2}{ (d_2, \ell_2 ) c_2  } } W_p ( s, it)^{-1}  .
    \end{align*}
 We want to separate the primes dividing $ \mathfrak p (K) $ from the others. Define
\begin{align}
 \mathfrak J_1 (s,&  z; \mathbb{L} , c_2, d_2 )   :=  \frac{(d_2 , \ell_2)^{1+s} }{d_2^{1 + s}} 
         \prod_{p |   d_2 }  \left(   1 - \frac{1}{p^{1+s}}   +  \frac{\delta_{p\nmid \m_2 } }{p^s (p-1)}   \right)\prod_{p  |  \ell_1  }  \left( \frac{p}{p-1}     \right) \notag
      \\    
    &  \times 
    \prod_{p | \frac{L_1}{\ell_1}  }  \left( 1  + \frac1p  +       \frac{1 }{p^{2+s } (p-1)}    \right)  
    \prod_{p |   \frac{ \ell_2}{( d_2 , \ell_2 )} }  \left( \frac1p - \frac{1}{p^{2+s}} +   \delta_{p\nmid \m_2 } \left( 1  +   \frac{1 }{p^{1+s } (p-1)}  \right) \right)  \label{def:J1euler} 
     \\  
      & \times \prod_{ \substack{  p|   \m_2 \\ p \nmid  \ell_2 d_2  }  } \left( W_p (s, z) - \frac{1}{p^{1+s}} - \frac{1}{p^{2+2s} (p-1) } \right)\prod_{p |  L_1 \cdot \frac{ \ell_2 \elltwo^2}{ (d_2, \ell_2 ) c_2  } } W_p ( s, z )^{-1} \notag     
 \end{align}
   for $\mathbb{L}$ as in \eqref{def:mathbbL} and $ \m_2 = \frac{\elltwo^2}{c_2d_2}$, then we have
 \begin{equation}\label{eqn:Fs formula 3}\begin{split}
F_{L_0, d_1  d_2 , \m    }  (s,it )   
    = & \zeta(1+s)   \widetilde{F}  (s,it)  \frac{ \mathfrak J_1(s, it ; \mathbb{L}, c_2, d_2 ) }{ d_1^{1 + s}} 
     \prod_{p | d_1  }  \left(   1    +  \frac{1}{p^{1+s} (p-1)}   \right)
      \\    
      &\times \prod_{p | \frac{ \mathfrak p(K) }{c_1}  } W_p ( s, it)^{-1} \prod_{   p| \frac{\mathfrak p(K)}{c_1d_1}   } \left( W_p (s, it) - \frac{1}{p^{1+s}} - \frac{1}{p^{2+2s} (p-1) }\right)   .
    \end{split}\end{equation}
Note that we will need the following special values to compute the off-diagonal main terms. 
\begin{lem} \label{lem:valueatsequal-1} We have
    $$ \widetilde{F}_0 (-1, z) = 1, $$
    $$ W_p(-1, z)^{-1} = 1 - \frac 1p.$$
Moreover,   
$$ \mathfrak J_1 (-1, z; \mathbb{L} , c_2, d_2 ) = \delta_{c_2d_2=\elltwo^2 }. $$
\end{lem}
\begin{proof}
We only compute $\mathfrak J_1(-1, z; \mathbb{L} , c_2, d_2 )$ when $ c_2 d_2 \neq \elltwo^2 $, since the other cases are straightforward from definitions. 
If $ (d_2 , m_2 ) \neq 1 $, then the product over $p|d_2$ in \eqref{def:J1euler} at $ s= -1$ is $0$. If $ ( \frac{ \ell_2}{(d_2 , \ell_2 ) } , \m_2 ) \neq 1$, then the product over $p | \frac{ \ell_2}{(d_2 , \ell_2 ) } $ in \eqref{def:J1euler} at $s=-1$ is $0$. Hence, the remaining case is that $(d_2 \ell_2 , \m_2 ) = 1 $ and $ \m_2 \neq 1 $. In this case, the product over $  p|   \m_2 $ and $ p \nmid  \ell_2 d_2$ in \eqref{def:J1euler} at $ s=-1$ is $0$. Thus, we have $\mathfrak J_1 (-1, z; \mathbb{L}, c_2, d_2 )=0 $ when $ \m_2 \neq 1 $.    
\end{proof}

Define
\es{\label{def:J2euler} 
 & \mathfrak J_2 (s, z; \mathbb{L} ) 
 : =       \sum_{ \substack{ c_2, d_2 \\ c_2 d_2|\elltwo^2}  } \frac{ \mu(  d_2)  (c_2 d_2)^{2z}  }{ \elltwo^{2z} } \frac{ \mathfrak J_1 (s, z; \mathbb{L} , c_2, d_2 )}{  L_1\ell_1 \ell_2    }    }
 for $\mathbb{L}$ as in \eqref{def:mathbbL}, then by \eqref{def:J1euler} we have
\begin{equation}\label{eqn:J2 bound}
  \mathfrak J_2 (-s, z ; \mathbb{L} ) \ll \frac{(L_1 \ell_2 \elltwo )^\epsilon}{L_1 \ell_1 \ell_2 }
  \end{equation}
 for any $\epsilon>0$ and for $ - \frac{10}{ \log Q } \leq \tRe(s) \leq 1 +  \frac{10}{ \log Q} $ and $ |\tRe(z)| \leq  \frac{10}{ \log Q}$.
By Lemma \ref{lem:eisenteincoeff}, \eqref{eqn:Fs formula 3} and \eqref{def:J2euler}, we find that
\begin{multline*}
\widetilde{ \varrho}_{L_0, \mathfrak p(K), \elltwo ; t} (s )  
	 =  \frac{ \zeta(1+s)    \widetilde{F}  (s,it)  }{   |\zeta(1+2it)|^2}    
     \mathfrak J_2 (s, it ; \mathbb{L} )    \sum_{ \substack{c_1, d_1 \\ c_1 d_1|\mathfrak p (K)  }}\frac{\mu(d_1  )\mathfrak p(K)^{it}}{c_1^{2it} d_1^{1+s+2it}  }  \\
 \times    \prod_{p | d_1  } \left( \left(   1    +  \frac{1}{p^{1+s} (p-1)}   \right) \frac{1}{ W_p ( s, it)} \right)  \prod_{   p| \frac{\mathfrak p(K)}{c_1d_1}   } \left( 1 - \left(  \frac{1}{p^{1+s}} + \frac{1}{p^{2+2s} (p-1) } \right) \frac{1}{ W_p (s, it) } \right)   .
    \end{multline*}
We can switch the order of integrals and sums in \eqref{Section9 eqn 3} and integrate the $\xi$-integral first by Lemma \ref{lem:MellinofprodofJbessel}.
Next, applying the above and Lemma \ref{lem:Fsformula} and substituting $ it = z$, we find that
\begin{equation} \label{eqn:SigmaCtn0}
 \Sigma_{\Ctn_0 }  =   \mathfrak M_1 ( -\epsilon_1, 0 ;  Q^s \mathfrak K_{ 1, L_1L_2}(K : s, z ) ) ,
\end{equation}
where
\begin{multline}\label{def:M1cKsz}
 \mathfrak M_1 ( c_s, c_z ; \mathfrak K ( s,z)) :=    \frac{ i^{-k}}{2}    \int_{(c_z)}     \int_{( c_s) } \widetilde{\Psi} (s)   \widetilde{F}_0 (-s,z)  \frac{ \mathfrak J_2 (-s, z ; \mathbb{L} ) }{(4 \pi L_1 L_2 \elltwo)^s }  
\\
   \times \frac{ \zeta(1-s)  \zeta (2-s) \zeta(2+2z-s) \zeta(2-2z-s)  }{    \zeta(1+2z)\zeta(1 - 2z) \sin ( \pi z) }     \mathfrak K ( s, z )
       \\
 \times  2^{s-1} \Gamma ( 1 - s ) ( \mathcal{G}_{2z, k-1} (s) -  \mathcal{G}_{-2z, k-1} (s) )    ds         dz  
\end{multline}
and 
\begin{equation}\label{def:K1Lsz}\begin{split}
\mathfrak K_{ 1, L_1L_2}(K : s, z ):= & \sumsharp_{\substack{ p_{k_1} , \ldots , p_{k_\kappa} \\ ( \mathfrak p (K) , L_1L_2 )=1   }}
  \prod_{ j=1}^\kappa \left( \frac{ \log p_{k_j} }{  p_{k_j}^{\frac12(1+s) }}  \widehat \Phi_{k_j}\left(\frac{\log p_{k_j}}{\log Q}\right)     \right)  \\
         & \times \sum_{  c_1 d_1|\mathfrak p (K)  }\frac{\mu(d_1  )\mathfrak p(K)^z}{c_1^{2z} d_1^{1-s+2z}  }         \prod_{p | d_1  } \left( \left(   1    +  \frac{1}{p^{1-s} (p-1)}   \right) \frac{1}{ W_p ( -s, z)} \right) \\
         & \times \prod_{   p| \frac{\mathfrak p(K)}{c_1d_1}   }  \left( 1 - \left(  \frac{1}{p^{1-s}} + \frac{1}{p^{2-2s} (p-1) } \right) \frac{1}{ W_p (-s, it) } \right) . 
\end{split}\end{equation}

We compute the integrals in \eqref{def:M1cKsz} by residue calculus. We want to shift the $s$-contour to $ \tRe(s) \approx 1 $. 
By Lemma \ref{lem:MellinofprodofJbessel}, we have
\begin{equation} \label{eqn:boundforGammaGoversin} \frac{ \Gamma(1-s) \mathcal G_{\pm 2z, k-1} (s) }{\sin ( \pi z )} \ll ( 1+ | \tIm (s) |)^{ \sigma - \frac52} ( 1+ | \tIm (z)|)^{2 \sigma - 2 }, 
\end{equation}
where $ \sigma =  \tRe(s) $. We shift the $z$-contour to $ \tRe (z) = \frac{1}{ \log Q}$ and the $s$-contour to $ \tRe(s) = \frac{4}{ \log Q} $ to find that
\begin{equation} \label{eqn:SigmaCtn0 1}
 \Sigma_{\Ctn_0 } 
 = \mathfrak M_1 \left( \frac{4}{ \log Q} ,  \frac{1}{ \log Q}  ;         Q^s  \mathfrak K_{1,   L_1L_2}(K: s, z )  \right)   .
\end{equation}
Here, the residue at $ s=0 $ vanishes by \eqref{eqn:diffJbesselat0}. 
Since the $z$-integral may not be convergent when $\sigma \geq 1/2 $ by \eqref{eqn:SigmaCtn0 1}, we cannot move the $s$-contour to the right of $1/2$. To overcome this difficulty, we will find small terms in $  \mathfrak K_{1,   L_1L_2}(K: s, z )$ such that we can shift the $s$-contour to the right of $1/2$ except for the small terms.

By letting $c_1 =\mathfrak p (K_1 ) $, $ d_1 = \mathfrak p (K_2)$ and $  \mathfrak p(K_3) = \mathfrak p (K) / c_1d_1 $ for $K_1 \sqcup K_2 \sqcup K_3 = K $
in \eqref{def:K1Lsz}, we have
\begin{equation} \label{eqn:K1Lsz 1} \begin{split}
\mathfrak K_{ 1, L_1L_2}(K : s, z ) & = \sum_{K_1 \sqcup K_2 \sqcup K_3 = K  } (-1)^{|K_2|} \mathfrak K_{ 2, L_1L_2}(\mathbb K  : s, z ) ,\\
\mathfrak K_{ 2, L_1L_2}(\mathbb K : s, z ) & := \sumsharp_{\substack{ p_{k_1} , \ldots , p_{k_\kappa} \\ ( \mathfrak p (K) , L_1L_2 )=1   }} \prod_{i=1,2,3} \left( \prod_{k_j \in K_i} \mathcal P_{i, k_j} ( p_{k_j} , s, z) \right) 
\end{split}\end{equation}
for $ \mathbb K :=  (  K_1 , K_2 , K_3 )   $, where
\begin{equation}\label{def:mathcal Pksz} \begin{split}
 \mathcal P_{1,k_j} ( p , s, z) & :=    \frac{ \log p  }{  p^{\frac12 + \frac{s}{2} +z }}  \widehat \Phi_{k_j}\left(\frac{\log p}{\log Q}\right)  , \\
 \mathcal  P_{2,k_j} ( p , s, z) & := \frac{ \log p }{  p^{\frac32 - \frac{s}{2} +z }} \left(   1    +  \frac{1}{p^{1-s} (p-1)}   \right) \frac{1}{ W_{p} ( -s, z)}  \widehat \Phi_{k_j} \left(\frac{\log p}{\log Q}\right)  ,        \\
\mathcal  P_{3,k_j} ( p, s, z) & :=     \frac{ \log p }{  p^{\frac12 + \frac{s}{2} -z }}   \left( 1 - \left(  \frac{1}{p^{1-s}} + \frac{1}{p^{2-2s} (p-1) } \right) \frac{1}{ W_p (-s, z) } \right)  \widehat \Phi_{k_j}\left(\frac{\log p}{\log Q}\right)        . 
\end{split}\end{equation}
By Lemma \ref{lem:cs}, we have
\begin{equation}\label{eqn:mathfrak KLsz 1}
    \mathfrak K_{ 2, L_1L_2}(\mathbb K : s, z ) = \sum_{\underline G \in \Pi_K } \mu^* (\underline G ) \prod_{j=1}^\nu \mathfrak P_{G_j ; \mathbb K} ( s, z)  ,
\end{equation}
where
\begin{equation} \label{def:mathfrak PGjKsz}
      \mathfrak P_{G_j ; \mathbb K} ( s, z) := \sum_{\substack{p \\ (p, L_1 L_2 )=1} } \prod_{i=1,2,3} \left( \prod_{k_j \in K_i\cap G_j } \mathcal P_{i, k_j} ( p , s, z) \right).
\end{equation}
By Lemmas \ref{lem:sumprimeDiagonal} and \ref{lemma:prime sum}, it is straightforward to estimate $ \mathfrak P_{G_j ; \mathbb K} ( s, z)  $ as follows. 
\begin{lem}\label{lemma:mathfrak PGKsz asymp 1}
 Assume RH. Let $ s, z $ be complex numbers satisfying $ \tRe (s) =  \frac{4}{ \log Q} $ and $ \tRe (z) = \frac{1}{ \log Q} $.
  If $ |G_j | \geq 2 $, then 
   $$ \mathfrak P_{G_j ; \mathbb K} ( s, z) \ll ( \log Q)^2 . $$
If $G_j =  \{ k_j \} $, then we have 
  $$  \mathfrak P_{G_j ; \mathbb K} ( s, z) =
  \begin{cases}
      \Phi_{k_j} (- i \mathcal U ( \tfrac12 - \tfrac{s}{2} - z )) \log Q + O( (\log Q)^2)    & \textup{ if } k_j \in K_1 , \\
       O(1) & \textup{ if } k_j \in K_2 , \\
       \Phi_{k_j} (- i \mathcal U ( \tfrac12 - \tfrac{s}{2} + z )) \log Q + O( (\log Q)^2)  & \textup{ if } k_j \in K_3  , 
  \end{cases}$$
  where $ \mathcal U = \frac{ \log Q}{ 2 \pi} $.
 \end{lem}
The above lemma readily implies the following corollary.
\begin{cor}\label{lemma:mathfrak K geq 2 is small}
  Assume RH.  Let $ \Pi_{K , E,1 }$ be the set of $ \underline G = \{ G_1 , \ldots , G_\nu \} \in \Pi_K $ such that $ G_j \subset K_2 $ whenever $ |G_j | = 1 $. Then we have
   $$ \sum_{\underline G \in \Pi_{K,E,1} } \mu^* (\underline G ) \prod_{j=1}^\nu \mathfrak P_{G_j ; \mathbb K} ( s, z)   \ll (\log Q)^{|K|}  $$
for $ \tRe (s) =  \frac{4}{ \log Q} $ and $ \tRe (z) = \frac{1}{ \log Q} $.

\end{cor}

 Let $ \Pi'_{K} := \Pi_K \setminus \Pi_{K, E,1 } $. Then $\underline G \in \Pi'_K $ means that there exists $ G_j \in \underline G$ such that $ |G_j | = 1 $ and $ G_j \subset K_1 \sqcup K_3 $. Motivated from Lemma \ref{lemma:mathfrak PGKsz asymp 1}, for $G_j= \{ k_j \} $ we let 
\begin{equation}\label{def:mathfrak PGK0}
   \frac{ \mathfrak P_{G_j ; \mathbb K , 0 } (s,z)}{\log Q}:=
    \begin{cases}
         \Phi_{k_j} (- i \mathcal U ( \tfrac12 - \tfrac{s}{2} - z ))   & \textup{ if } G_j  \subset  K_1 , \\     
          \Phi_{k_j} (- i \mathcal U (- \tfrac12 + \tfrac{s}{2} - z ))   & \textup{ if } G_j  \subset  K_2 , \\     
         \Phi_{k_j} ( -i \mathcal U ( \tfrac12 - \tfrac{s}{2} + z ))   & \textup{ if } G_j  \subset  K_3 ,
    \end{cases}
\end{equation}
\begin{equation}\label{def:mathfrak PGK1}
    \frac{\mathfrak P_{G_j ; \mathbb K , 1 } (s,z)}{\log Q }:=
    \begin{cases}
         -  \int_{-\infty}^0 \widehat \Phi_{k_j} (w) Q^{(\frac12 - \frac{s}{2} - z )w} dw & \textup{ if } G_j  \subset  K_1 , \\     
           -   \int_{-\infty}^0 \widehat \Phi_{k_j} (w) Q^{(-\frac12 + \frac{s}{2} - z )w} dw  & \textup{ if } G_j  \subset  K_2            ,
    \end{cases}
\end{equation}
$$ \frac{ \mathfrak P_{G_j ; \mathbb K , 1 } (s,z)}{\log Q}   :=   - \Phi_{k_j} ( -i \mathcal U (- \tfrac12 +\tfrac{s}{2} + z )) + \int_{-\infty}^0  \widehat \Phi_{k_j} (w)  (Q^{  ( -\frac12 + \frac{s}{2} + z ) w } - Q^{ ( \frac12 - \frac{s}{2} + z )  w}   ) dw   $$
if $G_j \subset  K_3  $, 
and 
\begin{equation}\label{def:mathfrak PGKE}
\mathfrak P_{G_j ; \mathbb K, E} (s,z) := \mathfrak P_{G_j ; \mathbb K} (s,z)  - \mathfrak P_{G_j ; \mathbb K, 0} (s,z).
\end{equation}
Then we obtain the following lemma.
\begin{lem} \label{lemma:K3GKsz}
  Assume RH. Let $ s, z $ be complex numbers satisfying $ \tRe (s) =  \frac{4}{ \log Q} $ and $ \tRe (z) = \frac{1}{ \log Q} $.
Define 
$$  K_\sing := K_\sing (  \underline G, \mathbb K   ) := 
\bigcup_{\substack{ G_j \in \underline{G} \\  |G_j| = 1 , ~ G_j \subset K_1 \sqcup K_3  }}  G_j , $$
then $ K_\sing \neq \emptyset$ for each $ \underline G \in \Pi'_K  $ and  
$$ \mathfrak K_{ 2, L_1L_2}(\mathbb K : s, z ) = \sum_{\underline G \in \Pi'_{K  } } \mu^* (\underline G )  
 \sum_{\substack{ K_{M }\sqcup K_{  E} =  K_\sing   \\  K_{M} \neq \emptyset  }}  \mathfrak K_{ 3, L_1L_2}(\underline G, \mathbb K , K_M , K_E : s, z )
+O(( \log Q)^{2|K|}), $$
where 
\begin{equation}\label{def:K3LGKsz} \begin{split}
   \mathfrak K_{ 3, L_1L_2}(\underline G, \mathbb K , K_M , K_E : s, z )  := &  \prod_{ G_j \subset K_{M}  }   \mathfrak P_{G_j ; \mathbb K,0} ( s, z)   \prod_{ G_j \subset K_{  E }  } \mathfrak P_{G_j ; \mathbb K,E} ( s, z)   \\
 & \times \prod_{ G_j \not \subset K_\sing   } \mathfrak P_{G_j ; \mathbb K} ( s, z) .
\end{split}\end{equation} 
\end{lem}
\begin{proof}
By \eqref{eqn:mathfrak KLsz 1} and Corollary \ref{lemma:mathfrak K geq 2 is small}, we have
$$    \mathfrak K_{ 2, L_1L_2}(\mathbb K : s, z )  = \sum_{\underline G \in \Pi'_{K} } \mu^* (\underline G ) \prod_{j=1}^\nu \mathfrak P_{G_j ; \mathbb K} ( s, z)  +O( ( \log Q)^{|K|}) .$$ 
For each $ \underline G \in \Pi'_{K} $, then we have
\begin{multline*}
\prod_{j=1}^\nu \mathfrak P_{G_j ; \mathbb K} ( s, z) =  \prod_{ G_j \subset   K_\sing   } ( \mathfrak P_{G_j ; \mathbb K,0} ( s, z)  + \mathfrak P_{G_j ; \mathbb K,E} ( s, z)  )\prod_{ G_j \not \subset K_\sing   } \mathfrak P_{G_j ; \mathbb K} ( s, z) \\    
= \sum_{K_{M}\sqcup K_{ E} =  K_\sing   } \mathfrak K_{ 3, L_1L_2}(\underline G, \mathbb K , K_M , K_E : s, z )  .
\end{multline*}
By Lemmas \ref{lemma:prime sum} and \ref{lemma:mathfrak PGKsz asymp 1}, 
  $\mathfrak K_{ 3, L_1L_2}(\underline G, \mathbb K , K_M , K_E : s, z )  = O( (\log Q)^{2|K|}) $ when $ K_{M} = \emptyset$. This proves the lemma.  
\end{proof}

By collecting the above results we have the following lemma.
\begin{lem}\label{lemma:SigmaCtn0GK}
 Assuming RH and \eqref{conditions mathbbL}, we have
$$ \Sigma_{\Ctn_0 } = \sum_{ K_1 \sqcup K_2 \sqcup K_3 = K } (-1)^{|K_2|} \sum_{ \underline G \in \Pi'_K  } \mu^* ( \underline G )  \sum_{\substack{ K_{M }\sqcup K_{  E} =  K_\sing   \\  K_{M} \neq \emptyset  }}  \Sigma_{\Ctn_0 } (\underline G , \mathbb K, K_M , K_E )  + O( Q^\epsilon)  $$
for any $ \epsilon>0$, where
\begin{equation} \label{def:SigmaCtn0GK}
 \Sigma_{\Ctn_0 } (\underline G , \mathbb K, K_M , K_E )  
 : = \mathfrak M_1 \left( \frac{4}{ \log Q} ,  \frac{1}{ \log Q}  ;         Q^s  \mathfrak K_{3,   L_1L_2}(  \underline G , \mathbb K , K_M , K_E  : s, z )  \right)    
\end{equation}
with $\mathfrak M_1 $ defined in \eqref{def:M1cKsz}.
\end{lem}

\begin{proof}
    By \eqref{eqn:J2 bound}, \eqref{eqn:SigmaCtn0 1}, \eqref{eqn:K1Lsz 1} and Lemma \ref{lemma:K3GKsz}, it is enough to show that
    \begin{multline*}  
    \int_{(\frac{1}{ \log Q})}     \int_{( \frac{4}{ \log Q} ) }     \frac{ | \widetilde{\Psi} (s)     \zeta(1-s) |   }{  |  \zeta(1+2z)\zeta(1 - 2z) \sin ( \pi z) | }    ( \log Q)^{2|K|}
       \\
   \times | \Gamma ( 1 - s ) ( \mathcal{G}_{2z, k-1} (s) -  \mathcal{G}_{-2z, k-1} (s) )  | |  ds |     |   dz  |    \ll Q^\epsilon
\end{multline*}
for any $\epsilon>0$. This can be easily justified by the following inequalities.

 By repeated integration by parts, we have $ \widetilde{\Psi }(s)  \ll \frac{1}{  |s| ( 1+|s|)^A } $
 for any $ A \geq 0 $.  We also have an upper bound for 
  $\big| \frac{\Gamma ( 1 - s )  \mathcal{G}_{\pm 2z, k-1} (s)}{\sin (\pi z)} \big| $
 in \eqref{eqn:boundforGammaGoversin}. Together with   well-known bounds for the Riemann zeta function near $\tRe(s) = 1 $, these bounds are sufficient to justify the lemma. 
\end{proof}

By Lemma \ref{lemma:SigmaCtn0GK}, we next compute the integral 
$ \Sigma_{\Ctn_0 } (\underline G , \mathbb K , K_M , K_E ) $ 
for $\underline G \in \Pi'_{K}$ and $K_M \neq \emptyset$. 
Each $ \mathfrak K_{ 3, L_1L_2}(\underline G, \mathbb K , K_M , K_E  : s, z )$ in \eqref{def:K3LGKsz} has a factor $ \mathfrak P_{G_j ; \mathbb K,0} ( s, z) $ defined in \eqref{def:mathfrak PGK0}, which is rapidly decreasing as $|\tIm ( \frac{s}{2} \pm z ) | \to \infty$ on given vertical lines of $s$ and $z$. Since $\widetilde \Psi (s) $ is also rapidly decreasing as $|\tIm (s)| \to \infty$, the convergence issue has been resolved. We can now move the $s$-contour in \eqref{def:SigmaCtn0GK} to $\tRe (s) = 1- \frac{4}{\log Q}  $ and see that 
\begin{equation} \label{eqn:SigmaCtn0GK 1}
 \Sigma_{\Ctn_0 } (\underline G , \mathbb K, K_M , K_E)
 = \mathfrak M_1 \left(  1- \frac{4}{ \log Q} ,   \frac{1}{ \log Q}  ;      Q^s \mathfrak K_{3,   L_1L_2}(  \underline G , \mathbb K , K_M , K_E : s, z ) \right).
\end{equation}
To estimate $\mathfrak K_{ 3, L_1L_2}(\underline G, \mathbb K ,  K_M , K_E: s, z )$, we find asymptotic formulas for $ \mathfrak P_{G_j ; \mathbb K} ( s, z)  $. 
\begin{lem}\label{lemma:mathfrak PGKsz asymp 2}
 Assume RH. Let $ s, z $ be complex numbers that satisfy $ \tRe (s) = 1- \frac{4}{ \log Q} $ and $ \tRe (z) = \frac{1}{ \log Q} $.
  If $ |G_j | \geq 2 $, then 
   $$ \mathfrak P_{G_j ; \mathbb K} ( s, z) \ll 1 . $$
If $ |G_j | = 1$, then   
  $$ \mathfrak P_{G_j ; \mathbb K , 2 } (s,z) :=   \mathfrak P_{G_j ; \mathbb K, E } ( s, z) -
 \mathfrak P_{G_j ; \mathbb K , 1 } (s,z)  \ll    1+|s|+|z|   . $$ 
 Moreover, $  \mathfrak P_{G_j ; \mathbb K, 0 } ( s, z)\ll \log Q  $ and $  \mathfrak P_{G_j ; \mathbb K, E } ( s, z)\ll \log Q + |s| + |z|$, when $ |G_j | = 1 $. 
  
\end{lem}
The proof is straightforward from Lemmas  \ref{lemma:prime sum} and \ref{lem:FourierResult}. 
The above lemma readily implies the following corollary.
\begin{cor}\label{lemma:mathfrak K geq 2 is small 2}
 Assume RH. Let $\underline G \in \Pi'_K  $ and assume that $ |G_j | \geq 2 $ for some $ G_j \in \underline G $.  Then we have
   $$   \mathfrak K_{ 3, L_1L_2}(\underline G, \mathbb K , K_M , K_E : s, z ) \ll  ( \log Q + |s|+|z| )^{|K|-2}
 \prod_{ G_j \subset K_{M }  } \frac{ | \mathfrak P_{G_j ; \mathbb K,0} ( s, z)|}{\log Q}     $$
for $ \tRe (s) = 1-  \frac{4}{ \log Q} $ and $ \tRe (z) = \frac{1}{ \log Q} $.  
\end{cor}

 Let $ \Pi_{K , E, 2 }$ be the set of $ \underline G = \{ G_1 , \ldots , G_\nu \} \in \Pi'_K  $ such that $ |G_j |\geq 2 $ for some $ j \leq \nu $, then $\Pi'_K  = \{ \pi_{K,1} \} \sqcup \Pi_{K,E,2} $, where
$ \pi_{K,1} = \{ \{k \} | k \in K \}  $. To show that the contribution of $\underline G \in \Pi_{K,E,2} $ is small, we need a technical lemma as follows.
\begin{lem}\label{lemma:M1 bound}
Let $ \Phi$ be an even Schwartz function with its Fourier transform compactly supported. Define
\begin{multline}\label{def:M1absolute}
 |\mathfrak M_1|  ( c_s, c_z ; \mathfrak K ( s,z)) :=    \frac{1}{( L_1 L_2 \elltwo )^{c_s} }  \int_{(c_z)}     \int_{( c_s) } |\widetilde{\Psi} (s)   \widetilde{F}_0 (-s,z)   \mathfrak J_2 (-s, z ;\mathbb{L} ) |   
\\
   \times \frac{ | \zeta(1-s)  \zeta (2-s) \zeta(2+2z-s) \zeta(2-2z-s) |  }{    | \zeta(1+2z)\zeta(1 - 2z) \sin ( \pi z) |  }     \mathfrak K ( s, z )
       \\
\times    | \Phi ( -i \mathcal U ( \tfrac12 - \tfrac{s}{2} \pm z )) \Gamma ( 1 - s ) ( \mathcal{G}_{2z, k-1} (s) -  \mathcal{G}_{-2z, k-1} (s) )  | |  ds |      |   dz | .
\end{multline}
Let $ c_s = 1 - \frac{4}{ \log Q}$, $ c_z = \frac{1}{ \log Q}$ and $ A_0 \geq 0 $, then we have 
$$ |\mathfrak M_1|  ( c_s, c_z ;  1+|s|^{A_0} + |z|^{A_0}) \ll   \frac{(L_1 \ell_2 \elltwo )^\epsilon}{ ( L_1 L_2 \elltwo )^{c_s} L_1 \ell_1 \ell_2 }  
    \log \log Q   $$
 for any $\epsilon>0 $.
\end{lem}

\begin{proof}
Let $\tRe(s) = 1 - \frac{4}{ \log Q} $ and $ \tRe (z) = \frac{1}{ \log Q}$. By repeated integration by parts, we have $ \widetilde \Psi(s) \ll   |s|^{-A}  $. Also, by Lemma \ref{lem:FourierResult},  $\Phi ( -i \mathcal U ( \tfrac12 - \tfrac{s}{2} \pm z )) \ll ( 1 + | \tIm(s  \mp 2z )| \log Q)^{- A}  $ for any $A \geq 0 $.
By these bounds, \eqref{eqn:J2 bound} and Lemmas \ref{lem:MellinofprodofJbessel} and \ref{lem:Fsformula}, it suffices to show that 
\begin{align*}
  \mathfrak I_1  :=&   \int_{(c_z)}     \int_{( c_s) } 
   \frac{( 1+|s|^{A_0} + |z|^{A_0} )  |  \zeta (2-s) \zeta(2+2z-s) \zeta(2-2z-s) |  e^{\pi | \tIm (z) |}}{  |s|^A ( 1 + | \tIm(s  \mp 2z )| \log Q)^{A}  | \zeta(1+2z)\zeta(1 - 2z) \sin ( \pi z) |   }     
       |  ds |      |   dz | \\
       & \ll  \log \log Q
       \end{align*}
for some $ A >0$. We only consider the minus case of $  \tIm(s  \mp 2z ) $, since the other case holds by the same way. 

 By the bounds
$$ \frac{ e^{\pi|\tIm(z)| }  }{  |  \zeta(1+2z)\zeta(1 - 2z)  \sin( \pi z) | } \ll \frac{1}{ \log Q} + |y| , $$
 $$\zeta( 2-s) \ll    \frac{1}{  \frac{1}{ \log Q} + |t|  } + |t|^\epsilon   , \qquad  \zeta(2\pm 2z - s ) \ll  \frac{1}{  \frac{1}{ \log Q} + |t \mp 2y |  } + |t\mp 2y|^\epsilon   $$
for $ z= \frac{1}{ \log Q} + iy $ and $ s = 1- \frac{4}{\log Q} + it $, we find that
\begin{multline*} 
  \mathfrak I_1 \ll  \int_{-\infty }^\infty \int_{-\infty }^\infty   \frac{  ( 1 + |t|+|y| )^{A_0 } }{ (1+|t|)^A ( 1 + | t - 2y| \log Q)^A   }  \left( \frac{1}{ \log Q} + |y| \right) \\
         \times  \left( \frac{1}{  \frac{1}{ \log Q} + |t|  } + |t|^\epsilon   \right)
          \left(  \frac{1}{  \frac{1}{ \log Q} + |t - 2y |  } + |t - 2y|^\epsilon \right)
          \left(  \frac{1}{  \frac{1}{ \log Q} + |t + 2y |  } + |t + 2y|^\epsilon \right)
           dt dy     
\end{multline*}
for any $ \epsilon >0$. By substituting $ t,y$ to $ \frac{t}{ \log Q}$, $\frac{y}{\log Q} $, we have
\begin{align*} 
\mathfrak I_1 \ll &   \int_{-\infty }^\infty \int_{-\infty }^\infty   \frac{  ( 1 +  \frac{|t|+|y| }{ \log Q  })^{A_0} (1 +|y|)}{ (1+\frac{|t|}{\log Q})^A ( 1 + | t - 2y|  )^A   }              \left( \frac{1}{ 1 + |t|  } + \frac{|t|^\epsilon }{(\log Q)^{1+\epsilon} }  \right)\\
   & \qquad  \times      \left(  \frac{1}{  1 + |t - 2y |  } + \frac{ |t - 2y|^\epsilon}{( \log Q)^{1+\epsilon} } \right)
          \left(  \frac{1 }{ 1+ |t + 2y |  } + \frac{ |t + 2y|^\epsilon}{ (\log Q)^{1+\epsilon}} \right)
            dt dy .
\end{align*}
Because of the factor $(1+\frac{|t|}{\log Q})^A ( 1 + | t - 2y|  )^A$ in the denominator with any choice of $A>0$, we expect that the main contribution comes from the region $|t| \ll \log Q$ and $|t - 2y| \ll 1 $. The integral over this region is bounded by
\est{  \int_{|t| < \log Q} \int_{|t - 2y | \ll 1} (1 + |y|)\left( \frac{1}{1 + |t|} \right) \left( \frac{1}{1 + |t + 2y|} \right) \> dy \> dt \ll \log \log Q. }
Here, the inequality $ 1+|y| \leq ( 1+ |t-2y |)(1+|t+2y|)$ may be useful. 
\end{proof}

Now we are ready to show that the contribution of $\underline G \in \Pi_{K,E,2} $ is small, so that only $ \pi_{K,1}$ contributes.
\begin{lem}\label{lemma:SigmaCtn0piK1}
    Assuming RH and \eqref{conditions mathbbL}, we have
    \begin{align*}
      \Sigma_{\Ctn_0} =&   \sum_{ K_1 \sqcup K_2 \sqcup K_3 = K } (-1)^{|K_2|}   \sum_{\substack{ K_{M }\sqcup K_{  E} =  K_\sing   \\  K_{M} \neq \emptyset  }}  \Sigma_{\Ctn_0 } ( \pi_{K,1} , \mathbb K, K_M , K_E )  \\
      & + O \left(   \frac{Q( \log Q)^{|K|-2} \log \log Q}{ L_1^{2-\epsilon} L_2 \ell_1 (\ell_2 \elltwo )^{1-\epsilon}    }   \right)
      \end{align*}      
    for any $\epsilon>0 $, where $\mathbb K = ( K_1, K_2, K_3 ) $ and $K_\sing = K_1 \sqcup K_3 $.
\end{lem}

\begin{proof}

By Lemma \ref{lemma:SigmaCtn0GK}, \eqref{eqn:SigmaCtn0GK 1} and the definition of $\Pi_{K,E,2} $ above Lemma \ref{lemma:M1 bound}, it suffices to show that
$$\mathfrak M_1 \left(  1- \frac{4}{ \log Q} ,   \frac{1}{ \log Q}  ;    Q^s   \mathfrak K_{3,   L_1L_2}(  \underline G , \mathbb K , K_M , K_E : s, z ) \right) \ll   \frac{Q( \log Q)^{|K|-2} \log \log Q}{ L_1^{2-\epsilon} L_2 \ell_1 (\ell_2 \elltwo  )^{1-\epsilon}    } $$
for $ \underline G \in \Pi_{K,E,2}$, $K_M \sqcup K_E = K_1 \sqcup K_3 $ and $ K_M \neq \emptyset$. It follows from \eqref{def:M1cKsz}, Corollary \ref{lemma:mathfrak K geq 2 is small 2} and Lemma \ref{lemma:M1 bound}. \end{proof}

To compute $ \Sigma_{\Ctn_0 } ( \pi_{K,1} , \mathbb K, K_M , K_E ) $ in Lemma \ref{lemma:SigmaCtn0piK1}, we see that
\begin{equation}\label{eqn:K3LKsz} \begin{split}
\mathfrak K_{ 3, L_1L_2}& ( \pi_{K,1} , \mathbb K , K_M , K_E : s, z )\\
   = &\prod_{ k_j  \in  K_{M}  }   \mathfrak P_{ \{ k_j \} ; \mathbb K,0} ( s, z)   \prod_{    k_j   \in K_{E }  } ( \mathfrak P_{ \{ k_j \} ; \mathbb K,1} ( s, z)+ O( 1+|s|+|z|)  )\\
   &  \times \prod_{   k_j \in  K_2   } (  \mathfrak P_{ \{ k_j \} ; \mathbb K,0} ( s, z) + \mathfrak P_{ \{ k_j \} ; \mathbb K,1} ( s, z)  +    O( 1+|s|+|z|)  ) 
\end{split}\end{equation}
by \eqref{def:K3LGKsz} and Lemma \ref{lemma:mathfrak PGKsz asymp 2}.
Since $ \mathfrak P_{ \{ k_j \} ; \mathbb K,0} ( s, z)\ll \log Q  $ and $ \mathfrak P_{ \{ k_j \} ; \mathbb K,1} ( s, z) \ll \log Q $, we expect that the contribution of the $O$-terms in \eqref{eqn:K3LKsz} is small. This is justified in the next lemma.
\begin{lem}\label{lemma:SigmaCtn0 SigmaCtn1}
 Assuming RH and \eqref{conditions mathbbL}, we have
 \begin{equation}\label{eqn:SigmaCtn0 SigmaCtn1}
\Sigma_{\Ctn_0 } =  \mathfrak M_1  \left(   1- \frac{4}{ \log Q},  \frac{1}{ \log Q};     Q^s   \mathfrak K_{ 4, L_1L_2}(   s, z )\right)   + O \left( \frac{Q ( \log Q)^{|K|-1}  \log \log Q }{ L_1^{2-\epsilon} L_2 \ell_1 (\ell_2 \elltwo )^{1-\epsilon} } \right)
\end{equation}
 for any $\epsilon>0 $,
 where   $ \mathfrak M_1$ is defined in \eqref{def:M1cKsz} and 
\begin{equation}  \label{eqn:K4Lsz 1}\begin{split}
     & \mathfrak K_{ 4, L_1L_2}(   s, z )    (\log Q)^{-|K|} \\
     = & \prod_{ k_j \in K}    \int_0^\infty  \widehat \Phi_{k_j} (w)  (   Q^{(\frac12 - \frac{s}{2} - z )w} - Q^{(-\frac12  + \frac{s}{2} - z )w} +   Q^{ ( \frac12 - \frac{s}{2} + z )  w}  - Q^{  ( -\frac12 + \frac{s}{2} + z ) w } )    dw \\
  & -    \prod_{    k_j \in  K     }    \bigg(      -2   \int_0^\infty    \widehat \Phi_{k_j} (w) ( Q^{( -\frac12 + \frac{s}{2} + z )w}  + Q^{(-\frac12 + \frac{s}{2} - z )w} ) dw  \bigg).
 \end{split}   \end{equation}
\end{lem}
\begin{proof}
Define
\begin{equation}\label{def:K4LKsz} \begin{split}
  \mathfrak K_{ 4, L_1L_2}( \mathbb K , K_M , K_E : s, z )  := &\prod_{ k_j  \in  K_{M }  }   \mathfrak P_{ \{ k_j \} ; \mathbb K,0} ( s, z)   \prod_{   k_j  \in K_{E}  }  \mathfrak P_{ \{ k_j \} ; \mathbb K,1} ( s, z)  \\
   &\times  \prod_{    k_j \in  K_2   } (  \mathfrak P_{ \{ k_j \} ; \mathbb K,0} ( s, z) + \mathfrak P_{ \{ k_j \} ; \mathbb K,1} ( s, z)     )
\end{split}\end{equation}
 for $\mathbb K = (K_1 , K_2 , K_3) $.
    By expanding the product \eqref{eqn:K3LKsz}, we find that
\begin{multline*}
  \mathfrak K_{ 3, L_1L_2}( \pi_{K,1} , \mathbb K , K_M , K_E : s, z ) -   \mathfrak K_{ 4, L_1L_2}( \mathbb K , K_M , K_E : s, z )  \\
  \ll \left(  (1+|s|+|z|) \log Q \right)^{|K|-1} \prod_{ k_j  \in  K_{M}  }  \frac{|  \mathfrak P_{ \{ k_j \} ; \mathbb K,0} ( s, z)|}{\log Q } ,
  \end{multline*}
  which is similar to the bound in Corollary \ref{lemma:mathfrak K geq 2 is small 2}, but essentially $\log Q $ larger. By Lemma \ref{lemma:M1 bound} and the above inequality, \eqref{eqn:SigmaCtn0 SigmaCtn1} holds with 
 \begin{equation}\label{def:K4Lsz}
     \mathfrak K_{ 4, L_1L_2}(    s, z )  := \sum_{ K_1 \sqcup K_2 \sqcup K_3 = K } (-1)^{|K_2|}     \sum_{\substack{ K_{M}\sqcup K_{E} =  K_\sing   \\  K_{M} \neq \emptyset  }}  \mathfrak K_{ 4, L_1L_2}(  \mathbb K , K_M , K_E : s, z ) .
 \end{equation} 

Next, we compute $\mathfrak K_{ 4, L_1L_2}(    s, z ) $. The inner sum in \eqref{def:K4Lsz} equals to 
\begin{align*}
      \sum_{  K_{M}\sqcup K_{E} =  K_1 \sqcup K_3   }  & \mathfrak K_{ 4, L_1L_2}(  \mathbb K , K_M , K_E : s, z ) -      \mathfrak K_{ 4, L_1L_2}(  \mathbb K ,  \emptyset,  K_\sing : s, z ) \\
 = & \prod_{    k_j \in  K    } (  \mathfrak P_{ \{ k_j \} ; \mathbb K,0} ( s, z) + \mathfrak P_{ \{ k_j \} ; \mathbb K,1} ( s, z)     ) \\
 & - \prod_{   k_j  \in K_1\sqcup K_3   }  \mathfrak P_{ \{ k_j \} ; \mathbb K,1} ( s, z)   \prod_{    k_j \in  K_2    } (  \mathfrak P_{ \{ k_j \} ; \mathbb K,0} ( s, z) + \mathfrak P_{ \{ k_j \} ; \mathbb K,1} ( s, z)     ).
\end{align*}
By the definitions \eqref{def:mathfrak PGK0} and \eqref{def:mathfrak PGK1} and the fact that each $\Phi_{k_j} $ is even, we find that 
\begin{align*}
     & \sum_{ K_1 \sqcup K_2 \sqcup K_3 = K } \frac{ (-1)^{|K_2|} }{(\log Q)^{|K|}} \prod_{    k_j \in  K    }   (  \mathfrak P_{ \{ k_j \} ; \mathbb K,0} ( s, z) + \mathfrak P_{ \{ k_j \} ; \mathbb K,1} ( s, z)     ) \\
  =  &  \sum_{ K_1 \sqcup K_2 \sqcup K_3 = K }       \prod_{    k_j \in  K_1    }     \int_0^\infty  \widehat \Phi_{k_j} (w) Q^{(\frac12 - \frac{s}{2} - z )w} dw     \prod_{    k_j \in  K_2    }  \left( -  \int_0^\infty  \widehat \Phi_{k_j} (w) Q^{(-\frac12 + \frac{s}{2} - z )w} dw   \right) \\
 &  \times \prod_{    k_j \in  K_3    }      \int_0^{\infty}  \widehat \Phi_{k_j} (w)  (   Q^{ ( \frac12 - \frac{s}{2} + z )  w}  - Q^{  ( -\frac12 + \frac{s}{2} + z ) w }   ) dw   \\
 =  & \prod_{ k_j \in K}    \int_0^\infty  \widehat \Phi_{k_j} (w)  (   Q^{(\frac12 - \frac{s}{2} - z )w} - Q^{(-\frac12  + \frac{s}{2} - z )w} +   Q^{ ( \frac12 - \frac{s}{2} + z )  w}  - Q^{  ( -\frac12 + \frac{s}{2} + z ) w } )    dw     
\end{align*}
and similarly
\begin{align*}
    &\sum_{ K_1 \sqcup K_2 \sqcup K_3 = K }    \frac{(-1)^{|K_2|} }{(\log Q)^{|K|}}     \prod_{   k_j  \in K_1\sqcup K_3   }    \mathfrak P_{ \{ k_j \} ; \mathbb K,1} ( s, z)   \prod_{    k_j \in  K_2    } (  \mathfrak P_{ \{ k_j \} ; \mathbb K,0} ( s, z) + \mathfrak P_{ \{ k_j \} ; \mathbb K,1} ( s, z)     ) \\
  = &   \sum_{ K_1 \sqcup K_2 \sqcup K_3 = K }        \prod_{    k_j \in  K_1    }   \left(   -    \int_{-\infty}^0 \widehat \Phi_{k_j} (w) Q^{(\frac12 - \frac{s}{2} - z )w} dw     \right) \prod_{    k_j \in  K_2    }   \left( -   \int_0^\infty   \widehat \Phi_{k_j} (w) Q^{(-\frac12 + \frac{s}{2} - z )w} dw   \right) \\
  &  \times
  \prod_{    k_j \in  K_3    }   \left(  -     \int_{-\infty}^0 \widehat \Phi_{k_j} (w) Q^{(\frac12 - \frac{s}{2} +z )w} dw - \int_0^\infty  \widehat \Phi_{k_j} (w) Q^{(-\frac12 + \frac{s}{2} + z )w} dw   \right) \\
  =          & \prod_{    k_j \in  K     }   \bigg(   - 2  \int_0^\infty  \widehat \Phi_{k_j} (w) Q^{(-\frac12 + \frac{s}{2} + z )w} dw       - 2 \int_0^\infty   \widehat \Phi_{k_j} (w) Q^{(-\frac12 + \frac{s}{2} - z )w} dw    \bigg) .
  \end{align*}
 These equations imply \eqref{eqn:K4Lsz 1}.
 \end{proof}

Next we perform the change of variables $ u = - \frac12 + \frac{s}{2} + z $ and $ v =  - \frac12 + \frac{s}{2} - z $, or equivalently $ s = u+v+1 $ and $ z = \frac12 (u-v )$, then we have
\begin{multline} \label{eqn:SigmaCtn1 SigmaCtn2 2}
 \mathfrak M_1  \left(   1- \frac{4}{ \log Q},  \frac{1}{ \log Q};      Q^s  \mathfrak K_{ 4, L_1L_2}(   s, z )\right)  \\
 = \frac{ Q ( \log Q)^{|K|}}{ L_1 L_2 e}   \mathfrak M_2  \left(    \frac{-1}{ \log Q},    \frac{ - 3}{ \log Q};     Q^{u+v}   \mathfrak K_{ 5, L_1L_2}(   u, v )\right)   
\end{multline}
and
 \begin{multline} \label{def:M2cuv}
\mathfrak M_2  \left( c_u,  c_v;        \mathfrak K (   u, v )\right)      : = \frac{ i^{-k}}{8\pi}    \int_{(c_v)}     \int_{(  c_u ) } \widetilde{\Psi} (u+v+1)    \mathfrak J_3 ( u,v ;\mathbb{L} ) 
\\ \times
   \frac{ \zeta(-u-v)  \zeta (1-u-v) \zeta(1-2v) \zeta(1-2u)  }{    \zeta(1+u-v)\zeta(1 - u+v) \sin ( \frac{\pi}{2} (u-v) ) }        \mathfrak K (   u,v )
       \\ \times
     \Gamma ( -u-v ) ( \mathcal{G}_{u-v, k-1} (u+v+1) -  \mathcal{G}_{v-u, k-1} (u+v+1) )    dudv,
\end{multline}
where
\begin{equation}\label{def:J3}
 \mathfrak J_3 ( u,v ;\mathbb{L} ) :=     \widetilde{F}_0 \left(-u-v-1, \frac{ u-v}{2} \right)  \frac{ \mathfrak J_2 (-u-v-1, \frac{u-v}{2} ; \mathbb{L} ) }{(2 \pi L_1 L_2 \elltwo )^{u+v} }
\end{equation}
and
$$ \mathfrak K_{ 5, L_1L_2}   (   u,v )  := \mathfrak K_{4, L_1L_2} ( u+v+1 , \tfrac12 (u-v)) ( \log Q)^{-|K|}    . $$
Note that by Lemma \ref{lem:Fsformula}, \eqref{def:J1euler}, \eqref{def:J2euler} and \eqref{def:J3}, we have
\begin{equation}\label{eqn:J3 bound}
  \mathfrak J_3 ( u,v ; \mathbb{L} )   \ll  \frac{(L_1 \ell_2 \elltwo )^\epsilon  (1+\elltwo^{ \tRe(v-u) }) }{L_1 \ell_1 \ell_2 (L_1 L_2 \elltwo )^{ \tRe (u+v)}}.
\end{equation}
Since
$\Phi_{k_j} ( i\mathcal U u ) = \int_0^\infty \widehat \Phi_{k_j} (w) Q^{uw} dw+\int_0^\infty \widehat \Phi_{k_j} (w) Q^{-uw} dw $, by \eqref{eqn:K4Lsz 1} with the substitution to $u,v$, 
we find that
\begin{equation}  \label{eqn:K5Luv 1}\begin{split}
     \mathfrak K_{ 5, L_1L_2}   (   u,v ) 
     & =    \prod_{ k_j \in K}   \bigg( \Phi_{k_j} ( i\mathcal U u ) + \Phi_{k_j} ( i\mathcal U v ) - 2\int_0^\infty \widehat \Phi_{k_j} (w)  ( Q^{uw} + Q^{vw} ) dw      \bigg)\\
   & \qquad -    \prod_{    k_j \in  K     }   \bigg(      - 2 \int_0^\infty  \widehat \Phi_{k_j} (w) ( Q^{uw} + Q^{vw})  dw  \bigg)\\
       & = \sum_{ \substack{ K_1 \sqcup K_2 \sqcup K_3 = K  \\  K_3 \neq K   } }  
       \mathfrak K_{5, L_1 L_2} ( K_1, K_2 , K_3 : u,v), 
 \end{split}   \end{equation}
where
\begin{multline}\label{def:K5LKuv}
    \mathfrak K_{5, L_1 L_2} ( K_1, K_2 , K_3 : u,v) \\
    :=       
       \prod_{ k_j \in K_1}      \Phi_{k_j} ( i\mathcal U u )  \prod_{k_j \in K_2 } \Phi_{k_j} ( i\mathcal U v )   \prod_{k_j \in K_3} \left( -2 \int_0^\infty \widehat \Phi_{k_j} (w) ( Q^{uw} + Q^{vw} ) dw  \right).
\end{multline}

We now analyze $ \mathfrak M_2  \left(   \frac{ - 1}{ \log Q},    \frac{-3}{ \log Q};      Q^{u+v}  \mathfrak K_{ 5, L_1L_2}(   u, v )\right) $. As an analogue of the $P_1^{it}P_2^{-it}$ structure discussed in the introduction, we show that the main contribution comes from $ \mathfrak K_{5, L_1 L_2} ( K_1, K_2 , K_3 : u,v)$ with $K_1  , K_2  \neq \emptyset $ in \eqref{eqn:K5Luv 1}. 
In this case, $\mathfrak K_{5, L_1 L_2} ( K_1, K_2 , K_3 : u,v)$ contains both factors $  \Phi_{\ell_1} ( i\mathcal U u ) $ and $ \Phi_{\ell_2} ( i\mathcal U v ) $
for some $ \ell_1 \neq \ell_2 $. Hence, define
\begin{equation}\label{def:K6Luv}
    \mathfrak K_{ 6, L_1L_2}   (   u,v ) := \sum_{ \substack{ K_1 \sqcup K_2 \sqcup K_3 = K  \\  K_1 ,K_2  \neq \emptyset   } }  
       \mathfrak K_{5, L_1 L_2} ( K_1, K_2 , K_3 : u,v),
\end{equation}
then we justify the above discussion.
\begin{lem} \label{lemma:SigmaCtn0 to SigmaCtn1} 
 Assuming RH and \eqref{conditions mathbbL}, we have
 \begin{equation} 
\Sigma_{\Ctn_0 } = \frac{ Q ( \log Q)^{|K|}  }{ L_1 L_2 \elltwo}  \Sigma_{\Ctn, 1}    + O\left( \frac{Q ( \log Q)^{|K|-1}  \log \log Q }{ L_1^{2-\epsilon} L_2 \ell_1 (\ell_2 \elltwo)^{1-\epsilon} } \right)
\end{equation}
 for any $\epsilon>0 $, where 
 \begin{equation}\label{def:SigmaCtn1}
   \Sigma_{\Ctn, 1} :=  \mathfrak M_2  \left(     \frac{-1}{ \log Q},  \frac{-3}{ \log Q};  Q^{u+v}      \mathfrak K_{ 6, L_1L_2}(  u,v )\right).  
   \end{equation}
\end{lem}
\begin{proof}
By Lemma \ref{lemma:SigmaCtn0 SigmaCtn1} and \eqref{eqn:SigmaCtn1 SigmaCtn2 2}, it is enough to show that 
$$    \mathfrak M_2  \left(     \frac{-1}{ \log Q},  \frac{-3}{ \log Q};      Q^{u+v}(  \mathfrak K_{ 5, L_1L_2}(  u,v )-\mathfrak K_{ 6, L_1L_2}(   u,v))\right) \ll \frac{ \elltwo^{ \epsilon}     }{ L_1^{1-\epsilon}   \ell_1 \ell_2^{1-\epsilon} } \frac{ \log \log Q}{ \log Q} .$$
  By \eqref{eqn:K5Luv 1} and \eqref{def:K6Luv}, 
 $ \mathfrak K_{ 5, L_1L_2}(  u,v )-\mathfrak K_{ 6, L_1L_2}(   u,v)$ is the sum of $\mathfrak K_{5, L_1 L_2} ( K_1,K_2 , K_3 : u,v)$ over $ K_1 \sqcup K_2 \sqcup K_3 = K $, $K_3 \neq K $ and  $K_1 $ or $ K_2 = \emptyset $.
 So it suffices to show that
\begin{equation}\label{eqn:lemma sigma Ctn01}
  \mathfrak M_2  \left(     \frac{-1}{ \log Q},  \frac{-3}{ \log Q};      Q^{u+v} \mathfrak K_{5, L_1 L_2} ( \emptyset ,K_2 , K_3 : u,v)   \right) \ll \frac{ \elltwo^{ \epsilon}     }{ L_1^{1-\epsilon}   \ell_1 \ell_2^{1-\epsilon} } \frac{ \log \log Q}{ \log Q} 
\end{equation}
 for $ K_2 \sqcup K_3 = K$ and $ K_2 \neq \emptyset $.  The case $ K_2 = \emptyset $ would hold similarly.

 By shifting the $u$-contour to $ \tRe(u) = - \frac17 $ and then applying Lemma \ref{lem:MellinofprodofJbessel} and 
 $$ \int_0^{\infty} \widehat{\Phi}_k(w) (Q^{uw} + Q^{vw}) \> dw \ll \int_0^{\infty} \widehat{\Phi}_k(w) (Q^{-w/7 } + e^{-3w}) \> dw\ll 1,$$
  we find that
 \begin{multline*}  
   \mathfrak M_2  \left(     \frac{-1}{ \log Q},  \frac{-3}{ \log Q};      Q^{u+v} \mathfrak K_{5, L_1 L_2} ( \emptyset ,K_2 , K_3 : u,v)   \right) \\
   \ll Q^{-\frac17} \int_{( \frac{- 3}{ \log Q})}     \int_{(    \frac{ - 1}{ 7} ) }   \prod_{k_j \in K_2 }| \Phi_{k_j} ( i\mathcal U v ) |   \frac{|  \mathfrak J_3 ( u,v  ; \mathbb{L}) |}{|u+v|^A} 
  \frac{ |  \zeta(1-2v) \zeta(1-2u) | | du||dv|  }{  |  \zeta(1+u-v)\zeta(1 - u+v)  | }  \ll Q^{\frac{-1}{7} + \epsilon}
\end{multline*}
for any $\epsilon>0$. This implies \eqref{eqn:lemma sigma Ctn01}. 
\end{proof}
In the next section, we compute $ \Sigma_{\Ctn,1}$ and conclude the proof of Proposition \ref{prop:Ctn0}.

\section{Residue calculation: computation of $\Sigma_{\Ctn,1}$} \label{sec:ctnrescalc}
In this section, assuming \eqref{conditions mathbbL} for $\mathbb L$, we compute $ \Sigma_{\Ctn, 1}  $ defined in \eqref{def:SigmaCtn1}. By \eqref{def:K5LKuv} and \eqref{def:K6Luv}, we have
\begin{align*}
       \mathfrak K_{ 6, L_1L_2}   (   u,v ) = & \sum_{ \substack{ K_1 \sqcup K_2 \sqcup K_3 \sqcup K_4  = K  \\  K_1 ,K_2  \neq \emptyset   } }   (-2)^{|K_3|+|K_4|}
           \prod_{k_j \in K_1}  \Phi_{k_j} ( i\mathcal U u )   \prod_{k_j \in K_2} \Phi_{k_j} ( i\mathcal U v )  \\
       & \prod_{k_j \in K_3} \int_0^\infty \widehat \Phi_{k_j} (w_{k_j})   Q^{uw_{k_j}}   dw_{k_j} \prod_{k_j \in K_4} \int_0^\infty \widehat \Phi_{k_j} (w_{k_j})   Q^{vw_{k_j}}   dw_{k_j} .
\end{align*}
By Fourier inversion and the change of variables, we have
$$ \Phi_{K_1} ( i \mathcal U u ) := \prod_{k_j \in K_1}  \Phi_{k_j} ( i\mathcal U u )  = \int_{-\infty}^\infty \widehat \Phi_{K_1} ( t_1 +W_1 ) Q^{-u (t_1 +W_1) } dt_1  ,  $$
and after integrating by parts twice
$$ \Phi_{K_2} ( i \mathcal U v )  =   \int_{-\infty}^\infty \widehat \Phi''_{K_2} ( t_2 +W_2  ) \frac{ Q^{-v (t_2 +W_2 ) } }{ v^2 ( \log Q)^2 } dt_2 ,   $$
where $$ W_1 := 1+ \sum_{ k_j \in K_3 } w_{k_j} ,\qquad W_2 := 1+ \sum_{k_j \in K_4 } w_{k_j} .$$
Therefore,
\begin{align*}
      Q^{u+v} \mathfrak K_{ 6, L_1L_2} &  (   u,v ) =  \sum_{ \substack{ K_1 \sqcup K_2 \sqcup K_3 \sqcup K_4  = K  \\  K_1 ,K_2  \neq \emptyset   } }  \frac{  (-2)^{|K_3|+ |K_4|} }{   ( \log Q)^2 } 
           \int_{[0,\infty)^{|K_3|+|K_4|}    }  \int_{-\infty}^\infty       \int_{-\infty}^\infty      \\
       & \frac{Q^{  -ut_1   -v t_2  }}{v^2} 
        \widehat \Phi_{K_1} ( t_1 +W_1) \widehat \Phi''_{K_2} ( t_2 +W_2) dt_1 dt_2  
        \prod_{k_j \in K_3 \sqcup K_4}   \left( \widehat \Phi_{k_j} (w_{k_j})       dw_{k_j} \right).
\end{align*}
The reason we had $  \widehat \Phi''_{K_2} ( t_2 ) $ is that the factor $ \frac1{v^2} $ provides an absolute convergence of the integrals in $\Sigma_{\Ctn,1} $ so that we can change their orders. We obtain
\begin{multline}  \label{eqn:Ctn1 to Ctn2}
  \Sigma_{\Ctn,1}   =  \sum_{ \substack{ K_1 \sqcup K_2 \sqcup K_3 \sqcup K_4  = K  \\  K_1 ,K_2  \neq \emptyset   } }  \frac{  (-2)^{ |K_3|+ |K_4|} }{  ( \log Q)^2 }  \\
             \int_{[0,\infty)^{|K_3|+|K_4|}    }    \Sigma_{\Ctn,2} (W_1 , W_2 )
         \prod_{k_j \in K_3 \sqcup K_4}   \left( \widehat \Phi_{k_j} (w_{k_j})       dw_{k_j} \right),
\end{multline}
where 
\begin{multline}\label{def:SigmaCtn2}
    \Sigma_{\Ctn,2}(W_1, W_2 ) 
    :=  \int_{-\infty}^\infty       \int_{-\infty}^\infty   \mathfrak M_2 \left(   \frac{-1}{ \log Q} , \frac{-3}{ \log Q} ;  
         \frac{  Q^{  - u  t_1  - v  t_2   }}{v^2}    \right) \\
         \widehat \Phi_{K_1} ( t_1 +W_1 ) \widehat \Phi''_{K_2} ( t_2 +W_2 ) dt_1 dt_2   .
\end{multline}

The factor $Q^{  - u  t_1  - v  t_2  }$ essentially determines the size of $\Sigma_{\Ctn,2}$. Depending on whether $   t_1 $ is positive or negative, we shift the $u$-integral to the right or left, respectively, so that the power of $Q$ becomes smaller after the shifts. We do the same for $  t_2$ and the $v$-integral. Then we expect to collect residues when we shift the integrals to the right and the resulting integrals are expected to be small. These observations are justified in the following two lemmas.
\begin{lem}\label{lemma:SigmaCtn2 1}
    Assume \eqref{conditions mathbbL} and $ |W_1 |, |W_2 |  \leq W $ for some $W>0$, then we have  
\begin{multline*}
    \Sigma_{\Ctn,2}(W_1, W_2 )  
     =  \int_{0}^\infty       \int_{0}^\infty   \mathfrak M_2 \left(   \frac{-1}{ \log Q} , \frac{-3}{ \log Q} ;  
         \frac{  Q^{ -u t_1  - v t_2  }}{v^2}    \right) \\
         \widehat \Phi_{K_1} ( t_1 +W_1 ) \widehat \Phi''_{K_2} ( t_2 +W_2 ) dt_1 dt_2    
         + O(( \log Q)^{1+\epsilon})
\end{multline*}
 for any $\epsilon>0$.
\end{lem}
\begin{proof}
It is enough to estimate the integral in   \eqref{def:SigmaCtn2} for $ t_1 \leq 0 $ or $t_2 \leq 0 $.
    For $  t_1 \leq 0  $, we first shift the $u$-contour in \eqref{def:M2cuv} to $ \tRe(u) = - \epsilon_1 $ and then change the order of integrals. We find that 
 \begin{align*}
    &   \int_{-\infty}^\infty       \int_{-\infty}^{0}   \mathfrak M_2 \left(   -\epsilon_1  , \frac{-3}{ \log Q} ;  
         \frac{  Q^{ - u t_1  - v t_2   }}{v^2}    \right)  \widehat \Phi_{K_1} ( t_1 +W_1 ) \widehat \Phi''_{K_2} ( t_2 +W_2 ) dt_1 dt_2   \\
    =&   \int_{-\infty}^\infty         \mathfrak M_2 \left(   -\epsilon_1  , \frac{-3}{ \log Q} ;  
         \frac{  Q^{ - vt_2    }}{v^2} \int_{-\infty}^{0 } Q^{  - u t_1   }    \widehat \Phi_{K_1} ( t_1 +W_1 ) dt_1    \right)  \widehat \Phi''_{K_2} ( t_2 +W_2 ) dt_2 .       
\end{align*}
After the $t_1 $-integration by parts, we find that the above is
\begin{align*}
    \ll         | \mathfrak M_2 | \left(   -\epsilon_1  , \frac{-3}{ \log Q} ;  
         \frac{   1 }{|v|^2  |u| \log Q }   \right)    ,
\end{align*}
where
\begin{multline} \label{def:M2 abs}
|\mathfrak M_2 |  \left( c_u,  c_v;        \mathfrak K (   u, v )\right)      : =     \int_{(c_v)}     \int_{(  c_u ) } |\widetilde{\Psi} (u+v+1)   \mathfrak J_3 ( u,v ; \mathbb{L} ) |
\\
   \frac{ |\zeta(-u-v)  \zeta (1-u-v) \zeta(1-2v) \zeta(1-2u) | }{   | \zeta(1+u-v)\zeta(1 - u+v) \sin ( \frac{\pi}{2} (u-v) ) |}        \mathfrak K (   u,v )
       \\
     |\Gamma ( -u-v ) ( \mathcal{G}_{u-v, k-1} (u+v+1) -  \mathcal{G}_{v-u, k-1} (u+v+1) ) |  | du||dv|.
\end{multline}
By Lemma \ref{lem:MellinofprodofJbessel} and \eqref{eqn:J3 bound},
we find that
\begin{align*}
      | \mathfrak M_2 | & \left(   -\epsilon_1  , \frac{-3}{ \log Q} ;       \frac{   1 }{|v|^2  |u| \log Q }   \right)   \\
      & \ll     \frac{  (\log Q)^{ \epsilon } }{\log Q}   \int_{(\frac{-3}{ \log Q} )}     \int_{(   -\epsilon_1 ) }\frac{|\zeta(1 - 2v)|}{|u+v|^A  | \zeta(1+u-v)   | |v|^2  |u|   }    | du||dv| \\
   &  \ll       {(\log Q)^{  \epsilon }}   \int_{(\frac{-3}{ \log Q} )}     
   \frac{ 1}{    |v|^2}      |dv| \ll  ( \log Q)^{1+\epsilon} 
\end{align*}
for any $A\geq 0$ and $\epsilon >0 $, where $ \epsilon_1 $ needs to be sufficiently small depending on $\epsilon$.

For $ t_1  \geq 0  $ and $   t_2 \leq 0 $, 
we shift the $v$-contour in \eqref{def:M2cuv} to $ \tRe(v) = - \epsilon_1 $ and then change the order of integrals, we find that 
 \begin{align*}
    &   \int_{-\infty}^{0}       \int_{0}^\infty   \mathfrak M_2 \left(  \frac{-1}{ \log Q} ,   -\epsilon_1     ;  
         \frac{  Q^{ -ut_1 -vt_2 }}{v^2}    \right)  \widehat \Phi_{K_1} ( t_1 +W_1 ) \widehat \Phi''_{K_2} ( t_2  +W_2) dt_1 dt_2   \\
    =&           \int_{0}^\infty   \mathfrak M_2 \left(  \frac{-1}{ \log Q} ,   -\epsilon_1     ;  
         \frac{  Q^{ -u t_1   }}{v^2}  \int_{-\infty}^{0}  Q^{  - v t_2  }   \widehat \Phi''_{K_2} ( t_2 +W_2 )  dt_2  \right)  \widehat \Phi_{K_1} ( t_1 +W_1 )  dt_1 .       
\end{align*}
After the $t_2 $-integration by parts, we find that the above is
\begin{align*}
    \ll &    |  \mathfrak M_2 | \left(  \frac{-1}{ \log Q} ,   -\epsilon_1     ;           \frac{ 1  }{|v|^3 \log Q}    \right)  \\
         \ll &  (\log Q)^{ \epsilon}  \int_{( -\epsilon_1  )}     \int_{(  \frac{-1}{ \log Q}) }  \frac{1}{|u+v|^A   |  \zeta(1 - u+v)   | |v|^3  }           | du||dv| \ll  (\log Q)^{ \epsilon} 
\end{align*}
for any $A\geq 0$ and $\epsilon >0 $ by Lemma \ref{lem:MellinofprodofJbessel} and \eqref{eqn:J3 bound}.
\end{proof}

Next, we compute the integral in Lemma \ref{lemma:SigmaCtn2 1}. 
\begin{lem}\label{lemma:SigmaCtn2 2}
  Assume \eqref{conditions mathbbL} and $ |W_1 |, |W_2 |  \leq W $ for some $W>0$, then we have  
 $$ \Sigma_{\Ctn,2}(W_1, W_2 ) =   ( \log Q)^2   
    \frac{ \widetilde{\Psi} ( 1)\delta_{\elltwo =1} }{ 8L_1 \ell_1 \ell_2  }  \mathscr{I}( \Phi_{K_1} , \Phi_{K_2} ; W_1 -1 , W_2 -1 )+O( ( \log Q)^{1+\epsilon} )   $$  
for any $ \epsilon>0$, where $\mathscr{I}$ is defined in \eqref{def:I12}.
\end{lem}
\begin{proof} 
 Let $ t_1 , t_2 \geq 0$. Then we shift the $u$-contour to $\tRe(u) = \epsilon_1 >0 $ and obtain 
\begin{equation}\label{eqn:M2 residue u}
 \mathfrak M_2  \left(   \frac{-  1}{ \log Q} , \frac{-3}{ \log Q} ;     \frac{  Q^{   -u t_1  - v t_2 }}{v^2}    \right)  
 =     \mathfrak R_1  +  \mathfrak R_2  
   + \mathfrak M_2 \left(   \epsilon_1 , \frac{-3}{ \log Q} ;     \frac{  Q^{   -u t_1  - v t_2 }}{v^2}    \right) , 
   \end{equation}
 where we recall that $\mathfrak M_2$ is defined in \eqref{def:M2cuv}, $\mathfrak R_1 $ is the residue at $ u=0 $ from the factor $\zeta(1-2u)$, and $ \mathfrak R_2 $ is the residue at $ u=-v$ from the factor $\zeta(1-u-v)\Gamma(-u-v)$.
 The contribution of the last term in \eqref{eqn:M2 residue u} to $\Sigma_{\Ctn,2}(W_1,W_2)$ in Lemma \ref{lemma:SigmaCtn2 1} is
\es{\label{eqn:M2 residue u error}
    & \int_{0}^\infty       \int_{0}^\infty   \mathfrak M_2 \left(   \epsilon_1 , \frac{-3}{ \log Q} ;  
         \frac{  Q^{ -u t_1  - v t_2  }}{v^2}    \right)  \widehat \Phi_{K_1} ( t_1 +W_1 ) \widehat \Phi''_{K_2} ( t_2 +W_2 ) dt_1 dt_2   \\
   &\ll  | \mathfrak M_2 | \left(   \epsilon_1 , \frac{-3}{ \log Q} ;  
         \frac{ 1}{|v|^2|u| \log Q }    \right) \ll ( \log Q)^{1+\epsilon}
}
for any $\epsilon>0$, similarly to the proof of Lemma \ref{lemma:SigmaCtn2 1}.
 
 Next we compute the residue $\mathfrak R_1 $ at $ u=0$.  
By \eqref{eqn:diffJbesselatvplus1} we find that
 $$      \mathfrak R_1 =      
         \frac{1}{4\pi i }    \int_{(\frac{-3}{ \log Q})}       \widetilde{\Psi} (v+1)    \mathfrak J_3 ( 0,v  ; \mathbb{L}) Q^{ -  v t_2     }
  \frac{ \zeta(-v)   \zeta(1-2v)  }{     \zeta(1 +v)   }       
  \frac{\Gamma(-v) \Gamma(v + \frac{k}{2}) \cos \frac{\pi v }{2} }{ v^2 \Gamma ( -v+ \frac{k}{2}) }     dv.  
 $$
  Since
\begin{equation*}\label{eqn:Fourier inversion second deriv}\begin{split}
    \int_0^\infty Q^{-vt_2 }    \widehat \Phi''_{K_2} ( t_2 +W_2 ) dt_2 =&  -  \widehat \Phi'_{K_2} (  W_2 ) - v \log Q   \widehat \Phi_{K_2} ( W_2 ) \\
& + v^2 ( \log Q)^2 \int_0^\infty Q^{-vt_2 }    \widehat \Phi_{K_2} ( t_2 +W_2 ) dt_2, 
\end{split}\end{equation*}
we find that
\begin{multline*}
   \int_0^\infty   \mathfrak R_1   \widehat \Phi''_{K_2} ( t_2 +W_2 ) dt_2  =  ( \log Q)^2 \int_0^\infty 
    \frac{1}{4\pi i }    \int_{(\frac{-3}{ \log Q})}       \widetilde{\Psi} (v+1)    \mathfrak J_3 ( 0,v  ; \mathbb{L}) Q^{ -  v t_2     } \\ \times
  \frac{ \zeta(-v)   \zeta(1-2v)  }{     \zeta(1 +v)   }       
  \frac{\Gamma(-v) \Gamma(v + \frac{k}{2}) \cos \frac{\pi v }{2} }{   \Gamma ( -v+ \frac{k}{2}) }     dv 
   \widehat \Phi_{K_2} ( t_2 +W_2 ) dt_2 + O( ( \log Q)^{1+\epsilon} )
\end{multline*}
for any $ \epsilon>0$. By shifting the $v$-contour to $ \tRe( v) = \epsilon_1 >0$, the residue at $ v=0$ is
$$  ( \log Q)^2   
    \frac{ 1}{ 8 }            \widetilde{\Psi} ( 1)    \mathfrak J_3 ( 0,0  ; \mathbb{L})   
  \int_0^\infty \widehat \Phi_{K_2} ( t_2 +W_2 ) dt_2.$$
 For the integral shifted to $ \tRe(v) = \epsilon_1 $, we integrate by parts twice with respect to $v$ and  obtain that it is $ O( ( \log Q)^{1+\epsilon} ) $ for any $\epsilon>0 $.
By Lemma \ref{lem:valueatsequal-1}, \eqref{def:J2euler} and \eqref{def:J3}, we find that 
\begin{equation}\label{eqn:J3-vv}
 \mathfrak J_3 ( -v, v ; \mathbb{L}) = \frac{ 1}{L_1 \ell_1 \ell_2 } \delta_{\elltwo =1} . 
 \end{equation}
By collecting the above estimations we find that
\begin{multline}\label{eqn:contribution of R1}
     \int_{0}^\infty       \int_{0}^\infty    \mathfrak R_1  \widehat \Phi_{K_1} ( t_1 +W_1 ) \widehat \Phi''_{K_2} ( t_2 +W_2 ) dt_1 dt_2   \\
     =  ( \log Q)^2   
    \frac{ \widetilde{\Psi} ( 1)\delta_{\elltwo=1} }{ 8L_1 \ell_1 \ell_2  }                 
  \int_0^\infty \widehat \Phi_{K_1} ( t_1 +W_1 ) dt_1
   \int_0^\infty \widehat \Phi_{K_2} ( t_2 +W_2 ) dt_2 +O( ( \log Q)^{1+\epsilon} )  
\end{multline}
for any $\epsilon>0 $.

Lastly, we find that   
\begin{multline*}  
 \mathfrak R_2
 = \frac{ -i^{-k}}{8\pi}    \int_{( \frac{-3}{ \log Q} )}     \oint_{ |u|= \frac{1}{ \log Q} } \widetilde{\Psi} (u+1)    \mathfrak J_3 ( u-v,v )   Q^{   - u t_1   + v( t_1 -t_2    ) }      
\\ \times
   \frac{ \zeta(-u)    \zeta(1-2v) \zeta(1-2u+2v)  ( \mathcal{G}_{u-2v, k-1} (u+1) -  \mathcal{G}_{2v-u, k-1} (u+1) )   }{    \zeta(1+u-2v)\zeta(1 - u+2v) \sin ( \frac{\pi}{2} (u-2v) ) }    \frac{du}{u^2}  \frac{dv}{v^2}.
\end{multline*}
We have a double pole at $ u=0$. 
By \eqref{eqn:diffGamma1} and \eqref{eqn:diffGammanumer}, we have
\begin{align*}  
 \mathfrak R_2
 =&  \frac14 \int_{( \frac{-3}{ \log Q} )}     \oint_{ |u|= \frac{1}{ \log Q} }  \mathfrak H ( u,v)   Q^{   - u t_1   + v( t_1 -t_2    ) }     
   \frac{ (u-2v)(-u+2v)    }{(-2v)(-2u+2v)    }   \frac{du}{u^2}  \frac{dv}{v^2} \\
   =& \frac{ \pi i }{2}\int_{( \frac{-3}{ \log Q} )} \left( - t_1 \log Q \mathfrak H(0,v) + \frac{\partial \mathfrak H }{ \partial u} (0,v) \right) Q^{v(t_1 - t_2 )}  \frac{dv}{v^2}    ,
\end{align*}
where
\begin{multline*}
    \mathfrak H ( u,v) := \widetilde{\Psi} (u+1)    \mathfrak J_3 ( u-v,v ; \mathbb{L} )     \frac{   \zeta(-u) \zeta(1-2v) \zeta(1-2u+2v)   }{    \zeta(1+u-2v)\zeta(1 - u+2v) }   
   \frac{\cos \frac{ \pi u }{2}  }{ \sin(\pi u - \pi v) \sin (\pi v) }  \\ \times
   \frac{(-2v)(-2u+2v)      }{ (u-2v)(-u+2v)      \Gamma(-v- \frac{k}{2} +1  )\Gamma(-u+v- \frac{k}{2} +1  ) \Gamma(  \frac{k}{2} -u+v ) \Gamma(  \frac{k}{2} -v)   }
\end{multline*}
is analytic at $ u=0$ and $ v=0$, and $ \mathfrak H ( 0,v) = - \frac{1}{2\pi^2} \widetilde \Psi(1)  \frac{ 1}{L_1 \ell_1 \ell_2 } \delta_{\elltwo =1}   $.
We shift the contour to $ \tRe( v ) = \epsilon_1 > 0 $ if $ t_1 \leq t_2 $ and to $ \tRe(v) = - \epsilon_1 $ if $ t_1 > t_2 $, then we find that
$$ \mathfrak R_2 = ( \log Q)^2      \frac{ \widetilde \Psi(1) \delta_{\elltwo=1}   }{2L_1 \ell_1 \ell_2 }  \delta_{ t_1 \leq t_2 } t_1 (t_1-t_2)    + O( ( \log Q)^{1+\epsilon})  $$
for any $\epsilon>0$. 
Since
$$ \int_{t_1}^\infty (t_1-t_2) \widehat \Phi''_{K_2} ( t_2 +W_2 ) dt_2 = - \widehat \Phi_{K_2} ( t_1 +W_2 ), $$
The contribution of $\mathfrak R_2 $ is 
\begin{multline}\label{eqn:contribution of R2}
     \int_{0}^\infty       \int_{0}^\infty    \mathfrak R_2  \widehat \Phi_{K_1} ( t_1 +W_1 ) \widehat \Phi''_{K_2} ( t_2 +W_2 ) dt_1 dt_2   \\
     = - ( \log Q)^2   
    \frac{ \widetilde{\Psi} ( 1)\delta_{\elltwo=1} }{ 2L_1 \ell_1 \ell_2  }                 
  \int_0^\infty  t_1 \widehat \Phi_{K_1} ( t_1 +W_1 )      \widehat \Phi_{K_2} ( t_1 +W_2 )  dt_1 +O( ( \log Q)^{1+\epsilon} )  .
\end{multline} 
The lemma follows from \eqref{eqn:M2 residue u}, \eqref{eqn:M2 residue u error}, \eqref{eqn:contribution of R1} and \eqref{eqn:contribution of R2}.
 \end{proof}

\subsection{Proof of Proposition \ref{prop:Ctn0} -- off-diagonal main terms.} 
By Lemmas \ref{lemma:SigmaCtn0 to SigmaCtn1} and \ref{lemma:SigmaCtn2 2}, and \eqref{eqn:Ctn1 to Ctn2},   
\begin{multline*}   
  \Sigma_{\Ctn_0 } =   Q ( \log Q)^{|K|}     \frac{ \widetilde{\Psi} ( 1)\delta_{\elltwo =1} }{ 8L_1^2 L_2  \ell_1 \ell_2  }    \sum_{ \substack{ K_1 \sqcup K_2 \sqcup K_3 \sqcup K_4  = K  \\  K_1 ,K_2  \neq \emptyset   } }   (-2)^{ |K_3|+ |K_4|}   \\
          \times   \int_{[0,\infty)^{|K_3|+|K_4|}    }    \mathscr{I} \bigg( \Phi_{K_1} , \Phi_{K_2} ; \sum_{k_j \in K_3 } w_{k_j} , \sum_{k_j \in K_4 } w_{k_j} \bigg) 
         \prod_{k_j \in K_3 \sqcup K_4}   \left( \widehat \Phi_{k_j} (w_{k_j})       dw_{k_j} \right) \\
          + O\left(  \frac{ Q ( \log Q)^{|K|-1+\epsilon }  }{ L_1 L_2 \elltwo}  \right)   
\end{multline*}
for any $ \epsilon>0$. Since $K_1 $ and $K_2$ are not empty, let $\ell_1 $ and $\ell_2 $ be their minimums, respectively. Then we replace $ K_1 =  \{\ell_1 \} \sqcup K_1' $ and $K_2 = \{ \ell_2 \} \sqcup K_2'$, and we see that $ \Phi_{K_1 } = \Phi_{\ell_1 , K_1'}$ and $ \Phi_{K_2} = \Phi_{\ell_2 , K_2'} $, where $\Phi_{\ell_i, K_i'}$ is defined in \eqref{eqn:PhikG}.  By \eqref{def:V}, the main term of $ \Sigma_{\Ctn_0 }$ is 
\begin{align*}   
    &Q ( \log Q)^{|K|}     \frac{ \widetilde{\Psi} ( 1)\delta_{\elltwo =1} }{  8L_1^2 L_2  \ell_1 \ell_2  }    \sum_{ \{ \ell_1 , \ell_2 \} \sqcup K'' = K } \frac{2}{ (-2)^{|K|-2}} \mathscr V ( \{ \ell_1 , \ell_2 \} , K'')  \\
    = &  \ Q ( \log Q)^{|K|}     \frac{ \widetilde{\Psi} ( 1)\delta_{\elltwo =1} }{ (-2)^{|K|}  L_1^2 L_2  \ell_1 \ell_2  }    \sum_{ \substack{ K'\sqcup K'' = K \\ |K'|=2}  } \mathscr V ( K' , K''). 
\end{align*}
This proves the proposition.


\section{Lemma \ref{lem:CMQ(R)_fill prime} -- the term $\mathscr C_{K_1, K_2, < }(Q)$} \label{sec:fillprimes}

Let $K= \{ k_1, \ldots , k_\kappa \} $ and assume that $ K_1 \sqcup K_2 \subset K $ and $ K_1 \neq \emptyset$. In this section, we prove Lemma \ref{lem:CMQ(R)_fill prime} by induction on $|K_2 | $. Recall the definitions in and below \eqref{def:CKQ}:
\begin{equation}  \label{def:CK1K2Q small}
\mathscr C_{K_1,K_2,<}  (Q) :=       \frac{  (-2)^\kappa}{ ( \log Q)^\kappa }     \sumprime_{ \mathbb{L}} \frac{\mu(L_1L_2)\zeta_{L_1} (2) }{L_1L_2}    \frac{\mu( \ell_1   \ell_2 )}{ \ell_1 ^2  \elltwo }   C_{K_1,K_2,<  }  (Q; \mathbb{L} ) , 
      \end{equation}
      where the prime sum is over $\mathbb{L}$ satisfying \eqref{conditions mathbbL}, and 
\begin{equation} \label{def:CK1K2QL small} \begin{split}
 C_{K_1,K_2,<} & (Q; \mathbb{L} )  := \sum_n  \Psi\left( \frac{L_1^2 L_2 \ell_1  \ell_2 n}{Q} \right)\\
 &  \times   \sumsharp_{\substack{ p_{k_1} , \ldots , p_{k_\kappa} \\ ( \mathfrak p (K) , L_1L_2 )=1    \\ \mathfrak p (K_1) |n , \, \mathfrak p(K_1)< \LL_{3\kappa} \\ ( \mathfrak p(K_2) , n ) = 1 }}
  \prod_{ j=1}^\kappa \left( \frac{ \log p_{k_j} }{ \sqrt{ p_{k_j} }} \widehat{\Phi}_{k_j} \left( \frac{ \log p_{k_j} }{ \log Q} \right)        \right)  
   \Delta_{L_1 \ell_1  \ell_2 n} ( \mathfrak p (K), \elltwo^2).
  \end{split}\end{equation}
 Since
$$  \sumsharp_{\substack{ p_{k_1} , \ldots , p_{k_\kappa} \\ ( \mathfrak p (K) , L_1L_2 )=1    \\ \mathfrak p (K_1) |n , \, \mathfrak p(K_1)< \LL_{3\kappa} \\ ( \mathfrak p(K_2) , n ) = 1 }}      =  \sumsharp_{\substack{ p_{k_1} , \ldots , p_{k_\kappa} \\ ( \mathfrak p (K) , L_1L_2 )=1    \\ \mathfrak p (K_1) |n , \, \mathfrak p(K_1)< \LL_{3\kappa}  }} - \sum_{\substack{ K_3 \sqcup K_4 = K_2  \\  K_3 \neq \emptyset } } \sumsharp_{\substack{ p_{k_1} , \ldots , p_{k_\kappa} \\ ( \mathfrak p (K) , L_1L_2 )=1    \\ \mathfrak p (K_1 \sqcup K_3 ) |n , \, \mathfrak p(K_1)< \LL_{3\kappa} \\ ( \mathfrak p(K_4) , n ) = 1 }} ,   $$
we expect that
\begin{equation}\label{eqn:inductive step}  
  \mathscr C_{K_1,K_2,<}  (Q) = \mathscr C_{K_1,\emptyset ,<}  (Q)  -    \sum_{\substack{ K_3 \sqcup K_4 = K_2  \\  K_3 \neq \emptyset } }  \mathscr C_{K_1\sqcup K_3 ,K_4,<}  (Q) +  O \left( \frac{Q}{ \log Q} \right) .
 \end{equation}
  The $O$-term is the contribution of the distinct primes $ p_{k_1} , \ldots , p_{k_\kappa} $ satisfying $ ( \mathfrak p (K) , L_1L_2 )=1  $, $ \mathfrak p (K_1 \sqcup K_3 ) |n $, $  \mathfrak p(K_1\sqcup K_3 ) \geq \LL_{3\kappa} $, $ \mathfrak p(K_1)< \LL_{3\kappa} $ and $  ( \mathfrak p(K_4) , n ) = 1 $ for $K_3 \sqcup K_4 = K_2 $ and $K_3 \neq \emptyset$. 
 If we replace $K_1 $, $K_1 \sqcup K_3 $ and $K_4 $ by $ K_1'$, $K_1$ and $K_2 $, respectively, with an additional condition $K_1' \subset K_1 $, then it is easy to see that \eqref{eqn:inductive step} follows from the proof of Lemma \ref{lem:boundtruncatedprimetR} by adding a condition $\mathfrak p(K_1') < \LL_{3\kappa}$ appropriately. So we omit the details.

Since $|K_4 |< |K_2|$ in \eqref{eqn:inductive step}, we already have the inductive step \eqref{eqn:inductive step} to prove Lemma \ref{lem:CMQ(R)_fill prime}. Thus, to complete the proof, it is sufficient to show the initial case 
\begin{equation}\label{eqn:induction initial step}
 \mathscr C_{K_1,\emptyset ,<}  (Q)  \ll \frac{Q}{ \log Q} 
 \end{equation}
for every nonempty $K_1 \subset K$. We will sketch how to modify the arguments in \S \ref{sec:applyKuznetsov}--\S \ref{sec:trivialchar off diag main terms} to prove \eqref{eqn:induction initial step}.

First, we remove the condition $ \mathfrak p (K_1 ) |n$ and replace $ n$ by $ \mathfrak p (K_1 ) n $ in \eqref{def:CK1K2QL small}. Then we see that
\begin{equation} \label{eqn:CK1K2QL small 1} \begin{split}
 C_{K_1,\emptyset ,<} & (Q; \mathbb{L} )   = \sum_n  \sumsharp_{\substack{ p_{k_1} , \ldots , p_{k_\kappa} \\ ( \mathfrak p (K) , L_1L_2 )=1    \\  \mathfrak p(K_1)< \LL_{3\kappa}  }}  \prod_{ j=1}^\kappa \left( \frac{ \log p_{k_j} }{ \sqrt{ p_{k_j} }} \widehat{\Phi}_{k_j} \left( \frac{ \log p_{k_j} }{ \log Q} \right)        \right)  
  \\
 &  \times   
  \Psi\left( \frac{L_1^2 L_2 \ell_1  \ell_2 \mathfrak p (K_1) n}{Q} \right) \Delta_{L_1 \ell_1  \ell_2 \mathfrak p (K_1) n} ( \mathfrak p (K), \elltwo^2).
  \end{split}\end{equation}
If we compare it with \eqref{def:CKQLre}, we can obtain \eqref{eqn:CK1K2QL small 1} by replacing $n$ with $ \mathfrak p(K_1) n $ and adding the condition $ \mathfrak p(K_1 ) < \LL_{3\kappa} $ to \eqref{def:CKQLre}.
Applying Petersson's formula and following the arguments in the beginning of \S \ref{sec:applyKuznetsov}, we have 
\begin{multline}\label{eqn:CK_1QL 1}
 C_{K_1, \emptyset, <}  (Q; \mathbb{L}) \\
 =   2  \pi i^{-k}  \sumd_{P_1, \ldots , P_\kappa }  \sum_{c \geq 1}  \int_{-\infty}^\infty \cdots \int_{-\infty}^\infty \widehat{H}(u, \vb)     \sumsharp_{\substack{ p_{k_1} , \ldots , p_{k_\kappa} \\ ( \mathfrak p (K) , L_1L_2 )=1  \\ \mathfrak p(K_1 ) < \LL_{3\kappa} }}
  \prod_{ j=1}^\kappa \left( \frac{ \log p_{k_j} }{ \sqrt{ p_{k_j} }} V \left( \frac{ p_{k_j}}{P_j} \right)   \e{ \frac{p_{k_j}}{P_j} v_j }   \right)  
   \\
\times    \sum_{n}   \frac{S( \elltwo^2, \mathfrak p (K) ; cL_1 \ell_1  \ell_2 \mathfrak p(K_1 ) n )}{ cL_1 \ell_1  \ell_2 \mathfrak p(K_1 ) n }  
 h_u  \left( \frac{4\pi \elltwo \sqrt{\mathfrak p (K) }}{cL_1 \ell_1  \ell_2 \mathfrak p(K_1 ) n}\right)   du dv_1 \cdots dv_{\kappa}   
  \end{multline}
similarly to \eqref{eqn:CKQ 2}.

We apply Kuznetsov's formula to \eqref{eqn:CK_1QL 1} as described in \S \ref{sec:applyKuznetsov}, but with $ N = c L_1 \ell_1 \ell_2 \mathfrak p(K_1) $ and $ \mathfrak p(K_1) < \LL_{3\kappa}$.
Arguing as in \S \ref{sec:dispart} and \S \ref{sec:Eisenwithnontrivialchar}, we obtain bounds analogous to Propositions \ref{prop:DisHol} and \ref{prop:Ctn non} for the contribution from the discrete spectrum, the holomorphic forms, and the Eisenstein series associated to non-principal characters. We essentially need to replace $n$ by $ \mathfrak p(K_1)n$ and add conditions $ \mathfrak p(K_1 ) < \LL_{3\kappa} $ to the sums over the primes in \S \ref{sec:applyKuznetsov}, \S \ref{sec:dispart} and \S \ref{sec:Eisenwithnontrivialchar}. Since the bounds in Propositions \ref{prop:DisHol} and \ref{prop:Ctn non} have a factor $Q^\epsilon$ with arbitrary $ \epsilon>0$, we see that $ \mathfrak p(K_1) < \LL_{3 \kappa} $ is small enough so that we may crudely estimate the sums over $p_{k_j}$ for $ k_j \in K_1 $.

By Kuznetsov's formula to \eqref{eqn:CK_1QL 1} as in the previous paragraph, we find that the contribution from the Eisenstein series associated to the principal characters is
\begin{multline*}
 \Sigma'_{\Ctn_0} := \frac{i^{-k}}{2} \sumd_{P_1, \ldots , P_\kappa }  \sum_{c \geq 1}  \int_{-\infty}^\infty \cdots \int_{-\infty}^\infty      \sumsharp_{\substack{ p_{k_1} , \ldots , p_{k_\kappa} \\ ( \mathfrak p (K) , L_1L_2 )=1  \\ \mathfrak p(K_1 ) < \LL_{3\kappa} }}
  \prod_{ j=1}^\kappa \left( \frac{ \log p_{k_j} }{ \sqrt{ p_{k_j} }} V \left( \frac{ p_{k_j}}{P_j} \right)   \e{ \frac{p_{k_j}}{P_j} v_j }   \right)  
   \\
\times   \sum_{  M | cL'_0}  \int_{-\infty}^{\infty}  \rho_{\chi_0, M, cL'_0}(\mathfrak p(K), t) \overline{\rho_{\chi_0, M, cL'_0}(\elltwo^2, t)} h_{u, +}(t) \> dt  \widehat{H}(u, \vb) du dv_1 \cdots dv_{\kappa}  
  \end{multline*}
with $ L'_0 = L_0 \mathfrak p(K_1) =   L_1 \ell_1 \ell_2 \mathfrak p(K_1) $. This equals to $ \Sigma_{\Ctn_0 }$ in \eqref{Section9 eqn 1} except for $L'_0  $ and the condition $ \mathfrak p (K_1 ) < \LL_{3 \kappa } $. 
By following the arguments in \S \ref{sec:trivialchar off diag main terms}, we find that
\begin{multline}  \label{eqn:Sigmaprime 1}
 \Sigma'_{\Ctn_0 }  = \frac{ i^{-k}}{2}    \int_{-\infty}^\infty    \int_{0}^\infty  \int_{( -\epsilon_1) } \frac{ \widetilde{\Psi} (s) Q^s }{(4 \pi L_1 L_2 \elltwo)^s }  \sumsharp_{\substack{ p_{k_1} , \ldots , p_{k_\kappa} \\ ( \mathfrak p (K) , L_1L_2 )=1  \\ \mathfrak p(K_1)< \LL_{3\kappa} }}
  \prod_{ j=1}^\kappa \left( \frac{ \log p_{k_j} }{  p_{k_j}^{\frac12(1+s) }}  \widehat \Phi_{k_j}\left(\frac{\log p_{k_j}}{\log Q}\right)     \right) 
\\
  \times \widetilde{\varrho}_{L'_0 , \mathfrak p(K), \elltwo ; t} (-s )                      (J_{2it}(\xi) - J_{-2it} ( \xi) ) J_{k-1} (\xi)   \xi^{s-1} ds     d\xi     \frac{  dt}{ \sinh ( \pi t)}     
\end{multline}
similarly to \eqref{Section9 eqn 3}.

Next, we find an expression for $ \widetilde{\varrho}_{L'_0 , \mathfrak p(K), \elltwo ; t} ( s )$. 
By Lemma \ref{lem:eisenteincoeff} with $L'_0 $ in place of $L_0$, we have
$$ 	\widetilde{ \varrho}_{L'_0, \mathfrak p(K), \elltwo ; t} (s )  
	 =  \frac{ \mathfrak p(K)^{it} }{  L_0 \mathfrak p(K_1)   \elltwo^{2it} |\zeta(1+2it)|^2}    
     \sum_{ \substack{ d_1|\mathfrak p (K) \\ d_2|\elltwo^2 } }\frac{\mu(d_1 d_2)}{d_1^{2it} d_2^{-2it}}   \sum_{ \substack{ c_1 | \mathfrak p (K) / d_1  \\  c_2 | \elltwo^2 / d_2  }} \frac{c_2^{2it}}{c_1^{2it}} F_{L'_0, d_1  d_2 , \m    }  (s,it ) ,
$$ 
where $ \m  =  \frac{ \mathfrak p(K) \elltwo^2}{c_1c_2d_1d_2}$. By applying Lemma \ref{lem:Fsformula} with 
$ \alpha = L_0' = L_0 \mathfrak p(K_1) $, $ r = d_1 d_2 $, $\alpha_1 := \frac{L_0 \mathfrak p(K_1) }{ ( L_0 \mathfrak p(K_1), d_1 d_2 )    } = \frac{\mathfrak p(K_1)}{( \mathfrak p(K_1) , d_1 )} \cdot \frac{L_0}{(d_2,\ell_2 )}  $ and $ r_1 := \frac{d_1d_2}{ ( L_0 \mathfrak p(K_1), d_1 d_2 ) } = \frac{d_1}{(d_1 , \mathfrak p(K_1)) } \cdot \frac{d_2}{(d_2, \ell_2 )} $ as in the beginning of \S \ref{Section:combinatorics and computations}, we find that
 \begin{multline}\label{eqn:Fs formula 4}
 F_{L_0' , d_1d_2 , \m } (s,it)  
	=   F_{L_0 , d_1d_2 , \m } (s,it) (d_1 , \mathfrak p (K_1))^{1 + s}  \\    
    \prod_{ \substack{ p |  \mathfrak p (K_1 ) \\ p \nmid d_1 }  }  \left( \frac1p - \frac{1}{p^{2+s}} +   \delta_{p\nmid \m_1 } \left( 1  +   \frac{1 }{p^{s+1} (p-1)}  \right) \right)    \\  
    \prod_{   \substack{ p| \m_1 \\ p|  \mathfrak p(K_1)/(d_1, \mathfrak p (K_1)) }   } \left(  W_p (s, it) - \frac{1}{p^{1+s}} - \frac{1}{p^{2+2s} (p-1) }\right)^{-1}  \prod_{ \substack{ p |\mathfrak p(K_1)   \\ p \nmid \m_1 d_1     }} W_p ( s, it)^{-1}
    \end{multline}
with $ \m_1 = \frac{\mathfrak p(K)}{c_1d_1}  $.
By \eqref{eqn:Fs formula 3} and \eqref{eqn:Fs formula 4}, we have
 \begin{multline*}\label{eqn:Fs formula 5}
 F_{L_0' , d_1d_2 , \m } (s,it)  
	=   \zeta(1+s)   \widetilde{F}  (s,it) \mathfrak J_1(s, it ; \mathbb{L}, c_2, d_2 )  \frac{ (d_1 , \mathfrak p (K_1))^{1 + s} }{ d_1^{1 + s}}   \\    
    \prod_{p | d_1  }  \left(   1    +  \frac{1}{p^{1+s} (p-1)}   \right)    
    \prod_{p | \m_1 d_1  } W_p ( s, it)^{-1} \prod_{   p| \m_1    } \left( W_p (s, it) - \frac{1}{p^{1+s}} - \frac{1}{p^{2+2s} (p-1) }\right) \\
    \prod_{ \substack{ p |  \mathfrak p (K_1 ) \\ p \nmid d_1 }  }  \left( \frac1p - \frac{1}{p^{2+s}} +   \delta_{p\nmid \m_1 } \left( 1  +   \frac{1 }{p^{s+1} (p-1)}  \right) \right)    \\
    \prod_{   \substack{ p|  \mathfrak p(K_1) \\ p| \m_1, ~ p \nmid d_1   }   } \left(  W_p (s, it) - \frac{1}{p^{1+s}} - \frac{1}{p^{2+2s} (p-1) }\right)^{-1}  \prod_{ \substack{ p |\mathfrak p(K_1)   \\ p \nmid \m_1 d_1     }} W_p ( s, it)^{-1}.
    \end{multline*}
Thus, we find that
\begin{multline*}
    \widetilde{ \varrho}_{L'_0, \mathfrak p(K), \elltwo ; t} (s )  
	 =  \frac{ \mathfrak p(K)^{it} \zeta(1+s)   \widetilde{F}  (s,it) \mathfrak J_2(s, it ; \mathbb{L} )  }{   \mathfrak p(K_1)     |\zeta(1+2it)|^2}    
     \sum_{  c_1 d_1|\mathfrak p (K)   }\frac{\mu(d_1  ) (d_1 , \mathfrak p (K_1))^{1 + s} }{ c_1^{2it} d_1^{1+s+2it}   }        \\    
    \prod_{p | d_1  }  \left(   1    +  \frac{1}{p^{1+s} (p-1)}   \right)    
    \prod_{p | \m_1 d_1  } W_p ( s, it)^{-1} \prod_{   p| \m_1    } \left( W_p (s, it) - \frac{1}{p^{1+s}} - \frac{1}{p^{2+2s} (p-1) }\right) \\
    \prod_{ \substack{ p |  \mathfrak p (K_1 ) \\ p \nmid d_1 }  }  \left( \frac1p - \frac{1}{p^{2+s}} +   \delta_{p\nmid \m_1 } \left( 1  +   \frac{1 }{p^{s+1} (p-1)}  \right) \right)    \\
    \prod_{   \substack{ p|  \mathfrak p(K_1) \\ p| \m_1, ~ p \nmid d_1   }   } \left(  W_p (s, it) - \frac{1}{p^{1+s}} - \frac{1}{p^{2+2s} (p-1) }\right)^{-1}  \prod_{ \substack{ p |\mathfrak p(K_1)   \\ p \nmid \m_1 d_1     }} W_p ( s, it)^{-1}.
\end{multline*}
Due to the factor $\frac{1}{ \mathfrak p(K_1)}$ in the first line of the above display, when we follow the arguments in \S \ref{sec:trivialchar off diag main terms}, we expect that the sum over the primes $p_{k_j}$ for $ k_j \in K_1 $ is $ (\log Q)^{|K_1|} $ smaller so that
$$ \mathscr C_{K_1,\emptyset ,<}  (Q)  \ll \frac{Q}{ (\log Q)^{|K_1|}} $$
would hold. This implies \eqref{eqn:induction initial step}.

\section{$n$-th centered moments for $O(N)$} \label{section:nth centered moment ON}
In this section we prove Theorem \ref{thm:Cn}. Recall that $SO(2N)$ and $USp(2N)$ are the classical even orthogonal and symplectic groups, respectively.  
Define
\es{  
\label{def T G N 0}  \mathscr{T}_{+} (S) := \lim_{N \to \infty} \mathscr{T}_{+,N} (S)  &  :=   \lim_{N \to \infty}   \int_{SO (2N)}   \bigg( \prod_{ \ell \in S}      \sum_{0<  |j| \leq N   } \Phi_\ell \left( \frac{N \theta_j  }{ \pi}  \right)  \bigg)   dX_{SO(2N)} , }
\es{
\label{def T G N 1} \mathscr{T}_{-} (S) :=  \lim_{N \to \infty} \mathscr{T}_{-, N} (S)  & := \lim_{N \to \infty}\int_{USp (2N)}   \bigg( \prod_{ \ell \in S}   \sum_{0 <  |j| \leq N   } \Phi_\ell \left( \frac{N \theta_j  }{ \pi}  \right)  \bigg)     dX_{USp(2N)}     }
 for $ S \subset \{ 1, \ldots , n \}$.  By Lemma \ref{lemma O- to USp} we have
 \begin{equation}\label{def T G N 2} 
  \mathscr{T}_{-} (S)  = \lim_{N \to \infty}\int_{O^-(2N+2)}   \bigg( \prod_{ \ell \in S}   \sum_{0 <  |j| \leq N   } \Phi_\ell \left( \frac{N \theta_j  }{ \pi}  \right)  \bigg)     dX_{O^-(2N+2)}  .   
  \end{equation}
Now, by \eqref{def C even}, \eqref{def C odd}, \eqref{def T G N 0} and \eqref{def T G N 2}, we have
   \begin{equation}\label{formula C even odd by T}\begin{split}
 C_{even}(n)  & = \sum_{ S_1 \sqcup S_2 = [n] }\mathscr{T}_+ (S_1 ) \prod_{ \ell \in S_2 } \left( - \widehat{\Phi}_\ell (0) - \frac{ \Phi_\ell ( 0 )}{2} \right)   ,  \\
C_{odd}(n)    & = \sum_{ S_1 \sqcup S_2 = [n] }\mathscr{T}_- (S_1 ) \prod_{ \ell \in S_2 } \left( - \widehat{\Phi}_\ell (0) + \frac{ \Phi_\ell ( 0 )}{2} \right)  .
\end{split}   \end{equation}
 This reduces our problem to computing the limits  
 $$ \mathscr{T}_{\pm} := \lim_{N \to \infty}  \mathscr{T}_{\pm, N } :=\lim_{N \to \infty}  \mathscr{T}_{ \pm , N}( [ \nu ] ),$$ 
 where without loss of generality, we have replaced $S$ by $[\nu]$ when $S\subset [n]$ with $|S| = \nu$, $1\le \nu \le n$.

First we apply results from Mason and Snaith \cite{MS} to \eqref{def T G N 0} and \eqref{def T G N 1}.  The work of Mason and Snaith builds upon the method introduced by Conrey and Snaith \cite{CS1}, who derived a new formula for $n$-correlations of eigenvalues of matrices from the unitary group and subsequently applied it in \cite{CS2} to match number theory results with random matrix predictions. Although their formula appears intricate, it expresses the result in terms of a test function whose Fourier transform may have restricted support. This formulation allows one to identify immediately which terms survive under the support restrictions on the test function, thereby aligning with the number theoretic calculations.

For notational convenience  we rename the functions $J^* $ and $J_{USp(2N)}^*$ defined at (2.26) and (3.12) in \cite{MS}  as  $ J_+^*$ and $J_-^*$, respectively, and define
\begin{equation}\label{def J star}  
  J_{\pm}^* (W , N)  :=  \sum_{ \substack{  W'\sqcup W''  \sqcup W_1 \sqcup \cdots \sqcup W_R  = W   \\     |W_r | = 2  } } e^{-2N \sum_{ w \in W' } w } H_0^{\mp}(W')  \prod_{ \alpha \in W''}   H_1^\mp ( W' ,   \alpha   )   \prod_{r=1}^R H_2  (  W_r )  , 
  \end{equation}
  where 
  \begin{equation}\label{def HWs} \begin{split}
 H_0^{\mp}(W') & := (-1)^{|W'|}    \prod_{   \{ \alpha, \beta \}  \subset W' , \alpha \neq \beta }   \frac{ ( 1- e^{- \alpha + \beta}) ( 1- e^{  \alpha - \beta})  }{(1- e^{-\alpha-\beta})(1- e^{\alpha+\beta})} \prod_{ \alpha \in W'} \frac{  1 }{   1 - e^{ \mp 2 \alpha }   }   , \\ 
 H_1^\mp ( W' ,  \alpha  )  & :=   \sum_{ w \in W'}  \left(  \frac{1}{1- e^{\alpha - w} }   - \frac1{1-e^{ \alpha + w } } \right) \mp \frac1{1-e^{2 \alpha}}    , \\ 
H_2  (  \{  \alpha, \beta \}  )  & :=  \frac{e^{\alpha + \beta} }{ (1-e^{\alpha + \beta})^2} .
\end{split}\end{equation} 
Here, $W'$ corresponds to $D$ in the definitions of  $J^* $ and $J_{USp(2N)}^*$ in \cite{MS}, $W''$ is the union of the $W_r$  with $ |W_r|=1 $ and $R$ is any positive integer.

We further fix, for $1\le \ell \le \nu$, positive numbers $\delta_\ell$ to be determined later.  Let
\begin{equation}\label{eqn:deltaellNdef}
    \delta(\ell , N) := \frac{ \log \log N}{ N} \delta_\ell.
\end{equation}
For clarity, we record the following lemma.
\begin{lemma}\label{lem:Phibdd}
Fix notation as above and suppose that $z_\ell = u  \pm it $ with $ |u| \leq \delta ( \ell, N)$ and the support of $ \widehat{\Phi}_\ell $ is contained in $[ - \sigma_\ell, \sigma_\ell ]. $
Then
\begin{equation}\label{eqn Phi ell bound 1}
 \Phi_\ell \left( \frac{ i N z_\ell  }{ \pi}  \right) 
 \ll_A  \frac{ (\log N)^{2 \sigma_\ell \delta_\ell}   }{(N|z_\ell|)^A },
\end{equation}
for any integer $A\ge 0$. 
\end{lemma}

\begin{proof} 
Our condition on $z_\ell$ implies that $|e^{2Nz_\ell x}| \le e^{2N \sigma_\ell \delta(\ell, N)} \le (\log N)^{2 \sigma_\ell \delta_\ell}$, for $\delta(\ell, N)$ as in \eqref{eqn:deltaellNdef}.  Thus, by integration by parts $A$ times, we have
 \begin{equation}
 \Phi_\ell \left( \frac{ i N z_\ell  }{ \pi}  \right) = \int_{ - \infty}^\infty \widehat{\Phi}_\ell ( x) e^{ - 2N z_\ell x } dx \ll_A  \frac{ (\log N)^{2 \sigma_\ell \delta_\ell}   }{(N|z_\ell|)^A },
\end{equation}
as desired.
\end{proof}

\begin{lemma}\label{lemma eqn MS}

Let notation be as above. Then we have 
\begin{equation}\label{eqn T pm lemma eqn MS}
\mathscr{T}_{\pm  }   =      \lim_{N \to \infty}    \sum_{ K  \sqcup \tilde{K} = [ \nu  ] }    \frac{  N^{|\tilde{K} |}  }{ (   \pi i )^\nu }  
 \int_{  ( \delta (\ell,N) ; [ \nu ] ) }    J_{\pm}^* (   z_{K }   ,N )
  \prod_{ \ell =1}^\nu      \Phi_\ell \left( \frac{   i N z_\ell  }{  \pi}  \right)  dz_1 \cdots dz_\nu ,   
\end{equation}
  where
 $ ( \delta(\ell,N)  ; L) $ means that   the $z_\ell$-integral for each $ \ell \in L $ is over the vertical line  from $ \delta(\ell,N) -  i  \infty $ to $  \delta(\ell,N) + i \infty  $, and  
 $$   z_K := \{  z_k : k \in K \} . $$ 
 
\end{lemma}

\begin{proof}
Let  
$C_\pm  $ be the path from $ \pm  \frac{ \log \log N}{ N} \delta  - \pi i $ to $ \pm \frac{ \log \log N}{ N} \delta + \pi i $ for a small $\delta >0$. 
 By applying \cite[Lemma 2.9]{MS} to \eqref{def T G N 0}, we have
\begin{align*}
\mathscr{T}_{+ ,N}   =     &    \sum_{ K_1 \sqcup K_2  \sqcup K_3 = [ \nu  ] }    \frac{ (2N)^{|K_3 |}  }{ ( 2 \pi i )^\nu } \int_{ C_+^{K_1}} \int_{C_-^{K_2 \sqcup K_3}}  J_+^* (   z_{K_1}  \sqcup  -z_{K_2}  ,N  ) \prod_{ \ell=1}^\nu      \Phi_\ell \left( \frac{ i N z_\ell  }{  \pi}  \right)  dz_1 \cdots dz_\nu\\
& + o(1)
\end{align*}
as $N \to \infty$, where    $\int_{ C_+^{K_1}} \int_{C_-^{K_2 \sqcup K_3}}$ means we are integrating all the variables in $z_{K_1}$ along the $C_+$ path and all others up to the $C_-$ path   and $  - z_K := \{  -z_k : k \in K \}$.  In our application of \cite[Lemma 2.9]{MS}, we have taken their $f$ to be 
$$f(\theta_{j_1}, \ldots ,\theta_{j_\nu}) = \prod_{\ell =1}^\nu \Phi_\ell \bfrac{N\theta_{j_\ell}}{\pi}.
$$

In \cite[Lemma 2.9]{MS}, $f$ is assumed to be periodic, which is used in the proof in order to show that certain horizontal integrals cancel out.  Those same horizontal integrals can be justified to be $o(1)$ in our case by the bound from Lemma \ref{lem:Phibdd}, which implies that
\begin{equation}\label{eqn Phi ell bound 2}
 \Phi_\ell \left( \frac{ i N z_\ell  }{ \pi}  \right) \ll_A  \frac{ (\log N)^{2 \sigma_\ell \delta_\ell}   }{(N|z_\ell|)^A } \ll_A  \frac{ (\log N)^{2 \sigma_\ell \delta_\ell}   }{N^A },
 \end{equation}
for any integer $A\ge 0$, since $|z_\ell| \ge \pi$.

Now, since the integrand is holomorphic in $z_k$ for all $z_k \in K_3$, we may shift each $ z_k$ contour for $ k \in K_3 $ to $C^+$, where again the horizontal parts are small by \eqref{eqn Phi ell bound 1} from Lemma \ref{lem:Phibdd}.  We further substitute $ z_k$ by $ - z_k $ for each $ k \in K_2$ to see that
\begin{equation}\label{eqn T even N proof 1}
  \mathscr{T}_{+ ,N}   =          \sum_{ K_1 \sqcup K_2  \sqcup K_3 = [\nu ] }    \frac{ (2N)^{|K_3 |}  }{ ( 2 \pi i )^\nu }  
 \int_{ C_+^{[\nu]}}    J_+^* (   z_{K_1   \sqcup  K_2}  ,N )   \prod_{ \ell =1 }^\nu       \Phi_\ell \left( \frac{ i N z_\ell  }{ \pi}  \right)  dz_1 \cdots dz_\nu +o(1),
\end{equation}
where we have used the fact that $ \Phi_\ell$ is even for all $\ell$. Since the integrand in \eqref{eqn T even N proof 1} is holomorphic, we shift each $z_\ell$ contour in \eqref{eqn T even N proof 1} to the line segment from $\delta (\ell ,N)  -  \pi i $ to $  \delta (\ell , N)  + \pi  i   $. By the shifts, we obtain extra terms containing horizontal $z_\ell$-integrals, which are also negligible by \eqref{eqn Phi ell bound 1}. 
We can also extend these integrals to the vertical line from $  \delta(\ell,N) - i \infty$ to $ \delta(\ell,N) + i \infty$ in a similar way. By combining the sum over $K_1 $ and $K_2 $ as a sum over $K$ and replacing $K_3 $ by $ \tilde{K}$, we prove the even case. 

 To prove the odd case, we apply \cite[Lemma 3.5]{MS} instead and argue similarly to the even case. 
\end{proof}

We integrate the $z_\ell$-integrals in \eqref{eqn T pm lemma eqn MS} for $ \ell \in \tilde{K} $. 
For each $ \ell \in \tilde{K} $, we see that
$$ \frac{N}{ \pi i } \int_{ ( \delta(\ell, N) ) } \Phi_\ell \left( \frac{ iN z_\ell }{  \pi }  \right) dz_\ell =\frac{N}{ \pi i } \int_{- i \infty }^{i \infty} \Phi_\ell \left( \frac{ iN z_\ell }{  \pi }  \right) dz_\ell =  \int_{-   \infty }^{  \infty} \Phi_\ell (   z_\ell ) dz_\ell =    \widehat{\Phi}_\ell (0). $$
Thus, we find that 
\begin{equation} \label{eqn T pm U pm K}
\mathscr{T}_{\pm   }    
   =           \sum_{  K  \sqcup \tilde{K} = [\nu] }  \left(   \prod_{ \ell \in \tilde{K} }    \widehat{\Phi}_\ell (0) \right)   \mathscr{U}_{\pm} (K )     ,
 \end{equation}
where
\begin{equation}\label{def U pm K}
  \mathscr{U}_\pm (K  ) := \lim_{N \to \infty}  \frac{1}{ (  \pi i )^{ |K  |}}  \int_{ (\delta(\ell, N) ; K)   }   J_\pm^* (z_{K } ,N  )    \prod_{ \ell \in K  }    \Phi_\ell \left( \frac{i N z_\ell  }{  \pi}  \right)    dz_\ell.
  \end{equation}
Define
\begin{equation}\label{def U pm j K}
  \mathscr{U}_{\pm, j}  (K  ) := \lim_{N \to \infty}  \frac{1}{ (  \pi i )^{ |K  |}}  \int_{ (\delta(\ell, N) ; K)   }   J_{\pm,j}^* (z_{K } ,N  )    \prod_{ \ell \in K  }    \Phi_\ell \left( \frac{i N z_\ell  }{  \pi}  \right)    dz_\ell 
  \end{equation}
  where
\begin{equation}\label{def J pm j star}  
  J_{\pm, j }^* (z_K  , N)  :=  \sum_{ \substack{  K'\sqcup K''  \sqcup K_1 \sqcup \cdots \sqcup K_R  = K   \\  |K'|=j, ~   |K_r | = 2  } } e^{-2N \sum_{ k \in K' } z_k } H_0^{\mp}(z_{K'} )  \prod_{ \ell \in K''}   H_1^\mp ( z_{K'}  ,   z_\ell   )   \prod_{r=1}^R H_2  ( z_{ K_r }  )    
  \end{equation}
  for $ j \geq 0 $, then we see that 
\begin{equation}\label{eqn U pm K sum over all j}
\mathscr{U}_{\pm}  (K  ) = \sum_{j = 0 }^{ |K| }   \mathscr{U}_{\pm, j}  (K  ).    
\end{equation}
Due to the support conditions on $\widehat{\Phi}_\ell$, the above sum is actually shorter. 
\begin{lemma} \label{lemma:supports restriction}
Assume that  the support of $ \widehat{\Phi}_\ell $ is contained in $[- \sigma_\ell, \sigma_\ell ]$ for $ \ell \leq \nu$ and 
$  \sum_{\ell=1}^\nu  \sigma_\ell <4 $. Then 
 \begin{equation}\label{eqn U pm K split}
 \mathscr{U}_\pm (K  ) = \sum_{j=0}^3 \mathscr{U}_{\pm,j} (K  )  .
 \end{equation}
 \end{lemma}

\begin{proof}
By \eqref{eqn U pm K sum over all j} it is enough to show that $\mathscr{U}_{\pm,j} (K  ) = 0$ for $j \geq 4 $. 
Let $ \epsilon_1 := 4 -  \sum_{\ell=1}^\nu  \sigma_\ell >0 $. 
We  choose $ \delta_1 , \ldots, \delta_\nu $ satisfying 
$$ 0 < \delta_1 <  \cdots < \delta_\nu \leq   \frac{ 8 - \epsilon_1 }{ 8 - 2 \epsilon_1 } \delta_1 . $$  For notational convenience, let $\lambda_N = \frac{\log \log N}{N},$ and note that $\tRe z_k \ge \delta_1 \lambda_N$, so that $e^{-N z_k} \ll (\log N)^{-\delta_1}$.  
Putting this into \eqref{def J pm j star}, we get
\begin{align*}
      |J_{\pm, j }^* (z_K  , N)|  
  &\ll   (\log N)^{ -2 j \delta_1 } \sum_{ \substack{  K'\sqcup K''  \sqcup K_1 \sqcup \cdots \sqcup K_R  = K   \\  |K'|=j, ~   |K_r | = 2  } }|H_0^{\mp}(z_{K'} )|  \prod_{ \ell \in K''}   |H_1^\mp ( z_{K'}  ,   z_\ell   )|   \prod_{r=1}^R |H_2  ( z_{ K_r }  )|\\
  &\ll  (\log N)^{ -2 j \delta_1 } \bfrac{N}{\log \log N}^{|K| - j} \sum_{ \substack{  K'\subset  K   \\  |K'|=j } }|H_0^{\mp}(z_{K'} )|,
\end{align*}
where we have used the crude bounds $|H_1^\mp ( z_{K'}  ,   z_\ell   )|\ll \frac{N}{\log \log N}$, and $|H_2  ( z_{ K_r }  )|\ll \bfrac{N}{\log \log N}^2$.
By Lemma \ref{lem:Phibdd}, we see that
$$ \int_{ ( \delta( \ell, N))} \left| \Phi_\ell \left(  \frac{ i N z_\ell }{   \pi } \right) \right| | dz_\ell |    \ll   \int_{-\infty}^\infty   \frac{   (\log N )^{ 2 \sigma_\ell \delta_\ell } }{  (N| \delta( \ell, N) + it  |)^2   } dt  \leq  \frac{  (\log N)^{ \sigma_\ell  \frac{ 8 - \epsilon_1 }{ 4 -   \epsilon_1 } \delta_1 }  }{  N \log \log N }  $$
for $ \ell \in K \setminus K'$.
Hence, by applying the above inequalities to \eqref{def U pm j K}, we find that
\begin{align}
     \mathscr{U}_{\pm, j}  (K  ) 
     &\ll \lim_{N \to \infty} \sum_{ \substack{  K'\subset  K   \\  |K'|=j } } (\log N)^{ -2 j \delta_1 }   \prod_{\ell \in K\backslash K'} \frac{  (\log N)^{ \sigma_\ell  \frac{ 8 - \epsilon_1 }{ 4 -   \epsilon_1 } \delta_1 }  }{   ( \log \log N )^2 } \notag \\
     &\times   \int_{ (\delta(\ell, N) ; K')   } |H_0^{\mp}(z_{K'} )|  \prod_{ \ell \in K'  }    \left|\Phi_\ell \left( \frac{i N z_\ell  }{  \pi}  \right)    dz_\ell\right|  . \label{eqn:UpmjK 1}
\end{align}

We now consider two cases of the $z_\ell $ for $\ell \in K' $ in \eqref{eqn:UpmjK 1}.  When $|z_{\ell_0}|  \ge \frac{1}{10}$ for some $\ell_0 \in K'$, we use the crude bound
$$|H_0^{\mp}(z_{K'} )|\ll \bfrac{N}{\log \log N}^{j+2\binom{j}{2}},
$$
and the bound from Lemma \ref{lem:Phibdd},
$$ \Phi_{\ell} \left( \frac{i N z_\ell  }{  \pi}  \right)  \ll \frac{N^{\epsilon_1 }}{(|z_\ell|N)^{A}},
$$
for any $A>0$, noting that the above also implies that the integral over $ z_{\ell_0} $ satisfying $ | z_{\ell_0}| \geq \frac{1}{10}$ is $  \ll \frac{1}{N^A}$ for any $A>0$.  Thus, the final contribution of this case of the $z_\ell$ to $ \mathscr{U}_{\pm, j}  (K  ) $ in \eqref{eqn:UpmjK 1} is $\ll \lim_{N\rightarrow \infty} N^{-A} = 0$.

For $|z_\ell| \le \frac{1}{10}$ for all $\ell \in K'$, we write $z_\ell = \delta_\ell \lambda_N + it_{\ell}$, and get  
\begin{align*}
    |H_0^{\mp}(z_{K'} )|  & \ll   \prod_{   \{\ell_1, \ell_2 \}  \subset K' , \ell_1 \neq \ell_2 }   \frac{ ( 1- e^{- z_{\ell_1} + z_{\ell_2}}) ( 1- e^{z_{\ell_1} - z_{\ell_2}})  } {(1- e^{-z_{\ell_1}-z_{\ell_2}})(1- e^{z_{\ell_1}+z_{\ell_2}}) }
    \prod_{ \ell \in K'} \frac{  1 }{   1 - e^{ \mp 2 z_\ell }}  \\
    &\asymp  \prod_{   \{\ell_1, \ell_2 \}  \subset K' , \ell_1 \neq \ell_2}   \frac{(\lambda_N+ |t_{\ell_1} - t_{\ell_2}|)^2}{(\lambda_N + |t_{\ell_1}+t_{\ell_2}|)^2 }
    \prod_{ \ell \in K'} \frac{  1 }{\lambda_N+ |t_\ell| }  \\
    &= \prod_{   \{\ell_1, \ell_2 \}  \subset K' , \ell_1 \neq \ell_2 }   \frac{(1+ \left|\frac{t_{\ell_1}}{\lambda_N} - \frac{t_{\ell_2}}{\lambda_N}\right|)^2}{(1 + \left|\frac{t_{\ell_1}}{\lambda_N} +\frac{t_{\ell_2}}{\lambda_N}\right|)^2 }
    \prod_{ \ell \in K'} \frac{  1 }{\lambda_N+ |t_\ell| }  \\
    &\le \prod_{   \{\ell_1, \ell_2 \}  \subset K' , \ell_1 \neq \ell_2 }   \left(1+ \left|\frac{t_{\ell_1}}{\lambda_N}\right|\right)^2 \left(1+\left|\frac{t_{\ell_2}}{\lambda_N}\right|\right)^2 
    \prod_{ \ell \in K'} \frac{  1 }{\lambda_N+ |t_\ell| }  \\
    &= \prod_{ \ell \in K'} \frac{ \left(1+ \left|\frac{t_{\ell}}{\lambda_N}\right|\right)^{2(j-1)} }{\lambda_N+ |t_\ell| } 
\end{align*}
Again by Lemma \ref{lem:Phibdd},
\begin{align*}
\int_{ ( \delta( \ell, N))}&  \left| \Phi_\ell \left(  \frac{ i N z_\ell }{   \pi } \right) \right| \frac{ \left(1+ \left|\frac{t_{\ell}}{\lambda_N}\right|\right)^{2(j-1)} }{\lambda_N+ |t_\ell| }  | dz_\ell |    \\
&\ll   \int_{-\infty}^\infty   \frac{   (\log N )^{ 2 \sigma_\ell \delta_\ell } }{  (N( \lambda_N  + |t_\ell|) )^{2+2(j-1)}   } \frac{ \left(1+ \left|\frac{t_{\ell}}{\lambda_N}\right|\right)^{2(j-1)} }{\lambda_N+ |t_\ell| }  dt_\ell 
\ll  \frac{  (\log N)^{ \sigma_\ell  \frac{ 8 - \epsilon_1 }{ 4 -   \epsilon_1  } \delta_1 }  }{ (\log \log N)^{2j}}.
\end{align*}
Hence, by applying the above inequalities to \eqref{eqn:UpmjK 1} and ignoring negative powers of $\log \log N$, we have
\begin{align*}
\mathscr{U}_{\pm, j}  (K) 
&\ll \lim_{N\rightarrow \infty}   (\log N)^{-2j\delta_1  +  \sum_{\ell=1}^\nu      \sigma_\ell \frac{ 8 - \epsilon_1 }{ 4 -   \epsilon_1 } \delta_1 }  \leq \lim_{N\rightarrow \infty}   (\log N)^{ -    \epsilon_1   \delta_1  }  =0
\end{align*}
  for $ j \geq 4 $.
  This proves the lemma.
 \end{proof}
 
Next we compute $   \mathscr{U}_{\pm,j} (K  ) $. 
 By shifting each contour $( \delta(\ell, N))$ to $ (  \pi \delta_\ell  / N)$ and then substituting $z_\ell = \pi w_\ell/N$, we find that   
 \begin{equation}\label{eqn U pm j K int 1}
  \mathscr{U}_{\pm, j}  (K  ) =   \frac{1}{ (\pi i)^{ |K  |}}   \int_{ (\delta_\ell  ; K)   }  J_{\pm, j}^{**} (w_K)   \prod_{ \ell \in K  }    \Phi_\ell (i w_\ell )    dw_\ell , 
    \end{equation}
    where
    $$J_{\pm, j}^{**} (w_K):=    \lim_{N \to \infty}  \frac{\pi^{|K|} }{ N^{ |K  |}}   J_{\pm,j}^* \left( \frac{ \pi}{N} w_{K } , N  \right)  .$$ 
  Then by changing the order of the limit and the integrals we find that
  \begin{equation*} 
 J_{\pm, j}^{**} (w_K)  =  \sum_{ \substack{  K' \sqcup K''  \sqcup K_0   = K   \\  |K' |=j } }  \sum_{ \underline{G} \in \Pi_{K_0 , 2 }}  e^{-2\pi  \sum_{ k \in K' } w_k } \cH_0^{\mp}(w_{K'} )  \prod_{ \ell \in K''}  \cH_1^\mp ( w_{K'}  ,   w_\ell   )   \prod_{G_i \in \underline G } \cH_2  ( w_{ G_i }  ),    
  \end{equation*}
  where
  \begin{align*}
  \cH_0^{\mp}(w_{K'} )   & :=   \lim_{N \to \infty}  \frac{\pi^{|K'|}}{ N^{ |K'  |}}   H_0^{\mp}\left(  \frac{\pi}{N} w_{K'} \right), \\
    \cH_1^\mp ( w_{K'}  ,   w_\ell   )   & :=  \lim_{N \to \infty}  \frac{\pi }{ N  }       H_1^\mp \left( \frac{\pi}{N} w_{K'}  , \frac{\pi}{N}  w_\ell   \right),  \\
   \cH_2  ( w_{ G_i }  )   & :=  \lim_{N \to \infty}  \frac{\pi^2 }{ N^{ 2}}    H_2  \left( \frac{\pi}{N} w_{ G_i}  \right)   .  
  \end{align*}
  By \eqref{def HWs} we have 
   \begin{align*}
  \cH_0^{\mp}(w_{K'} )  & = (\mp 1)^{|K'|}   \cH_0 (w_{K'} )  :=   (\mp 1)^{|K'|}   \prod_{ \substack{ k_1, k_2  \in K'  \\ k_1 > k_2  }}   \frac{  (w_{k_1}  - w_{k_2} )^2  }{ (w_{k_1} + w_{k_2} )^2 } \prod_{ k \in K'} \frac{   1 }{ 2w_k }    , \\ 
  \cH_1^\mp ( w_{K'}  ,   w_\ell   ) & =   \sum_{ k \in K'}  \left(  \frac{1}{ w_\ell + w_k   }   -  \frac1{  w_\ell -w_k  } \right) \pm  \frac1{2w_\ell }    , \\ 
\cH_2  ( \{ w_{k_1} , w_{k_2} \}  ) &  =  \frac{1 }{ (w_{k_1}  + w_{k_2} )^2} .
\end{align*}

 We first integrate the $w_{G_i }$ integrals in \eqref{eqn U pm j K int 1}. By Lemma \ref{lemma integrals 1} and \eqref{def:I2Kr} we have
   $$
      \frac{1}{  (  \pi i )^2  }  \int_{ (\delta_k) }  \int_{ (\delta_m ) }   \frac{    1}{ ( w_m + w_k)^2 }          \Phi_m (iw_m) \Phi_k   (iw_k )    dw_m dw_k  = \mathscr{I}_{2 }  ( \{ m,k \} ) .  $$
Then we have
 \begin{multline}\label{eqn U pm j K int 2 }
  \mathscr{U}_{\pm, j}  (K  ) = \sum_{ \substack{  K' \sqcup K''  \sqcup K_0   = K   \\  |K' |=j } }  C_0 ( K_0)   \frac{ (\mp 1)^{|K'|}  }{ (\pi i)^{ |K'\sqcup K'' |}}\\
    \times   \int_{ (\delta_\ell  ; K'\sqcup K'')   }  e^{-2\pi  \sum_{ k \in K' } w_k } \cH_0(w_{K'} )  \prod_{ \ell \in K''}  \cH_1^\mp ( w_{K'}  ,   w_\ell   )    \prod_{ \ell \in K' \sqcup K''  }    \Phi_\ell (i w_\ell )    dw_\ell , 
  \end{multline}
  where
  \begin{equation}\label{def: C0K0}
   C_0 ( K_0 ) :=  \sum_{ \underline{G} \in \Pi_{K_0 , 2 }}  \prod_{G_i \in \underline{G}}  \mathscr{I}_{2 }  ( G_i  )  . 
   \end{equation}
 
  Next, for $ \ell \in K''$ we have
 \begin{align*} 
  \frac{1}{ \pi i}   \int_{ (\delta_\ell )   }  &   \cH_1^\mp ( w_{K'}  ,   w_\ell   )      \Phi_\ell (i w_\ell )    dw_\ell   
  =    \sum_{ k \in K'} 4 \int_0^\infty \widehat{\Phi}_\ell (t) e^{-2 \pi t w_k } dt  -  \sum_{ k \in K',   k < \ell }  2 \Phi_\ell ( i w_k)   \pm \frac{ \Phi_\ell (0)}2   
  \end{align*}
  by Lemma \ref{lemma integrals 1}.   Thus, we have  
   \begin{equation}\label{eqn U pm j K int 3 }
 \begin{split}
 & \mathscr{U}_{\pm, j}    (K  ) \\
   & =   \sum_{ \substack{  K' \sqcup K''  \sqcup K_0   = K   \\  |K' |=j } }  C_0 ( K_0)    \frac{(\mp 1)^{|K'|} }{ (\pi i)^{ |K' |}}   \int_{ (\delta_\ell  ; K' )   }  e^{-2\pi  \sum_{ k \in K' } w_k } \cH_0(w_{K'} )     \\
  & \times \prod_{\ell \in K''} \left(  \sum_{ k \in K'} 4 \int_0^\infty \widehat{\Phi}_\ell (t) e^{-2 \pi t w_k } dt  -  \sum_{ k \in K',   k < \ell }  2 \Phi_\ell ( i w_k)   \pm \frac{ \Phi_\ell (0)}2 \right)
   \prod_{ \ell \in K'    }    \Phi_\ell (i w_\ell )    dw_\ell  \\
   & =  \sum_{ \substack{  K' \sqcup K'' \sqcup K''' \sqcup K_0   = K   \\  |K' |=j } }  C_0 ( K_0)
\left( \prod_{\ell \in K'''}   \pm \frac{ \Phi_\ell (0)}2  \right) (\mp 1)^{|K'|}  \mathscr{V} ( K', K'')  
  \end{split}     \end{equation}
for $ j \geq 0$,  where $  \mathscr{V} ( \emptyset, \emptyset)=1$, $  \mathscr{V} ( \emptyset,  K'' ) = 0 $ for $ K'' \neq \emptyset$ and 
  \begin{multline} \label{def V K}
   \mathscr{V} ( K', K'') := 
\frac{1 }{ (2\pi i)^{ |K' |}}   \int_{ (\delta_\ell  ; K' )   }  e^{-2\pi  \sum_{ k \in K' } w_k } \prod_{ \substack{ k_1, k_2  \in K' \\  k_1> k_2  }}   \frac{  (w_{k_1}   - w_{k_2} )^2  }{ (w_{k_1} + w_{k_2} )^2 }       \\
\times \prod_{\ell \in K''} \left(  \sum_{ k \in K'} 4 \int_0^\infty \widehat{\Phi}_\ell (t) e^{-2 \pi t w_k } dt  -  \sum_{ k \in K',   k < \ell }  2 \Phi_\ell ( i w_k)   \right)
   \prod_{ \ell \in K'    }  \frac{  \Phi_\ell (i w_\ell ) }{w_\ell}    dw_\ell 
\end{multline}
for $  K' \neq \emptyset  $.

By \eqref{eqn T pm U pm K}, \eqref{eqn U pm K split} and \eqref{eqn U pm j K int 3 }, we find that 
\begin{equation} \label{eqn T pm 1} 
\mathscr{T}_{\pm   }    
   =         \sum_{ \substack{  K'\sqcup K'' \sqcup \tilde{K}  \sqcup K_0 = [\nu ]    \\  |K'| \leq 3    } }  C_0 (K_0)        \prod_{ \ell \in \tilde{K} }    \left(\widehat{\Phi}_\ell (0)   \pm \frac{ \Phi_\ell (0)}2 \right)        (\mp 1)^{|K'|} \mathscr{V} ( K', K'')   .
 \end{equation}
  By \eqref{formula C even odd by T} and \eqref{eqn T pm 1}, we find that
   \begin{equation}\label{eqn:C even odd nice cancellation}\begin{split}
     C_{even}(n)  
  &=  \sum_{ \substack{  K'\sqcup K''   \sqcup K_0 = [n]    \\  |K'| \leq 3   } }      C_0 (K_0) (- 1)^{|K'|} \mathscr{V} ( K', K'')   \\
    C_{odd}(n)  
  & =   \sum_{ \substack{  K'\sqcup K''   \sqcup K_0 = [n]    \\  |K'| \leq 3   } }      C_0 (K_0)  \mathscr{V} ( K', K'')  . 
\end{split}   \end{equation}
By \eqref{def C} and the above, the $n$-th centered moment for $O(N)$ is 
\begin{equation}\label{eqn Cn Tplusn} 
 C (n) =   \sum_{ \substack{  K'\sqcup K''   \sqcup K_0 = [n]    \\  |K'| = 0, 2    } }      C_0 (K_0)  \mathscr{V} ( K', K'')    .  
 \end{equation}
 By letting $C_j (n)$ the contribution of the $K'$ with $|K'|=j $, one can easily deduce the first part of Theorem \ref{thm:Cn}. 

To complete the proof of Theorem \ref{thm:Cn}, it remains to compute $\mathscr{V} ( \{ k_1 , k_2 \} , G ) $ for $\{ k_1 , k_2 \} \sqcup G \subset [n]$. By \eqref{def V K}, we see that
  \begin{equation*} \begin{split}
   \mathscr{V} ( \{ k_1 , k_2 \} , G ) = &
\frac{1  }{ (2 \pi i)^2 }   \int_{ (\delta_{k_1}  )   }\int_{ (\delta_{k_2}  )   }  e^{-2\pi  ( w_{k_1} + w_{k_2} )}          \\
& \times \prod_{\ell \in G} \left(  \sum_{ j=1,2 } \bigg(  4 \int_0^\infty \widehat{\Phi}_\ell (t) e^{-2 \pi t w_{k_j} } dt  -   2 \Phi_\ell ( i w_{k_j} ) {\bf 1}_{k_j  < \ell } \bigg)   \right)\\
 & \times \frac{  (w_{k_1}   - w_{k_2}  )^2  }{ (w_{k_1}  + w_{k_2}  )^2 }
     \frac{   \Phi_{k_1} (i w_{k_1} )  \Phi_{k_2} (i w_{k_2}  )  }{  w_{k_1}w_{k_2} }     dw_{k_2} dw_{k_1} .
\end{split}\end{equation*} 
By expanding the product over $\ell \in G$ we find that
\begin{align*} 
   \mathscr{V}& ( \{ k_1 , k_2 \} , G )  \\
   = &\sum_{ G_1 \sqcup G_2 \sqcup G_3 \sqcup G_4 = G} 4^{|G_1 |+|G_2| } (-2)^{ |G_3|+|G_4| }
\frac{1  }{ (2\pi i)^2 }   \int_{ (\delta_{k_1}  )   }\int_{ (\delta_{k_2}  )   }  e^{-2\pi  ( w_{k_1} + w_{k_2} )}          \\
&\times \prod_{\ell \in G_1} \left(    \int_0^\infty \widehat{\Phi}_\ell (t) e^{-2 \pi t w_{k_1} } dt    \right)\prod_{\ell \in G_2} \left(     \int_0^\infty \widehat{\Phi}_\ell (t) e^{-2 \pi t w_{k_2} } dt    \right)\\
&\times \prod_{\ell \in G_3 } \left(     \Phi_\ell ( i w_{k_1} ) {\bf 1}_{k_1  < \ell }     \right)\prod_{\ell \in G_4} \left(    \Phi_\ell ( i w_{k_2} ) {\bf 1}_{k_2  < \ell }    \right)  \frac{  (w_{k_1}   - w_{k_2}  )^2  }{ (w_{k_1}  + w_{k_2}  )^2 }
     \frac{   \Phi_{k_1} (i w_{k_1} )  \Phi_{k_2} (i w_{k_2}  )  }{ w_{k_1}w_{k_2} }     dw_{k_2} dw_{k_1}  .
     \end{align*}
     If every element of $G_3 $ is bigger than $k_1$, then 
     $$ \prod_{\ell \in G_3 } \left(     \Phi_\ell ( i w_{k_1} ) {\bf 1}_{k_1  < \ell }     \right)   \Phi_{k_1} (i w_{k_1} )  = \Phi_{k_1, G_3 }( i w_{k_1})   ,$$
     and equals $  0$  otherwise. Thus, we have
     \begin{align*}
 \mathscr{V} ( \{ k_1 , k_2 \} , G ) = &\sum_{ \substack{ G_1 \sqcup G_2 \sqcup G_3 \sqcup G_4 = G  \\ G_3 \subset \{ k_1 + 1, \ldots, n \} \\ G_4 \subset \{ k_2 + 1, \ldots, n \}   }}  4^{|G_1 |+|G_2| } (-2)^{ |G_3|+|G_4| }
\frac{1  }{ (2\pi i)^2 }   \int_{ (\delta_{k_1}  )   }\int_{ (\delta_{k_2}  )   }  e^{-2\pi  ( w_{k_1} + w_{k_2} )}          \\
&\times \prod_{\ell \in G_1} \left(    \int_0^\infty \widehat{\Phi}_\ell (u_\ell ) e^{-2 \pi u_\ell  w_{k_1} } du_\ell    \right)\prod_{\ell \in G_2} \left(     \int_0^\infty \widehat{\Phi}_\ell (u_\ell ) e^{-2 \pi u_\ell w_{k_2} } du_\ell    \right)\\
&\times  \frac{  (w_{k_1}   - w_{k_2}  )^2  }{ (w_{k_1}  + w_{k_2}  )^2 }
     \frac{   \Phi_{k_1, G_3 } (i w_{k_1} )  \Phi_{k_2, G_4} (i w_{k_2}  )  }{  w_{k_1}w_{k_2} }     dw_{k_2} dw_{k_1} .
  \end{align*}
Now change  the order of integration so that the $ w_{k_1} , w_{k_2}$ are the innermost integrals. By Lemma \ref{lemma integral I 12} applied to the $ w_{k_1} , w_{k_2}$ integrals, we obtain \eqref{def:V}. This completes the proof of Theorem \ref{thm:Cn}.

\subsection{Technical lemmas required in this section}\label{sec:technical lemmas}

      \begin{lemma}\label{lemma O- to USp}
 Let $\Phi_i $ be an even Schwartz function for each $i \leq n $.  We have 
\begin{multline*}  
 \lim_{N \to \infty}     \int_{O^{-} (2N+2)}  \summany_{\substack{ -N  \leq   j_1 , \dots, j_n \leq N  \\  j_1 , \dots, j_n \neq 0      }} \prod_{\ell =1}^n \Phi_\ell  \bigg( \frac{N \theta_{j_\ell}}{ \pi}   \bigg)  dX_{O^{-} (2N+2)}    \\
 = \lim_{N \to \infty}     \int_{USp (2N)}  \summany_{ -N  \leq   j_1 , \dots, j_n \leq N      }   \prod_{\ell =1}^n \Phi_\ell  \bigg( \frac{N \theta_{j_\ell}}{ \pi}   \bigg) dX_{USp(2N)}  .
 \end{multline*}
\end{lemma}

\begin{proof}
By \cite[Theorem AD.2.2]{KaSa} we find that
\begin{multline*}
 \lim_{N \to \infty}   \  \int_{O^{-} (2N+2)} \  \sumsharp_{1  \leq   j_1 , \dots, j_n \leq N   }  \prod_{\ell =1}^n \Phi_\ell  \bigg( \frac{N \theta_{j_\ell}}{ \pi}   \bigg)  dX_{O^{-} (2N+2)}    \\
 = \lim_{N \to \infty}   \  \int_{USp (2N)}  \
  \sumsharp_{ 1  \leq   j_1 , \dots, j_n \leq N   }  ~ \prod_{\ell =1}^n \Phi_\ell  \bigg( \frac{N \theta_{j_\ell}}{ \pi}   \bigg)  dX_{USp(2N)}   .
 \end{multline*}
 By \eqref{eqn:combinatorial sieve} with
$$   C_{\underline G } = \sum_{ 1 \leq j_1 , \ldots j_\nu \leq N } \prod_{\ell=1}^\nu \Phi_{G_\ell }  \bigg( \frac{N \theta_{j_\ell}}{ \pi}   \bigg) , \qquad 
  R_{\underline G }  =  \sumsharp_{ 1 \leq j_1 , \ldots j_\nu \leq N } \prod_{\ell=1}^\nu \Phi_{G_\ell }  \bigg( \frac{N \theta_{j_\ell}}{ \pi}   \bigg)$$
   for $ \underline G = \{ G_1 , \ldots , G_\nu \} \in \Pi_n $, we have $\displaystyle C_{\underline O } = \sum_{ \underline G \in \Pi_n } R_{\underline G} $, in other words, 
\begin{multline*}
 \lim_{N \to \infty}   \  \int_{O^{-} (2N+2)} \  \sum_{1  \leq   j_1 , \dots, j_n \leq N   }  \prod_{\ell =1}^n \Phi_\ell  \bigg( \frac{N \theta_{j_\ell}}{ \pi}   \bigg)  dX_{O^{-} (2N+2)}    \\
 = \lim_{N \to \infty}   \  \int_{USp (2N)}  \sum_{ 1  \leq   j_1 , \dots, j_n \leq N   }  ~ \prod_{\ell =1}^n \Phi_\ell  \bigg( \frac{N \theta_{j_\ell}}{ \pi}   \bigg)  dX_{USp(2N)}   .
 \end{multline*}
 By symmetry, the lemma holds.
\end{proof}

\begin{lemma}\label{lemma integrals 1}
Let $ \delta, U_1 , U_2 \in \mathbb{R}$ and $ \delta_1 , \delta_2  > 0 $. Assume that $ \Phi_1 $ and $ \Phi_2 $ are even and their Fourier transforms are compactly supported.  Then we have
$$ \frac{1}{   2 \pi i  }  \int_{ (\delta) }       \frac{e^{- 2 \pi U_1 w } }{ w - z  }    \Phi_1 (iw)    dw = 
\begin{cases}
\int_0^{\infty} \widehat{\Phi}_1( t  + U_1 )  e^{ 2 \pi t z}  dt    & \mathrm{~if~} \delta > \tRe(z), \\  
- \int_0^{\infty} \widehat{\Phi}_1( t  - U_1 )   e^{- 2 \pi t z }   dt   & \mathrm{~if~} \delta < \tRe(z),
\end{cases}$$
and
$$ \frac{1}{   (2 \pi i)^2   }    \int_{ (\delta_2) }       \int_{ (\delta_1) }   \frac{ e^{- 2 \pi ( U_1 w_1 + U_2 w_2 )} }{ ( w_1 + w_2)^2 }          \Phi_1 (iw_1) \Phi_2   (iw_2 )    dw_1 dw_2 = \int_0^{\infty} t \widehat{\Phi}_1( t + U_1 )   \widehat{\Phi}_2 ( t + U_2 )  dt
.$$
In particular, by letting $U_1  = z= 0 $ 
$$ \frac{1}{   2 \pi i  }  \int_{ (\delta_1) }       \frac{1 }{ w_1  }    \Phi_1 (iw_1)    dw_1 = \int_0^{\infty} \widehat{\Phi}_1( t   )    dt  = \frac{  \Phi_1 (0)}{2} .$$

\end{lemma} 
 
 \begin{proof}
 By Fourier inversion, we have that
 \begin{equation}\label{eqn Fourier inversion}
 e^{-2 \pi Uw}  \Phi ( i w ) = e^{-2 \pi Uw}  \int_{-\infty}^{\infty} \widehat{\Phi}( t ) e^{ - 2 \pi t w } dt  = \int_{-\infty}^{\infty} \widehat{\Phi}( t -U ) e^{ - 2 \pi t w    } dt  
 \end{equation}
 for any real $U$.  Thus, if $ \delta > \tRe(z)$,
 \begin{multline*}
 \frac{1}{   2 \pi i  }  \int_{ (\delta) }       \frac{e^{- 2 \pi U_1 w } }{ w -z }    \Phi_1 (iw)    dw =    
 \int_{-\infty}^{\infty} \widehat{\Phi}_1( t_1 -U_1 ) \frac{1}{   2 \pi i  }  \int_{ (\delta) }       \frac{ e^{ - 2 \pi t_1 w    }  }{ w -z  }         dw   dt_1 \\
 =    
 \int_{-\infty}^{0 } \widehat{\Phi}_1( t_1 -U_1 )  e^{ - 2 \pi t_1 z    }  dt_1   
  =    
 \int_0^{\infty} \widehat{\Phi}_1( t + U_1 )   e^{  2 \pi t z    }  dt. 
\end{multline*}
 For $\delta < \tRe(z)$, the formula follows by the same arguments.

For the second expression,  by \eqref{eqn Fourier inversion}, we obtain that
\begin{align*}
\frac{1}{   (2 \pi i)^2   }   &  \int_{ (\delta_2) }       \int_{ (\delta_1) }   \frac{ e^{- 2 \pi ( U_1 w_1 + U_2 w_2 )} }{ ( w_1 + w_2)^2 }          \Phi_1 (iw_1) \Phi_2   (iw_2 )    dw_1 dw_2   \\
=&   \int_{-\infty}^{\infty} \widehat{\Phi}_1( t_1 -U_1 )   \frac{1}{   (2 \pi i)^2   }    \int_{ (\delta_2) }        \int_{ (\delta_1) }     \frac{e^{ - 2 \pi t_1 w_1  } }{ ( w_1 + w_2)^2 } dw_1e^{- 2 \pi U_2 w_2 } \Phi_2 ( iw_2)  dw_2 dt_1 . 
\end{align*}
For $ t_1 >0 $, we shift the $w_1$-contour far to the right and the $w_1$-integral is zero. For $ t_1 \leq 0 $, we shift the $w_1 $-contour far to the left and pick up a residue at $ w_1 = - w_2 $. Hence, the above equals
\begin{align*}
& \int_{-\infty}^{0} \widehat{\Phi}_1( t_1 -U_1 )   \frac{1}{    2 \pi i    }    \int_{ (\delta_2) }          (  - 2 \pi t_1   )  e^{  2 \pi t_1   w_2 }      e^{- 2 \pi U_2 w_2 } \Phi_2 ( iw_2)  dw_2   dt_1   \\
= &  \int_{-\infty}^{0} \widehat{\Phi}_1( t_1 -U_1 )  (-t_1 ) \widehat{\Phi}_2 ( t_1 - U_2 )  dt_1  
=   \int_0^{\infty} t \widehat{\Phi}_1( t + U_1 )   \widehat{\Phi}_2 ( t + U_2 )  dt . 
\end{align*}

 \end{proof}

\begin{lemma}\label{lemma integral I 12}
Let  $ \delta_1 , \delta_2 > 0 $ and $ U_1 , U_2 \in \mathbb{R}$. Then we have
\begin{align*}
 \mathscr{I}_{1,2} := & \frac{1}{ (  2 \pi i )^{ 2}}  \int_{ (\delta_1  )} \int_{ (\delta_2 )}           e^{-2\pi U_1  w_1  - 2 \pi U_2 w_2 }           \frac{ ( w_1 - w_2)^2 }{   w_1 w_2  ( w_1 + w_2)^2      }     \Phi_{1 }  ( i w_1 )  \Phi_{2} ( i w_2 )    dw_{2}     dw_{1}  \\
 = & \int_0^{   \infty}   \widehat{\Phi}_1 ( t + U_1 ) dt  \int_{ 0}^{ \infty}   \widehat{\Phi}_2 ( t + U_2 ) dt   - 4 \int_0^{ \infty}   t \widehat{ \Phi}_{1 }  (t+U_1 ) \widehat{ \Phi}_{2 }  (t+U_2 )        dt.
\end{align*}

\end{lemma}

\begin{proof}  
 Since 
$$  \frac{ ( w_1 - w_2)^2 }{   w_1 w_2  ( w_1 + w_2)^2      }   =    \frac{ ( w_1 + w_2)^2  - 4 w_1 w_2  }{   w_1 w_2  ( w_1 + w_2)^2      }  = \frac{1}{ w_1 w_2} - \frac{4}{ ( w_1 + w_2 )^2 }  , $$
we have
\begin{align*}
\mathscr{I}_{1,2}  = &  \frac{1}{ (  2 \pi i )^{ 2}}  \int_{ (\delta_1  )} \int_{ (\delta_2 )}           e^{-2\pi U_1  w_1  - 2 \pi U_2 w_2 }           \frac{ 1 }{   w_1 w_2        }     \Phi_{1 }  ( i w_1 )  \Phi_{2} ( i w_2 )    dw_{2}     dw_{1}  \\
& -  \frac{1}{ (  2 \pi i )^{ 2}}  \int_{ (\delta_1  )} \int_{ (\delta_2 )}           e^{-2\pi U_1  w_1  - 2 \pi U_2 w_2 }           \frac{ 4}{     ( w_1 + w_2)^2      }     \Phi_{1 }  ( i w_1 )  \Phi_{2} ( i w_2 )    dw_{2}     dw_{1}    . 
\end{align*}
 The lemma follows by applying Lemma \ref{lemma integrals 1} to the above.
\end{proof}

\section*{Acknowledgement}

 V.C. acknowledges support from NSF grant DMS-2502599. X.L. acknowledges support from Simons Travel Grant 962494 and NSF grant DMS-2302672. Y.L. was supported by the National Research Foundation of Korea (NRF) grant funded by the Korea government (MSIT) (No. RS-2024-00415601, RS-2025-16070236).

 We would like to thank Henryk Iwaniec and Nina Snaith for comments.


\end{document}